\documentclass[10pt]{article}
\usepackage[
	a4paper, scale=0.80,
]{geometry}

\usepackage{ifpdf}
\usepackage{amsmath}
\usepackage{amsfonts}
\usepackage{amsthm}
\usepackage{amssymb}
\usepackage{latexsym}
\usepackage{fixmath}
\usepackage{mathdots}
\usepackage{kbordermatrix}
\usepackage{xcolor}
\usepackage{graphicx}
\usepackage{subfig}

\ifpdf
\usepackage{hyperref}
\else
\usepackage[hypertex]{hyperref}
\fi

\usepackage{siunitx}
\usepackage{cleveref}
\crefname{assumption}{assumption}{assumptions}
\crefname{figure}{figure}{figures}
\crefname{equation}{}{}
\crefname{subsection}{section}{sections}

\usepackage[shortlabels]{enumitem}
\setlist[enumerate,2]{label=(\alph*),ref=\theenumi(\alph*)}
\setlist[enumerate,3]{label=\roman*.,ref=\theenumii\roman*}

\usepackage{algorithm,algorithmic}

\let\set\relax
\let\abs\relax
\usepackage{mathtools}
\providecommand\given{\:\vert\:}
\DeclarePairedDelimiterXPP\set[1]{}{\{}{\}}{}{
	\renewcommand\given{\nonscript\:\delimsize\vert\nonscript\:\mathopen{}}
	#1}
\DeclarePairedDelimiter\abs{|}{|}
\DeclarePairedDelimiter\N{\|}{\|}

\DeclarePairedDelimiter\ceil{\lceil}{\rceil}
\DeclarePairedDelimiter\floor{\lfloor}{\rfloor}

\usepackage{mathrsfs}
\def\scrT{\mathscr{T}}

\def\bX{\boldsymbol{X}}
\def\bY{\boldsymbol{Y}}

\def\bbA{\mathbb{A}}

\def\bbH{\mathbb{H}}
\def\bbM{\mathbb{M}}
\def\bbN{\mathbb{N}}
\def\bbQ{\mathbb{Q}}
\def\bbR{\mathbb{R}}

\def\bbT{\mathbb{T}}

\def\cR{{\cal R}}
\def\cU{{\cal U}}

\def\cX{{\cal X}}
\def\cY{{\cal Y}}

\usepackage{accents}
\newcommand{\ubar}[1]{\underaccent{\bar}{#1}}
\newcommand{\ol}[2][1]{{}\mkern#1mu\bar{\mkern-#1mu#2}}
\newcommand{\ul}[2][1]{{}\mkern#1mu\ubar{\mkern-#1mu#2}}
\def\wtd{\widetilde}
\def\what{\widehat}

\def\txs{\textstyle}
\def\scs{\scriptstyle}

\def\ee{\mathrm{e}}

\DeclareMathOperator{\diff}{d\!}

\DeclareMathOperator{\diag}{diag}

\DeclareMathOperator{\trace}{trace}

\DeclareMathOperator{\UI}{ui}
\DeclareMathOperator{\F}{F}

\DeclareMathOperator{\T}{T}

\DeclareMathOperator{\rmc}{c}

\DeclareMathOperator{\rmo}{o}

\newtheorem{theorem}{Theorem}[section]
\newtheorem{lemma}{Lemma}[section]

\theoremstyle{definition}
\newtheorem{definition}{Definition}[section]
\newtheorem{remark}{Remark}[section]

\newtheorem{assumption}{Assumption}[section]

\numberwithin{algorithm}{section}
\numberwithin{equation}{section}
\numberwithin{figure}{section}
\numberwithin{table}{section}
\allowdisplaybreaks

\usepackage[mathlines]{lineno}
\newcommand*\patchAmsMathEnvironmentForLineno[1]{%
	\expandafter\let\csname old#1\expandafter\endcsname\csname #1\endcsname
	\expandafter\let\csname oldend#1\expandafter\endcsname\csname end#1\endcsname
	\renewenvironment{#1}%
	{\linenomath\csname old#1\endcsname}%
	{\csname oldend#1\endcsname\endlinenomath}%
}%
\newcommand*\patchBothAmsMathEnvironmentsForLineno[1]{%
	\patchAmsMathEnvironmentForLineno{#1}%
	\patchAmsMathEnvironmentForLineno{#1*}%
}%
\AtBeginDocument{%
	\patchBothAmsMathEnvironmentsForLineno{equation}%
	\patchBothAmsMathEnvironmentsForLineno{align}%
	\patchBothAmsMathEnvironmentsForLineno{alignat}%
	\patchBothAmsMathEnvironmentsForLineno{flalign}%
	\patchBothAmsMathEnvironmentsForLineno{gather}%
	\patchBothAmsMathEnvironmentsForLineno{multline}%
}

\usepackage[normalem]{ulem}

\usepackage{tikz}
\usetikzlibrary{graphs}

\def\etal{\emph{et al.}}
\def\LWLZ{Li \etal}
\def\overevent{;\:}
\def\setdelim{:}
\def\opdelim{\colon}
\def\new{^+}

\def\fil{\mathbb{F}}
\def\opL{\mathcal{L}}
\def\Grassmann{\mathbb{G}}
\DeclareMathOperator{\sphere}{\mathbb{S}}
\DeclareMathOperator{\opE}{\mathrm{E}}
\DeclareMathOperator{\opprob}{\mathrm{P}}
\newcommand\E[2][*]{\opE\set#1{#2}}
\newcommand\prob[2][*]{\opprob\set#1{#2}}
\newcommand\Nout[2][*]{N_{\mathrm{out}}\set#1{#2}}
\newcommand\Nin[2][*]{N_{\mathrm{in}}\set#1{#2}}
\newcommand\Nqb[2][*]{N_{\mathrm{qb}}\set#1{#2}}
\DeclareMathOperator{\Sum}{sum}

\DeclareMathOperator{\OO}{O}
\DeclareMathOperator{\oo}{o}
\DeclareMathOperator{\var}{var}
\DeclareMathOperator{\cov}{cov}

\newcommand\ind[1]{\mathbf{1}_{#1}}
\newcommand\varc[1]{\var_{\circ}\!\left(#1\right)}
\newcommand\covc[1]{\cov_{\circ}\!\left(#1\right)}
\newcommand\hyperg[1][2]{\def\hypergTemp{#1}\hypergForm}
\newcommand\hypergForm[1][1]{{}_{\hypergTemp}F_{#1}}

\allowdisplaybreaks
\begin{document}

\title{Nearly optimal stochastic approximation for online principal subspace estimation}

\author{
Xin Liang\thanks{
Yau Mathematical Sciences Center, Tsinghua University, Beijing {\rm100084}, China; and
Yanqi Lake Beijing Institute of Mathematical Sciences and Applications, Beijing {\rm101408}, China.
Email: {\tt liangxinslm@tsinghua.edu.cn}.
}
\and
Zhen-Chen Guo\thanks{
Department of Mathematics, Nanjing University, Nanjing {\rm210093}, China.
Email: {\tt guozhenchen@nju.edu.cn}.
}
\and
Li Wang\thanks{
Department of Mathematics, University of Texas at Arlington, Arlington, TX {\rm76019}, USA.
Email: {\tt li.wang@uta.edu}.
}
\and
Ren-Cang Li\thanks{
Department of Mathematics, University of Texas at Arlington, Arlington, TX {\rm76019}, USA; and
Department of Mathematics, Hong Kong Baptist University, Hong Kong.
Email: {\tt rcli@uta.edu}.
}
\and
Wen-Wei Lin\thanks{
Nanjing Center for Applied Mathematics, Nanjing {\rm211135}, China; and
Department of Applied Mathematics, National Yang Ming Chiao Tung University, Hsinchu {\rm300}, Taiwan.
Email: {\tt wwlin@math.nctu.edu.tw}.
}
}

\maketitle

\begin{abstract}
Principal component analysis (PCA) has been widely used in analyzing high-dimensional data. It converts a set of observed data points of possibly correlated variables into a set of linearly uncorrelated variables via an orthogonal transformation.
To handle streaming data and reduce the complexities of PCA, (subspace) online PCA iterations were proposed to iteratively update the orthogonal transformation by taking one observed data point at a time. Existing works on the convergence of (subspace) online PCA iterations mostly focus on the case where sample are almost surely uniformly bounded.
In this paper, we analyze the convergence of a subspace online PCA iteration under more practical assumption and obtain a nearly optimal finite-sample error bound. Our convergence rate almost matches the minimax information lower bound. We prove that the convergence is nearly global
in the sense that the subspace online PCA iteration is convergent with high probability for random initial guesses. This work also leads to a simpler proof of the recent work on analyzing online PCA for the first principal component only.
\end{abstract}




\section{Introduction}
\label{sec:intro}

Principal component analysis (PCA) introduced in \cite{hotelling1933analysis,pearson1901lines} is
one of the most well-known and popular methods for dimensionality reduction in high-dimensional data analysis. With the volume of data continuously increases, the classical PCA suffers from two major bottlenecks: 1) the high-computational complexity, including computing empirical covariance matrix and solving eigen-decomposition problem, and 2) the high storage requirement for the large covariance matrix. These issues prevent PCA from being used for solving problems with large-scale and high-dimensional data.

To reduce both the time and space complexities, Oja~\cite{oja1982simplified} in 1982 proposed an online PCA iteration to approximate the first principal component -- the top eigenvector of the empirical covariance matrix.
Computing the first principal component only is rarely adequate in real-world applications.
Later in 1985, Oja and Karhunen~\cite{ojaK1985stochastic} proposed a subspace online PCA iteration to approximate a principal subspace of
any prescribed dimension.
These methods update  approximations incrementally by processing data one vector at a time as
soon as it comes in such that calculating/storing the empirical covariance matrix explicitly is completely avoided and therefore
result in no memory burden.
In the rest of this paper, by online PCA iteration we mean the one just for computing the first principal component whereas
a subspace online PCA iteration refers to one for computing a principal subspace.

Although the online PCA iteration \cite{oja1982simplified} was proposed over 30 years ago, its convergence analysis is rather scarce until recently.
Some recent works \cite{2015arXiv150103796B,2016arXiv160206929J,2015arXiv150909002S}
studied the convergence of the online PCA for the first principal component from different points of view
and obtained some results for the case where the samples are almost surely uniformly bounded.
For such a case, De Sa \etal~\cite{desaOR2015global}
studied a different but closely related problem, in which the angular part is equivalent to the online PCA,
and obtained some convergence results.
In contrast, for the distributions with sub-Gaussian tails (note that the samples of this kind
of distributions may be unbounded), \LWLZ~\cite{liWLZ2017near}
proved a nearly optimal convergence rate
for the online PCA iteration when the initial guess is randomly chosen according to a uniform distribution
and the stepsize chosen in accordance with the sample size.
This result is more general than previous ones in \cite{2015arXiv150103796B,2016arXiv160206929J,2015arXiv150909002S}, because
it is for distributions that can possibly be unbounded,
and the convergence rate is nearly optimal and nearly global.

For the subspace online PCA~\cite{ojaK1985stochastic},
some recent works studied the convergence for the case where the samples are almost surely uniformly bounded.
In a series of papers~\cite{aroraCLS2012stochastic,aroraCS2013stochastic,marinovMA2018streaming,mianjyA2018stochastic},
Arora \etal\ studied PCA as a stochastic optimization problem and its variations via direct optimization approaches,
namely using convex relaxation and adding regularizations. The subspace iteration falls into one variant of their methods.
Hardt and Price~\cite{hardtP2014noisy} and Balcan \etal~\cite{balcanDWY2016improved} treated the subspace iteration
as a noisy power method and analyzed its convergence.
Li \etal~\cite{liLL2016rivalry} investigated the convergence for the case where the initial guess follows the
normal distribution.
Garber \etal~\cite{garberHJKMNS2016faster} used the shift-and-invert technique to speed up the convergence,
but their analysis was only done for the top eigenvector.
Allen-Zhu and Li~\cite{allenzhuL2017first} proposed  a faster variant of subspace online PCA iteration,
along with their gap-dependent and gap-free convergence results.
However, those works are performed under the assumption that samples are almost surely uniformly bounded.
For distributions, e.g., sub-Gaussians, that are possibly unbounded,
a thorough convergence analysis of the subspace online PCA remains elusive.

In this paper, we aim to fill up the gap by establishing a nearly optimal and nearly global convergence rate
for the subspace online PCA for samples of possibly unbounded distributions of sub-Gaussians. In going through the proving process in \cite{liWLZ2017near} for online PCA iteration, we find that there are three major hurdles,
as we will explain in detail in \cref{sec:comparison},
that prevent their proving technique for one-dimensional case, i.e., the most significant principal component, from
being straightforwardly generalized to analyze the multidimensional case, i.e., significant principal subspaces.
To overcome these challenging difficulties, we adopt a new proving technique
and apply it to a variant of subspace online PCA to fulfill the goal.
The variant is mathematically equivalent to the original one in Oja and Karhunen~\cite{ojaK1985stochastic}
except without explicit
references to QR decompositions for orthogonalization, and is essentially the same as
the orthogonal Oja algorithm of Abed-Meraim \etal~\cite{abedmeraimACH2000orthogonal}.
In addition to the advantages inherited from online PCA,
it leads to a computationally economical formula for the subspace online iteration.
Some of the proving techniques are built by ourselves with the help of the theory of special functions of a matrix argument, which is rarely used in the statistical community.
We mention in passing that our proving technique may be specialized to the online PCA for a simpler proof than that in~\cite{liWLZ2017near} for the most significant principal component.

The rest of this paper is organized as follows. We first briefly introduce the related works in \cref{sec:related-work}.
In \cref{sec:PCA-reformulate}, we proposed a variant of the subspace online PCA iteration \cref{eq:onlinePCA-9v}, which will be the version to be analyzed.
Our main results are stated in \cref{sec:main-results} together with three main theorems and discussions of the newly invented proving technique, where we compare our results for one-dimensional case with the recent results
in \cite{liWLZ2017near} and outline the technical differences in proofs between ours and those from \cite{liWLZ2017near}.
Our proofs are given in \cref{sec:proof-of-cref-thm-thm1-ppr1,sec:proof-of-cref-thm-thm2-ppr1}.
Finally, in \cref{sec:conclusion} we draw our conclusions. Some of the complicated calculations are deferred
to \cref{sec:supplementary-proofs} for clarity.

\emph{Notation.} 
$\bbR^{n\times m}$ is the set
of all $n\times m$  real matrices, $\bbR^n=\bbR^{n\times 1}$,
and $\bbR=\bbR^1$. $I_n$ (or simply $I$ if its dimension is
clear from the context) is the $n\times n$ identity matrix and $e_j$ is its $j$th column (usually with dimension determined by the context).
For a matrix $X$, $\sigma(X)$, $\N{X}_\infty$, $\N{X}_2$ and $\N{X}_{\F}$ are the multiset of the singular values,
the $\ell_{\infty}$-operator norm, the spectral norm, and the Frobenius norm of $X$, respectively.
$\cR(X)$ is the subspace spanned by the columns of $X$, $X_{(i,j)}$ is the $(i,j)$th entry of $X$, and $X_{(k:\ell,:)}$ and
$X_{(:,i:j)}$ are two submatrices of $X$ consisting of its row $k$ to row $\ell$ and column $i$ to column $j$, respectively.
$X\circ Y$ is the Hadamard, i.e., entrywise, product of matrices (vector) $X$ and $Y$ of the same size.

For any vector or matrix $X,\,Y$,
$X\le Y$ ($X<Y$) means $X_{(i,j)}\le Y_{(i,j)}$ ($X_{(i,j)} < Y_{(i,j)}$) for any $i,j$. $X\ge Y$ ($X>Y$) if $-X\le-Y$ ($-X<-Y$);
$X\le \alpha$ ($X<\alpha$) for a scalar $\alpha$ means $X_{(i,j)}\le \alpha$ ($X_{(i,j)}< \alpha$) for any $i,j$;
similarly $X\ge\alpha$ and $X>\alpha$.
For a subset or an event $\bbA$, $\bbA^{\rmc}$ is the complement set of $\bbA$.
By $\sigma\txs\set{\bbA_1,\dots, \bbA_p}$ we denote the $\sigma$-algebra generated by the events $\bbA_1,\dots, \bbA_p$.
$\bbN=\{1,2,3,\ldots\}$. $\txs\E{\bX\overevent\bbA}:=\txs\E{\bX\ind{\bbA}}$ denotes the expectation of a random variable $\bX$ over event $\bbA$.
Note that
\begin{equation}\label{eq:expectation-over2cond}
\txs\E{\bX\overevent \bbA}=\txs\E{\bX\given \bbA}\txs\prob{\bbA}.
\end{equation}
For a random vector or matrix $\bX$, $\txs\E{\bX}:=\left[ \txs\E{\bX_{(i,j)}} \right]$.
Note that $\N{\txs\E{\bX}}_{\UI}\le\txs\E{\N{\bX}_{\UI}}$ for $\UI=2,{\scs\F}$.
Write $\covc{\bX,\bY}:=\txs\E{[\bX-\txs\E{\bX}]\circ[\bY-\txs\E{\bY}]}$ and $\varc{\bX}:=\covc{\bX,\bX}$.

Denote by $\Grassmann_p(\bbR^d)$ the Grassmann manifold of all $p$-dimensional subspaces of $\bbR^d$.
For two subspaces $\cX,\,\cY\in\Grassmann_p(\bbR^d)$,
let $X,\,Y\in{\mathbb C}^{d\times p}$  be the  basis matrices of
$\cX$ and $\cY$, respectively, i.e.,
$\cX=\cR(X)$ and $\cY=\cR(Y)$,
and denote by $\sigma_j$ for $1\le j\le p$ in  nondecreasing order, i.e.,
$\sigma_1\le\cdots\le\sigma_p$, the singular values of
$
(X^{\T}X)^{-1/2}X^{\T}Y(Y^{\T}Y)^{-1/2}.
$
The $p$ \emph{canonical angles $\theta_j(\cX,\cY)$
	between $\cX$ and $\cY$\/} are defined
by
$ 0\le\theta_j(\cX,\cY):=\arccos\sigma_j\le\tfrac {\pi}2$ 
for $1\le j\le p$.
They are in  non-increasing order, i.e., $\theta_1(\cX,\cY)\ge\cdots\ge\theta_p(\cX,\cY)$. Set
$
\Theta(\cX,\cY)=\diag(\theta_1(\cX,\cY),\ldots,\theta_p(\cX,\cY)).
$
It can be seen that angles so defined are independent of the  basis matrices $X$ and $Y$, which are not unique.
With the definition of canonical angles,
$ \N{\sin\Theta(\cX,{\cal Y})}_{\UI}$ 
for $\UI=2,{\scriptstyle \F}$ 
are metrics on $\Grassmann_p(\bbR^d)$ \cite[Section~II.4]{stewartS1990matrix}.

In what follows, we sometimes place a vector or matrix in one or both
arguments of $\theta_j(\,\cdot\,,\,\cdot\,)$ and $\Theta(\,\cdot\,,\,\cdot\,)$ with the understanding that it is about
the subspace spanned by the vector or the columns of the matrix argument.
For any
$X\in \bbR^{d\times p}$,
if $X_{(1:p,:)}$ is nonsingular, then we can define
\begin{equation}\label{eq:scrT-dfn}
\scrT(X):=X_{(p+1:d,:)}X_{(1:p,:)}^{-1}.
\end{equation}

\section{Related Work} \label{sec:related-work}

Let $\bX\in \bbR^d$ be a $d$-dimensional random vector with mean $\txs\E{\bX}$ and covariance
\[
\Sigma=\txs\E{(\bX-\txs\E{\bX})(\bX-\txs\E{\bX})^{\T}}.
\]
To reduce the dimension of $\bX$ from $d$ to $p$ (usually $p \ll d$),
PCA looks for a $p$-dimensional linear subspace that is closest to the centered random vector $\bX-\txs\E{\bX}$ in the mean squared sense,
through the independent and identically distributed samples $X^{(1)},\dots,X^{(n)}$.

Without loss of generality, we assume $\txs\E{\bX}=0$.
Then PCA corresponds to a stochastic optimization problem
\begin{equation}\label{eq:PCA-opt}
\min\limits_{\cU\in\Grassmann_p(\bbR^d)}\,\txs\E{\N{(I_d-\Pi_{\cU})\bX}_2^2},
\end{equation}
where  $\Pi_{\cU}$ is the orthogonal projector
onto the subspace $\cU$.
Let $\Sigma=U\Lambda U^{\T}$ be the spectral decomposition of $\Sigma$,
where
\begin{equation}\label{eq:eigD-covar}
\Lambda=\diag(\lambda_1,\dots,\lambda_d)
\,\,\text{with $\lambda_1\ge\dots\ge \lambda_p \ge \lambda_{p+1}\ge\dots\ge\lambda_d\ge0$},
\end{equation}
and orthogonal $U=[u_1,\dots,u_d]$.
If $\lambda_p>\lambda_{p+1}$,
then the unique solution to the optimization problem \cref{eq:PCA-opt}, namely the $p$-dimensional principal
subspace of $\Sigma$, is $\cU_*=\cR([u_1,\dots,u_p])$, the subspace spanned by $u_1,\dots,u_p$.
In practice, $\Sigma$ is unknown, and sample data $\{X^{(1)},\ldots,X^{(n)} \}$ is generally used to estimate $\cU_*$.
The classical PCA does it by the spectral decomposition of the empirical covariance matrix $\what \Sigma=\tfrac{1}{n}\txs\sum\limits_{i=1}^n X^{(i)}(X^{(i)})^{\T}$. 
Specifically, the classical PCA uses $\what{\cU}_*=\cR([\what u_1,\dots,\what u_p])$ to estimate $\cU_*$,
where $\what u_i$ is the corresponding eigenvectors of $\what \Sigma$.
In the classical PCA, obtaining the empirical covariance matrix has time complexity $\OO(nd^2)$ and space complexity $\OO(d^2)$.
So storing and calculating a large empirical covariance matrix can be very expensive when the data are of high dimension,
not to mention the cost $\OO(d^3)$
by dense solvers or $\OO(pnd)$ (more of $\OO(p^2nd)$ with full reorthogonalization for robustness) by some iterative methods for computing its eigenvalues and eigenvectors \cite{demm:1997}.

To analyze the accuracy of the above estimation using a finite number of samples, an important quantity is the distance between $\cU_*$ and $\what{\cU}_*$ by their canonical angles. 
Vu and Lei~\cite[Theorem~3.1]{vuL2013minimax} proved that
if $p(d-p)\tfrac{\sigma_*^2}{n}$ is bounded for some constant $\sigma_*$, then
\begin{equation}\label{eq:minimax-bound}
\inf\limits_{\what{\cU}_*\in\Grassmann_p(\bbR^d)}
\sup\limits_{\bX\in\mathcal{P}_0(\sigma_*^2,d)}
\txs\E{\N{\sin\Theta(\what{\cU}_*,{\cU}_*)}_{\F}^2}
\ge cp(d-p)\tfrac{\sigma_*^2}{n},
\end{equation}
where $c>0$ is an absolute constant, and $\mathcal{P}_0(\sigma_*^2,d)$ is the set of all $d$-dimensional
sub-Gaussian distributions for which the eigenvalues of the covariance matrix satisfy
\begin{equation}\label{eq:common-rate}
\frac{\lambda_1\lambda_{p+1}}{(\lambda_p-\lambda_{p+1})^2}\le\sigma_*^2.
\end{equation}
Note that its left-hand side is the effective noise variance.

To reduce both the time and space complexities, Oja~\cite{oja1982simplified} proposed an
\emph{online PCA iteration}:
\begin{equation}\label{eq:onlinePCA-1v}
\wtd u^{(n)}= u^{(n-1)}+\beta^{(n-1)} X^{(n)}(X^{(n)})^{\T}u^{(n-1)},
\,\,
u^{(n)}=\wtd u^{(n)}\N{\wtd u^{(n)}}_2^{-1},
\end{equation}
to approximate the first principal component,
where $\beta^{(n)}>0$ is a stepsize.
Later Oja and Karhunen~\cite{ojaK1985stochastic} proposed a \emph{subspace online PCA iteration}:
\begin{equation}\label{eq:onlinePCA-9v}
\wtd U^{(n)}= U^{(n-1)}+X^{(n)}(X^{(n)})^{\T}U^{(n-1)}\diag(\beta_1^{(n-1)},\dots,\beta_p^{(n-1)}),
\,\,
U^{(n)}=\wtd U^{(n)}R^{(n)},
\end{equation}
to approximate the principal subspace $\cU_*$,
where $\beta_i^{(n)}>0$ for $1\le i\le p$ are stepsizes, and $R^{(n)}$ is a normalization matrix to make $U^{(n)}$ have orthonormal columns. The QR decomposition  is often used by almost all existing works in the literature; see, e.g., \cite{ojaK1985stochastic,allenzhuL2017first,mianjyA2018stochastic} and references therein. It can be seen that these methods update the approximations incrementally by processing data one vector at a time as
soon as it comes in, completely avoiding the explicit calculation of the empirical covariance matrix.
In the subspace online PCA, obtaining an approximate principal subspace has time complexity $\OO(p^2d)$  and space complexity $\OO(pd)$ per iterative step. 

Recently, \LWLZ~\cite{liWLZ2017near}
proved a nearly optimal convergence rate
for the iteration \cref{eq:onlinePCA-1v} for the distributions with sub-Gaussian tails (note 
the samples of this kind of distributions may be unbounded). One of their main results reads as follows.
For the initial guess $u^{(0)}$ that is randomly chosen according to a uniform distribution
and the stepsize $\beta$ is chosen in accordance with the sample size $n$,
there exists a high-probability event $\bbA_*$ with $\txs\prob{\bbA_*}\ge 1-\delta$ such that%
\begin{subequations}\label{eq:result-p1}
	\begin{align}
	\txs\E{\abs{\tan\Theta(u^{(n)},u_*)}^2 \given \bbA_*}
	&\le C(d,n,\delta)\tfrac{\ln n}{n}\tfrac{1}{\lambda_1-\lambda_2}\txs\sum\limits_{i=2}^{d} \tfrac{\lambda_1\lambda_i}{\lambda_1-\lambda_i}
	\label{eq:result-p1a}\\
	&\le C(d,n,\delta)\tfrac{\lambda_1\lambda_2}{(\lambda_1-\lambda_2)^2}\tfrac{(d-1)\ln n}{n}, \label{eq:result-p1b}
	\end{align}
\end{subequations}
where $\delta\in[0,1)$, $u_*=u_1$  is the first principal component, and $C(d,n,\delta)$ can be approximately treated as a constant
because for sufficiently large $d$, $C(d,n,\delta)$ goes to a constant as $n\to\infty$.
It can be seen that this bound matches the minimax lower bound \cref{eq:minimax-bound}
up to a logarithmic factor of $n$, hence, \emph{nearly optimal}.
It is significant because a uniformly distributed initial value is nearly orthogonal
to the principal component with high probability when $d$ is large \cite[section~2.4]{blumHK2020foundations},
and thus such a random initial vector is not a very good initial guess to start an iteration with.
This result is more general than previous ones in \cite{2015arXiv150103796B,2016arXiv160206929J,2015arXiv150909002S}, because
it is for distributions that can possibly be unbounded,
and the convergence rate is nearly optimal and nearly global.
Unfortunately, the above significant work~\cite{liWLZ2017near} on the online PCA iteration cannot be trivially generalized to the subspace online PCA iteration due to three major difficulties 
to be discussed in \cref{sec:comparison}.

\section{Efficient Subspace Online PCA} \label{sec:PCA-reformulate}
Let $X^{(n)}\in\bbR^d$ for $n=1,2,\dots$ be independent and identically distributed samples of $\bX$. As $\{X^{(1)},\dots,X^{(n)}\}$ comes in a sequential order, the subspace online PCA iteration \cref{eq:onlinePCA-9v} of
Oja and Karhunen~\cite{ojaK1985stochastic} is used to compute the principal subspace of dimension $p$. Differing from \cref{eq:onlinePCA-9v}, our proposed subspace online PCA has the following changes:
\begin{enumerate}
	\item a fixed stepsize $\beta_{i}^{(n)} = \beta > 0, \forall n, i=1,\ldots,p$, is used;
	\item the normalization matrix to make $U^{(n)}$ have orthonormal columns are explicitly given by
	\begin{equation}\label{eq:onlinePCA-9v:R}
	R^{(n)}=[(\wtd U^{(n)})^{\T}\wtd U^{(n)}]^{-1/2}.
	\end{equation}
\end{enumerate}
With the changes, our subspace online PCA iteration becomes
\begin{equation}\label{eq:our-onlinePCA-9v}
\wtd U^{(n)}= U^{(n-1)}+ \beta X^{(n)}(X^{(n)})^{\T}U^{(n-1)},
\,\,
U^{(n)}=\wtd U^{(n)} [(\wtd U^{(n)})^{\T}\wtd U^{(n)}]^{-1/2}.
\end{equation}
It can be verified that $U^{(n)}$ has orthonormal columns.
This variant is equivalent to \cref{eq:onlinePCA-9v} in the sense that
both $U^{(n)}$ here and the one there  have the same column space.
It turns out that the matrix square root and the inverse in \cref{eq:onlinePCA-9v:R} can be done
analytically as in Lemma~\ref{lm:basic} below, leading to a simple and computationally economical formula
for $U^{(n)}$ of \cref{eq:our-onlinePCA-9v}.
An equivalent of Lemma~\ref{lm:basic}  was  implied in \cite{abedmeraimACH2000orthogonal}, although not
explicitly and rigorously stated,
to analytically transform iteration formula \cref{eq:our-onlinePCA-9v}.
For that reason, we credit the lemma to  \cite{abedmeraimACH2000orthogonal}, but provide a proof for
completeness because of some missing details in the derivation in  \cite{abedmeraimACH2000orthogonal}.

\begin{lemma}[{\cite{abedmeraimACH2000orthogonal}}]\label{lm:basic}
	Let $V\in{\mathbb R}^{d\times p}$ with $V^{\T}V=I_p$, $0\ne x\in{\mathbb R}^d$, and $0<\beta\in{\mathbb R}$, and let
	$
	W:=V+\beta xx^{\T}V=(I_d+\beta xx^{\T})V, \quad V_+:=W(W^{\T}W)^{-1/2}.
	$
	If $V^{\T}x\ne 0$, then
	\[
		V_+=V+\beta\tilde\alpha xz^{\T}-\frac {1-\tilde\alpha}{\gamma^2}Vzz^{\T},
	\]
	where
	$z=V^{\T}x$, $\gamma=\N{z}_2$, $\tilde z= z/{\gamma}$, and $\alpha=\beta(2+\beta \N{x}_2^2)\gamma^2$, and $\tilde\alpha=(1+\alpha)^{-1/2}$.
	In particular, $V_+^{\T}V_+=I_p$.
\end{lemma}

\begin{proof}
	We have
	$ 
	W^{\T}W  = V^{\T}[I_d+\beta xx^{\T}]^2V
	=I_p+\alpha \tilde z\tilde z^{\T}.
	$ 
	Let $Z_{\bot}\in{\mathbb R}^{p\times (p-1)}$ such that $[\tilde z,Z_{\bot}]^{\T}[\tilde z,Z_{\bot}]=I_p$. The eigen-decomposition of
	$W^{\T}W$ is
	$
	W^{\T}W=[\tilde z,Z_{\bot}]\begin{bsmallmatrix}
	1+\alpha &  \\
	& I_{p-1}
	\end{bsmallmatrix}[\tilde z,Z_{\bot}]^{\T}
	$
	which yields
	\[
	(W^{\T}W)^{-1/2} 
	=[\tilde z,Z_{\bot}]\begin{bsmallmatrix}
	(1+\alpha)^{-1/2} &  \\
	& I_{p-1}
	\end{bsmallmatrix}[\tilde z,Z_{\bot}]^{\T} 
	=I_p-[1-(1+\alpha)^{-1/2}]\tilde z\tilde z^{\T}.
	\]
	Therefore,
	\begin{align*}
	V_+&=(V+\beta xx^{\T}V)\{I_p-[1-(1+\alpha)^{-1/2}]\tilde z\tilde z^{\T}\} \\
	&=V+\beta xx^{\T}V-[1-(1+\alpha)^{-1/2}](V+\beta xx^{\T}V)\tilde z\tilde z^{\T} 
	\qquad(\text{use $x^{\T}V=z^{\T}=\gamma\tilde z^{\T}$})\\
	&=V+\beta\gamma x\tilde z^{\T}-[1-(1+\alpha)^{-1/2}]V\tilde z\tilde z^{\T}-[1-(1+\alpha)^{-1/2}]\beta\gamma x\tilde z^{\T}\\
	&=V+(1+\alpha)^{-1/2}\beta x z^{\T}-\tfrac {1-(1+\alpha)^{-1/2}}{\gamma^2}V z z^{\T},
	\end{align*}
	as expected, knowing $\tilde\alpha=(1+\alpha)^{-1/2}$.
\end{proof}

To apply this lemma to transform \cref{eq:our-onlinePCA-9v}, we perform substitutions:
\[
\wtd U^{(n)} \leftarrow W,\,\, U^{(n-1)}\leftarrow V,\,\,  U^{(n)}\leftarrow V_+,\,\,  X^{(n)}\leftarrow x,\,\,
Z^{(n)}\leftarrow z.
\]
to get
\[
U^{(n)} 
=U^{(n-1)} + \beta (1+\alpha^{(n)})^{-1/2}X^{(n)} (Z^{(n)})^{\T}
           - [1-(1+\alpha^{(n)})^{-1/2}]\tfrac{U^{(n-1)} Z^{(n)} (Z^{(n)})^{\T} }{\N{Z^{(n)}}_2^2},
\]
where $\alpha^{(n)} = \beta\left(2+\beta (X^{(n)})^{\T}X^{(n)}\right) \N{Z^{(n)}}_2^2$ and
$Z^{(n)}= (U^{(n-1)})^{\T}X^{(n)}$.
\begin{algorithm}[t]
	\caption{Subspace Online PCA}\label{alg:onlinePCA:ppr1}
	\begin{algorithmic}[1]
		\STATE Choose $U^{(0)}\in\bbR^{d\times p}$ with $(U^{(0)})^{\T}U^{(0)}=I$,
		and choose stepsize $\beta>0$.
		\FOR{$n=1,2,\dots$ until convergence}
		\STATE Take an $\bX$'s sample $X^{(n)}$;
		\STATE $Z^{(n)}= (U^{(n-1)})^{\T}X^{(n)}$,
		$\alpha^{(n)}=\beta\left(2+\beta (X^{(n)})^{\T}X^{(n)}\right)(Z^{(n)})^{\T}Z^{(n)}$,
		$\wtd\alpha^{(n)}=(1+\alpha^{(n)})^{-1/2}$;
		\STATE
		$
		U^{(n)}
		= U^{(n-1)}+\beta \wtd\alpha^{(n)} X^{(n)}(Z^{(n)})^{\T} - \tfrac{1-\wtd\alpha^{(n)}}{(Z^{(n)})^{\T}Z^{(n)}}U^{(n-1)}Z^{(n)}(Z^{(n)})^{\T}.
		$
		\ENDFOR
	\end{algorithmic}
\end{algorithm}
Finally, we outline
in \Cref{alg:onlinePCA:ppr1} the subspace online PCA algorithm derived from \cref{eq:our-onlinePCA-9v}.
This is essentially the same as the orthogonal Oja algorithm \cite{abedmeraimACH2000orthogonal} and will be one we
are going to analyze. Computationally, it has the advantages of not
involving any explicit orthogonalization by the Gram-Schmidt process or matrix square root,
but only in terms of matrix-vector multiplications. This formulation is numerically  stable
and computationally fast.
At convergence, it is expected that
\[
U^{(n)}\to U_*:= U\begin{bsmallmatrix} I_p \\ 0\end{bsmallmatrix}=[u_1,u_2,\ldots,u_p]
\]
in the sense that $\N{\sin\Theta(U^{(n)}, U_*)}_{\UI}\to 0$ as $n\to\infty$.
The rest of this paper is devoted to analyze its convergence, with the help of the next lemma.

\begin{lemma}\label{lm:Tan(Theta)}
	For $V\in \bbR^{d\times p}$
	with nonsingular $V_{(1:p,:)}$, we have for $\UI=2,{\scs \F}$
	\begin{equation}\label{eq:Tan(Theta)}
	\N*{\tan\Theta(V,\begin{bsmallmatrix} I_p \\ 0 \end{bsmallmatrix})}_{\UI}
	=\N{\scrT(V)}_{\UI},
	\end{equation}
where $\scrT(V)$ is defined as in \eqref{eq:scrT-dfn}.
\end{lemma}

\begin{proof}
Let $Y=\begin{bsmallmatrix} I_p \\ 0 \end{bsmallmatrix}\in \bbR^{d\times p}$.
It can be seen that the singular values $\sigma_j=\cos\theta_j(V,Y)$ of
\[
	\left[I+\scrT(V)^{\T}\scrT(V)\right]^{-1/2}\begin{bsmallmatrix} I \\ \scrT(V) \end{bsmallmatrix}^{\T}\begin{bsmallmatrix} I \\ 0 \end{bsmallmatrix}
		=\left[I+\scrT(V)^{\T}\scrT(V)\right]^{-1/2}
\]
and the singular values $\tau_j$ of $\scrT(V)$  are related by
$\tau_j=\tfrac {\sqrt{1-\sigma_j^2}}{\sigma_j}=\tan\theta_j(V,Y)$.
Hence the identity \cref{eq:Tan(Theta)} holds.
\end{proof}


Notations introduced in this section, except those in Lemma~\ref{lm:basic} will be adopted throughout the rest of this paper.

\section{Main  Results} \label{sec:main-results}
For convenience, we first review our setting.
Let
$
\bX=[\bX_1, \bX_2, \ldots, \bX_d]^{\T}
$
be a random vector in $\bbR^d$. Assume $\txs\E{\bX}=0$. Its covariance matrix
$\Sigma:=\txs\E{\bX\bX^{\T}}$ has the spectral decomposition
\begin{equation}\label{eq:eigD-convar}
\Sigma=U\Lambda U^{\T}
\quad\text{with}\quad
U=[u_1,u_2,\ldots,u_d],\,\,
\Lambda=\diag(\lambda_1,\dots,\lambda_d),
\end{equation}
where $U\in\bbR^{d\times d}$ is orthogonal, and $\lambda_i$ for $1\le i\le d$ are the eigenvalues of $\Sigma$,
arranged for convenience in  non-increasing order.
Assume
\begin{equation}\label{eq:eigv-assume}
\lambda_1\ge\dots\ge\lambda_p>\lambda_{p+1}\ge\dots\ge\lambda_d>0.
\end{equation}
Given $\{X^{(1)},\dots,X^{(n)}\}$ in a sequential order, the proposed subspace online PCA iteration \cref{eq:our-onlinePCA-9v} is used to compute the principal subspace $U^{(n)}$ of dimension $p$ to estimate
\begin{equation}\label{eq:calUp}
{\cal U}_*=\cR(U_{(:,1:p)})=\cR([u_1,u_2,\ldots,u_p]).
\end{equation}

Our major result on the convergence rate of the subspace online PCA iteration in \Cref{alg:onlinePCA:ppr1} states as follows:
if the initial guess $U^{(0)}$ is randomly chosen to satisfy that $\cR(U^{(0)})$ is uniformly sampled from $\Grassmann_p(\bbR^d)$,
and the stepsize $\beta_i^{(n)}$
is chosen the same for $1\le i\le p$ and in accordance with the sample size $n$,
then there exists a high-probability event $\bbH_*$ with $\txs\prob{\bbH_*}\ge 1- 2\delta^{p^2}$, such that%
\begin{subequations}\label{eq:result-p}
	\begin{align}
	\txs\E{\N{{\tan\Theta(U^{(n)},U_*)}}_{\F}^2\given \bbH_*}
	&\le C(d,n,\delta)\tfrac{\ln n}{n}\tfrac{1}{\lambda_p-\lambda_{p+1}}\txs\sum\limits_{j=1}^{p}\txs\sum\limits_{i=p+1}^{d} \tfrac{\lambda_j\lambda_i}{\lambda_j-\lambda_i}
	\label{eq:result-p-a} \\
	&\le C(d,n,\delta)\tfrac{\lambda_p\lambda_{p+1}}{(\lambda_p-\lambda_{p+1})^2}\tfrac{ p(d-p)\ln n}{n}		,
	\label{eq:result-p-b}
	\end{align}
\end{subequations}
where the constant $C(d,n,\delta)\to 24\psi^4/(1-\delta^{p^2})$
as $d\to\infty$ and $n\to\infty$, and $\psi$ is $\bX$'s Orlicz-$\psi_2$ norm (see Definition~\ref{dfn:orlicz-norm} below). 
This also matches the minimax lower bound \cref{eq:minimax-bound}
up to a logarithmic factor of $n$, and hence is \emph{nearly optimal} and \emph{nearly global} for the subspace online PCA, in the same way as \eqref{eq:result-p1} of \LWLZ~\cite{liWLZ2017near} for
the vector online PCA. Both are valid for any sub-Gaussian distribution.

Comparing \cref{eq:result-p} and \cref{eq:result-p1}, we find that \cref{eq:result-p1} becomes the special case of our results \cref{eq:result-p} in the case of $p=1$. Unfortunately, the proving technique in \cite{liWLZ2017near} used for the one-dimensional case ($p=1$) is not generalizable to the multi-dimensional case ($p>1$).
More detail will be forthcoming in \cref{sec:comparison}.

We also note that the factor in our result is
\[
\tfrac{\lambda_p\lambda_{p+1}}{(\lambda_p-\lambda_{p+1})^2}
\quad\text{\emph{vs.}}\quad
\tfrac{\lambda_1\lambda_{p+1}}{(\lambda_p-\lambda_{p+1})^2}.
\]
The second quantity appeared in \cref{eq:common-rate}.
The first quantity is always smaller but both are of similar order if $\lambda_1$ and $\lambda_p$ are
of similar order. However, their magnitude can differ greatly
when $\lambda_p\ll \lambda_1$. 

\subsection{Three Main Theorems}

In this subsection, we will state our three main theorems of the paper for the multi-dimensional case and \cref{eq:result-p} is a consequence of them. Before that, we will introduce necessary definitions and assumptions. We would like to point out that any statement we will make is meant to hold \emph{almost surely}.

We are concerned with random variables/vectors that have  a sub-Gaussian distribution.
To that end, we need to introduce the Orlicz $\psi_\alpha$-norm of a random variable/vector.
More details can be found in \cite{vaartW1996weak}.

\begin{definition}\label{dfn:orlicz-norm}
	The \emph{Orlicz $\psi_\alpha$-norm\/} of a random variable $\bX\in \bbR$ is defined as
	\[
	\N{\bX}_{\psi_\alpha}:=\inf\txs\set*{\xi>0\setdelim \txs\E{\exp\left( \abs*{\tfrac{\bX}{\xi}}^\alpha \right)}\le2},
	\]
	and the \emph{Orlicz $\psi_\alpha$-norm\/} of a random vector $\bX\in \bbR^d$ is defined as
	\[
	\N{\bX}_{\psi_\alpha}:=\sup\limits_{\N{v}_2=1}\N{v^{\T}\bX}_{\psi_\alpha}.
	\]
	We say random variable/vector $\bX$ follows a \emph{sub-Gaussian distribution\/} if $\N{\bX}_{\psi_2}<\infty$.
\end{definition}

By the definition, 
any bounded random variable/vector follows a sub-Gaussian distribution.
To prepare our convergence analysis, we make a few assumptions.



\begin{assumption} \label{asm:simplify:X}
	$\bX=[\bX_1, \bX_2, \ldots, \bX_d]^{\T}\in\bbR^d$ is a random vector.
	\begin{enumerate}[label=(A-\arabic*),ref=(A-\arabic*),leftmargin=3\parindent]
		\item \label{asm:simplify:X:moment}
		$\txs\E{\bX}=0$, and $\Sigma:=\txs\E{\bX\bX^{\T}}$ has the spectral decomposition \cref{eq:eigD-convar} satisfying
		\cref{eq:eigv-assume};
		\item \label{asm:simplify:X:subGaussian}
		$\psi:=\N{\Sigma^{-1/2}\bX}_{\psi_2}<\infty$.
	\end{enumerate}
\end{assumption}

The principal subspace $\cU_*$ in \cref{eq:calUp} is uniquely determined under \cref{asm:simplify:X:moment} of \Cref{asm:simplify:X}.
On the other hand, \cref{asm:simplify:X:subGaussian} of \Cref{asm:simplify:X}  ensures that all 1-dimensional marginals of $\bX$ have
sub-Gaussian tails, or equivalently,
$\bX$ follows a sub-Gaussian distribution.
This is also an assumption that is used in \cite{liWLZ2017near}.

In what follows, we will state our main results under the assumption and leave their proofs to \cref{sec:proof-of-cref-thm-thm1-ppr1,sec:proof-of-cref-thm-thm2-ppr1} because of their high complexity.
To that  end, first we introduce some quantities:
\begin{itemize}
	\item
the eigenvalue gap
$
\gamma:=\lambda_p-\lambda_{p+1},
$
\item
	the sum of the top $i$ eigenvalues
$
\eta_i:=\lambda_1+\dots+\lambda_i,\quad
i=1,\dots,d,
$
\item
	the dominance of the top $i$ eigenvalues
$
\mu_i:
=\tfrac{\eta_i}{\eta_d}
\in \left[\tfrac{i}{d}, 1\right],
$
\item
	for $s>0$ and the stepsize $\beta<1$ such that $\beta\gamma<1$,  integer function
\begin{equation}\label{eq:Ns-dfn}
		N_s(\beta):=\min\txs\set*{n\in \mathbb{N}\setdelim (1-\beta\gamma)^n\le \beta^s}
		= \ceil*{\tfrac{s\ln\beta }{\ln (1-\beta\gamma)}},
\end{equation}
where $\ceil*{\cdot}$ is the ceiling function taking the smallest integer that is no smaller than its argument,
\item and finally,  for $0<\varepsilon<1/7$, integer function
	\[
		M(\varepsilon)
		:=\min\txs\set*{m\in \mathbb{N}\setdelim \beta^{7\varepsilon/2-1/2}\le\beta^{(1-2^{1-m})(3\varepsilon-1/2)}}
		= 2+\ceil*{\tfrac{\ln \tfrac{1/2-3\varepsilon}{\varepsilon}}{\ln 2}}\ge 2.
	\]
\end{itemize}
In practice, it is always desirable to use a good initial guess in an iterative method whenever there is one available
because
it  positively affects computational efficiency in reducing the number
of iterations required to achieve an approximation within a prescribed tolerance.
On the other hand, when there isn't one known,
a randomly chosen initial guess is often taken.
Our first main result  in
\Cref{thm:thm1:ppr1} covers the case when a somewhat good initial subspace
$U^{(0)}$ is available whereas our second main result  in \Cref{thm:thm2:ppr1} is about using a randomly chosen initial subspace.

\begin{theorem}\label{thm:thm1:ppr1}
	Given
		$\varepsilon\in(0,1/7)$,
		$\omega\in(0,1)$,
		$\phi>0$,
and $\kappa$ and $\beta$ satisfying
		\begin{equation}\label{eq:beta-bound}
\kappa >6^{[M(\varepsilon)-1]/2}\max\txs\set{\sqrt{2},\linebreak 2(\sqrt{2}-1)^{1/2}\phi \lambda_1^{-1/2}\omega^{1/2}},\,\,
		0< \beta < \min\txs\set*{1,
			\left( \tfrac{1}{8\kappa \eta_p } \right)^{\tfrac{2}{1-4\varepsilon}},
			\left(\tfrac{\gamma}{130\kappa ^2\eta_p ^2}\right)^{\tfrac{1}{\varepsilon}}
		}
		.
		\end{equation}
	Let $U^{(n)}$ for $n=1,2,\dots$ be the approximations of $U_*$ generated by \emph{\Cref{alg:onlinePCA:ppr1}}.
	Under \emph{\Cref{asm:simplify:X}},
	if
\begin{equation}\label{eq:initial-guess}
\N{\tan \Theta(U^{(0)}, U_*)}_2^2\le \phi^2 d-1,
\end{equation}
and
	\begin{equation}\label{eq:beta-d-bound}
	(\sqrt{2}+1)\lambda_1d\beta^{1-7\varepsilon}\le \omega,
	\quad
	K>N_{3/2-37\varepsilon/4}(\beta),
	\end{equation}
	then there exist absolute
	constants\footnote {We attach each with a subscript for the convenience of indicating
		their associations. They don't change as the values of the subscript variables
		vary, by which we mean \emph{absolute constants}. Later in
		\cref{eq:C}, we explicitly bound these absolute constants.}
	$C_\psi$, $C_{\nu}$, $C_{\circ}$ and a high-probability event $\bbH$ with
	\begin{equation}\label{eq:prob-bbH}
		\txs\prob{\bbH}\ge 1-
		K[(2+\ee)d+ p+1]\exp(-C_{\nu\psi}\beta^{-\varepsilon})
	\end{equation}
	such that for any $n\in [N_{3/2-37\varepsilon/4}(\beta),K]$
\begin{equation}\label{eq:main-result-thm1}
	\txs\E{\N{\tan\Theta(U^{(n)},U_*)}_{\F}^2\overevent \bbH}
	\le(1-\beta\gamma)^{2(n-1)}p\phi^2d
	+\tfrac{32\psi^4\beta}{2-\lambda_1 \beta}\varphi(p,d;\Lambda)
	 +C_{\circ}\kappa ^4\mu_p^{-2}\eta_p ^2\gamma^{-1}p\sqrt{d-p}\beta^{3/2-7\varepsilon},
\end{equation}
	where $\ee=\exp(1)$ is Euler's number, $C_{\nu\psi}=\max\txs\set{C_{\nu} \mu_p,C_\psi\min\txs\set{\psi^{-1},\psi^{-2}}}$,
	and
	\begin{equation}\label{eq:varphi}
	\varphi(p,d;\Lambda)
	:= \sum_{j=1}^{p}\sum_{i=p+1}^{d} \frac{\lambda_j\lambda_i}{\lambda_j-\lambda_i}
	\in \left[\frac{p(d-p)\lambda_1\lambda_d}{\lambda_1-\lambda_d}, \frac{p(d-p)\lambda_p\lambda_{p+1}}{\lambda_p-\lambda_{p+1}}\right].
	\end{equation}
\end{theorem}

\begin{remark}\label{rk:main-1}
	\begin{enumerate}
		\item
Although an interval is presented in \cref{eq:varphi}  to bound $\varphi(p,d;\Lambda)$, there are more informative ones
under additional assumptions on random vector $\bX$. For example,
in some of the past works \cite{allenzhuL2017first,2015arXiv150103796B,2016arXiv160206929J,2015arXiv150909002S,aroraCLS2012stochastic,aroraCS2013stochastic,marinovMA2018streaming,mianjyA2018stochastic}, it is assumed
$\sum_{i=1}^d \lambda_i=\E{\N{\bX}_2^2}\le c$ for some constant $c$, independent of dimension $d$. Then
\begin{equation}\label{eq:psi-bd-by-tr}
\varphi(p,d;\Lambda)=\sum_{j=1}^{p}\sum_{i=p+1}^{d} \frac{\lambda_j\lambda_i}{\lambda_j-\lambda_i}
   \le\frac 1{\lambda_p-\lambda_{p+1}}\sum_{j=1}^{p}\sum_{i=p+1}^{d} \lambda_j\lambda_i
   \le\frac 1{\gamma}\left(\sum_{j=1}^{p}\lambda_j\right)\left(c-\sum_{i=1}^{p}\lambda_i\right)
   \le\frac {c^2}{4\gamma}.
\end{equation}
As a result, the second term in the right-hand side of \cref{eq:main-result-thm1} is of $O(\beta)$.
Under the same assumption, after a careful check (of Appendix~\ref{ssec:estimation-in-proof-of-lemma-lm}), 
the third term can be ensured of $O(\beta)$, too, by making $7\varepsilon\le 1/2$.
Both terms do not go to $0$ as $n\to\infty$, as we would like to  ideally have. Nonetheless, we argue that it doesn't diminish the usefulness of the error bound.
Here is why. Like in any iterative method, the ultimate goal is to drive approximation error down to a prescribed level. Since the terms are of $O(\beta)$, given a prescribed error tolerance, we can always take the stepsize $\beta$ in the same order of the tolerance to yield an eventual approximation to the subspace within the desired error level.
\item 
      Theorem~\ref{thm:thm1:ppr1} involves a  set of pre-chosen constant parameters:
      $\epsilon$, $\omega$, $\phi$, $\kappa$, and $\beta$ subject to the inequalities in \eqref{eq:beta-bound} so that $K[(2+\operatorname{e})d+ p+1]\exp(-C_{\nu\psi}\beta^{-\varepsilon})$ is sufficiently tiny to make $\mathbb H$ a high probability event. For that reason
      $n$ is limited to no bigger than $K$. Ideally, the event $\mathbb H$ should exist
      with high probability for all sufficiently large $n$.
According to our proof, the theorem remains valid with simply setting $K$ to $n$:
		\begin{equation}\tag{\ref{eq:prob-bbH}'}
	n>N_{3/2-37\varepsilon/4}(\beta),
	\qquad
	\operatorname{P}(\mathbb H)\ge 1-
	n[(2+\operatorname{e})d+ p+1]\exp(-C_{\nu\psi}\beta^{-\varepsilon}),
		\end{equation}
everything else being equal. This means with any given constant parameters, there is no guarantee that
$\mathbb H$ is still a high probability event if $n$ is too large. While this is not ideal,
we argue that if the number $n$ of samples or some rough range of it is known, we can always optimize these constant parameters,
by making $\beta$ small enough, so that $n[(2+\operatorname{e})d+ p+1]\exp(-C_{\nu\psi}\beta^{-\varepsilon})$ is still tiny to render
a high probability event $\mathbb H$. For example, in Theorem 4.5, we specify what is needed
on the constant parameters.
We point out in passing that the results in \LWLZ\ \cite{liWLZ2017near} for the vector online PCA also require that the number $n$ of samples be bounded from above.

One subtlety in bounding $\operatorname{P}(\mathbb H)$ from below as in (\ref{eq:prob-bbH}') is that now the event
$\mathbb H$ depends on $n$. Theorem 4.2 as stated with the preset $K$ ensures
one high probability event $\mathbb H$ for all $n\in [N_{3/2-37\varepsilon/4}(\beta),K]$.
From the practical point of view, the number of samples is always finite, i.e., such a $K$ does exist, and
one might have some idea about what it is. When we do, the constant parameters can be judiciously chosen to ensure $K[(2+\operatorname{e})d+ p+1]\exp(-C_{\nu\psi}\beta^{-\varepsilon})$ tiny.
\item
This remark applies to \Cref{thm:thm2:ppr1} later as well.
	\end{enumerate}
\end{remark}

\Cref{thm:thm1:ppr1} assumes a somewhat accurate initial subspace  $U^{(0)}$, i.e.,
satisfying \cref{eq:initial-guess} which isn't very restrictive because $\phi^2d-1$ can be very big for huge $d$.
As we mentioned earlier, often we don't have a good initial subspace, in which case,
we may simply resort to
a randomly selected $U^{(0)}$.
%
Consider the uniform distribution on $\Grassmann_p(\bbR^d)$, the one with the Haar invariant probability measure
(see \cite[Section~1.4]{chikuse2003statistics} and \cite[Section~4.6]{james1954normal}).
We are interested in a randomly selected $U^{(0)}$ such that
\begin{equation}\label{eq:random-U0}
\boxed{
	\parbox{7.0cm}{
		\; $\cR(U^{(0)})$ is uniformly sampled from $\Grassmann_p(\bbR^d)$.
	}
}
\end{equation}
The reader is referred to \cite[Section~2.2]{chikuse2003statistics} on how to generate such a uniform distribution
on $\Grassmann_p(\bbR^d)$.


\begin{theorem}\label{thm:thm2:ppr1}
	Under \emph{\Cref{asm:simplify:X}},
	for sufficiently large $d$ and any $\beta$ satisfying \cref{eq:beta-bound}
	with 
	\[
		\kappa =6^{[M(\varepsilon)-1]/2}\max\txs\set{2C_p,\sqrt{2}},
	\]
	and
	\[
	p<(d+1)/2,\quad
	\varepsilon\in(0,1/7),\quad
	\delta\in(0,2^{-1/p^2}),\quad
	K>N_{3/2-37\varepsilon/4}(\beta),
	\]
	where $C_p$ is a constant only dependent on $p$,
	if \cref{eq:random-U0} holds,
	and
	\[
	d\beta^{1-3\varepsilon}\le \delta^2,
	\,\,
	K[(2+\ee)d+ p+1]\exp(-C_{\nu\psi}\beta^{-\varepsilon})
	\le\delta^{p^2},
	\]
	then there exists a high-probability event $\bbH_*$ with
	$\txs\prob{\bbH_*}\ge 1- 2\delta^{p^2}$
	such that
\begin{equation}\label{eq:main-result-thm2}
	\txs\E{\N{\tan\Theta(U^{(n)},U_*)}_{\F}^2\overevent \bbH_*}
	\le (1-\beta\gamma)^{2(n-1)}p\,C_p^2\delta^{-2}d
	+\tfrac{32\psi^4\beta}{2-\lambda_1 \beta}\varphi(p,d;\Lambda)
	+C_{\circ}\kappa ^4\mu_p^{-2}\eta_p ^2\gamma^{-1}p\sqrt{d-p}\beta^{3/2-7\varepsilon}
\end{equation}
	for any $n\in [N_{3/2-37\varepsilon/4}(\beta),K]$,
	where $\varphi(p,d;\Lambda)$ is as in \cref{eq:varphi}.
\end{theorem}

Our third main result is about picking a nearly optimal stepsize $\beta$ for nearly optimal convergence rate, assuming
the sample size is reasonably large and fixed at $N_*$.
The idea is to pick a good $\beta$ to balance the terms in
the right-hand side of \cref{eq:main-result-thm2} subject to $N_*\ge N_{3/2}(\beta)$ (thus we also
need a large enough number of samples). The nearly optimal stepsize $\beta$ is
%
\begin{equation}\label{eq:beta-opt}
\beta=\beta_*:=\tfrac{3\ln N_*}{2\gamma N_*},
\end{equation}
which is consistent with the choice in \cite{liWLZ2017near} for $p=1$.


\begin{theorem}\label{thm:thm3:ppr1}
	Under \emph{\Cref{asm:simplify:X}},
	for sufficiently large  $d\ge 2p$  and a sufficiently large number $N_*$ of samples,
	$\varepsilon\in(0,1/7)$,
	$\delta\in(0,2^{-1/p^2})$
	satisfying
	\begin{equation}\label{eq:cond-beta}
	d\beta_*^{1-3\varepsilon}\le \delta^2,
	\,\,
	N_* [(2+\ee)d+ p+1]\exp(-C_{\nu\psi}\beta_*^{-\varepsilon})
	\le\delta^{p^2},
	\end{equation}
	where $\beta_*$ is given by \cref{eq:beta-opt},
	if \cref{eq:random-U0} holds,
	then there exists a high-probability event $\bbH_*$ with
	$\txs\prob{\bbH_*}\ge 1- 2\delta^{p^2}$, such that	
	\begin{equation}\label{eq:main-result-thm3}
	\txs\E{\N{{\tan\Theta(U^{(N_*)},U_*)}}_{\F}^2\overevent \bbH_*}
	\le C_*(d,N_*,\delta)\tfrac{\varphi(p,d;\Lambda)}{\lambda_p-\lambda_{p+1}}\tfrac{\ln N_*}{N_*}		,
	\end{equation}
	where the constant $C_*(d,N_*,\delta)\to 24\psi^4$
	as $d\to\infty, N_*\to\infty$,
	and $\varphi(p,d;\Lambda)$ is as in \cref{eq:varphi}.
\end{theorem}

In \Cref{thm:thm1:ppr1,,thm:thm2:ppr1,thm:thm3:ppr1}, the conclusions are stated
in term of the expectation of $\N{\tan\Theta(U^{(n)},U_*)}_{\F}^2$
over some highly probable event. These expectations can be turned into conditional expectations, thanks to the relation
\cref{eq:expectation-over2cond}. In fact, \cref{eq:result-p} is a consequence of \cref{eq:main-result-thm3} and
\cref{eq:expectation-over2cond}.

\subsection{Discussions of New Proving Techniques}\label{sec:comparison}
Our three theorems in the previous section, namely \Cref{thm:thm1:ppr1,,thm:thm2:ppr1,thm:thm3:ppr1}, are the analogs for $p>1$
of \LWLZ's three theorems \cite[Theorems~1, 2, and~3]{liWLZ2017near} which are for $p=1$ only. Naturally,
we would like to know how our results when applied to the case $p=1$  and our proofs would stand against
those in \cite{liWLZ2017near}.
We choose to compare our results with those in \cite{liWLZ2017near}
because \LWLZ\ \cite{liWLZ2017near}
dealt with sub-Gaussian samples whereas other existing works
in the literature studied the vector/subspace online PCA for bounded samples only.
In what follows, we will do a fairly detailed comparison. 
Before we do that,
let us state their theorems (in our notation).

\begin{theorem}[{\cite[Theorem~1]{liWLZ2017near}}]\label{thm:thm1:ppr1-LWLZ}
	Under \emph{\Cref{asm:simplify:X}} and $p=1$, suppose there exists a constant $\phi>1$ such that  $\tan\Theta(U^{(0)},U_*)\le \phi^2d$.
	Let
	\begin{alignat*}{2}
	\what N^{\rmo}(\beta,\phi)&:=\min\txs\set*{n\in \mathbb{N}\setdelim (1-\beta\gamma)^n\le [4\phi^2d]^{-1}}
	&&= \ceil*{\tfrac{-\ln[4\phi^2d] }{\ln (1-\beta\gamma)}},
	\\
	\what N_s(\beta)&:=\min\txs\set*{n\in \mathbb{N}\setdelim (1-\beta\gamma)^n\le [\lambda_1^2\gamma^{-1}\beta]^s}
	&&= \ceil*{\tfrac{s\ln[\lambda_1^2\gamma^{-1}\beta] }{\ln (1-\beta\gamma)}}.
	\end{alignat*}
	Then for any $\varepsilon\in(0,1/8)$, stepsize $\beta>0$ satisfying $d[\lambda_1^2\gamma^{-1}\beta]^{1-2\varepsilon}\le b_1\phi^{-2}$, and
	any $t>1$,
	there exists an event $\bbH$ with
	\[
	\txs\prob{\bbH}\ge 1-2(d+2)\what N^{\rmo}(\beta,\phi)\exp\left( -C_0[\lambda_1^2\gamma^{-1}\beta]^{-2\varepsilon} \right)-4d\what N_t(\beta)\exp\left( -C_1[\lambda_1^2\gamma^{-1}\beta]^{-2\varepsilon} \right),
	\]
	such that for any $n\in [\what N_1(\beta)+\what N^{\rmo}(\beta,\phi),\what N_t(\beta)]$
\begin{equation}\label{eq:main-result-thm1-LWLZ}
	\txs\E{\tan^2\Theta(U^{(n)},U_*) \overevent \bbH}
	\le (1-\beta\gamma)^{2[n-\what N^{\rmo}(\beta,\phi)]} + C_2\beta\varphi(1,d;\Lambda) + C_2 \txs\sum\limits_{i=2}^{d}\tfrac{\lambda_1-\lambda_2}{\lambda_1-\lambda_i}[\lambda_1^2\gamma^{-1}\beta]^{3/2-4\varepsilon},
\end{equation}
	where $b_1\in(0,\ln^2, 2/16)$, $C_0,C_1,C_2$ are  absolute constants.
\end{theorem}

What we can see that  \Cref{thm:thm1:ppr1} for $p=1$ is essentially the same as
\Cref{thm:thm1:ppr1-LWLZ}. In fact, since $(1-\beta\gamma)^{1-\what N^{\rmo}(\beta,\phi)}\le4\phi^2d\le (1-\beta\gamma)^{-\what N^{\rmo}(\beta,\phi)}$,
the upper bounds by \cref{eq:main-result-thm1} for $p=1$ and by \cref{eq:main-result-thm1-LWLZ}
are comparable in the sense that they are in the same order in $d,\beta,\delta$.
Naturally one may try to
generalize the proving techniques in \cite{liWLZ2017near} which is for the one-dimensional case ($p = 1$)
to handle the multi-dimensional case ($p > 1$). 
Indeed, we tried but didn't succeed, due to we believe there insurmountable obstacles.
In fact, one of the key steps in proof works for $p=1$ does not seem to work for $p>1$.
Next, we explain these obstacles in details.

The basic structure of the proof in \cite{liWLZ2017near} is to split the Grassmann manifold $\Grassmann_p(\bbR^d)$,
where the initial guess comes from, into two regions: the \emph{cold region} and \emph{warm region}.
Roughly speaking, an approximation $U^{(n)}$ in the warm region means that $\N{\tan\Theta(U^{(n)}, U_* )}_{\F}$ is small while
it in the cold region means that $\N{\tan\Theta(U^{(n)}, U_* )}_{\F}$ is not that small.
$U_*$ sits at the ``\emph{center}'' of the warm region which is wrapped around by the cold region.
The proof is divided into two cases: the first case is when the initial guess is in the warm region
and the other one is when it is in the cold region. For the first case,
they proved that the algorithm will produce a sequence convergent to the principal subspace
(which is actually the most significant principal component because it is for $p=1$) with high probability. For the second case,
they first proved that the algorithm will produce a sequence of approximations that,
after a finite number of iterations, will fall into the warm region with high probability, and
then use the conclusion proved for the first case to conclude the proof due to 
the Markov property.

For our situation $p>1$, we still structure our proof in the same way, i.e., dividing the whole proof into
two cases: $U^{(0)}$ coming from the \emph{cold region} or \emph{warm region}.
The proof in \cite{liWLZ2017near} for the warm region case can be carried over with a little extra effort,
as we will see later,
but it was not possible for us to use a similar argument in \cite{liWLZ2017near} to get the job done
for the cold region case.
Three major difficulties
are as follows.
\begin{enumerate}
	\item In \cite{liWLZ2017near}, essentially $\N{\cot\Theta(U^{(n)}, U_* )}_{\F}$
	was used to track the behavior of a martingale along with the power iteration. Note
	$\cot\Theta(U^{(n)}, U_* )$ is $p\times p$. Thus it is a scalar when $p=1$, perfectly well-conditioned if
	treated as a matrix, but for $p>1$, it is a genuine matrix and, in fact, an inverse of a random matrix in the proof.
	The first difficulty is how to estimate the inverse because it may not even exist!
	
	\item We  tried to
	separate the flow of $U^{(n)}$  into two subflows: the ill-conditioned flow and the well-conditioned flow,
	and estimate the related quantities separately.
	Here the ill-conditioned flow at each step represents the subspace generated by the singular vectors  of $\cot\Theta(U^{(n)}, U_* )$
	whose corresponding singular values are tiny, while
	the well-conditioned flow at each step represents the subspace generated by the other singular vectors,
	of which the inverse (restricted to this subspace) is well conditioned.
	Unfortunately,
	tracking the two flows can be an impossible task because, due to the randomness,
	some elements in the ill-conditioned flow could jump to the well-conditioned flow during the iteration, and vice versa.
	
	\item The third one is to build a martingale to go along  with a proper power iteration, or equivalently,
	to find the Doob decomposition of the process,
	because the recursion formula of the main part of the inverse --- the drift in the Doob decomposition,  even if limited to the well-conditioned flow,
	is not a linear operator, which makes it impossible to build a proper power iteration.
\end{enumerate}

In the end, to deal with the cold region, we gave up the idea of estimating $\N{\cot\Theta(U^{(n)}, U_* )}_{\F}$.
Instead,
%
we invented another method:
cutting the cold region into many layers, each wrapped around by another  with
the innermost one around the warm region.
We prove the initial guess in any layer will produce a sequence of approximations that will fall into its inner neighbor layer
(or the warm region if the layer is innermost) in a finite number of iterations with high probability.
Therefore eventually, any initial guess in the cold region will lead to an approximation
in the warm region within a finite number of iterations with high probability, returning to the case of
initial guesses coming from the warm region because of the Markov property.
This enables us to completely avoid the difficulties mentioned above.
%
This technique works for $p=1$, too, and it can result in a simpler proof for the online PCA than that  in \cite{liWLZ2017near}. 

The other two main theorems of \LWLZ~\cite[Theorems 2 and~3]{liWLZ2017near} are stated as follows.

\begin{theorem}[{\cite[Theorem~2]{liWLZ2017near}}]\label{thm:thm2:ppr1-LWLZ}
	Under \emph{\Cref{asm:simplify:X}}  and $p=1$, suppose that $U^{(0)}$ is uniformly sampled from the unit sphere.
	Then for any $\varepsilon\in(0,1/8)$, stepsize $\beta>0,\delta>0$ satisfying
	\[
	d[\lambda_1^2\gamma^{-1}\beta]^{1-2\varepsilon}\le b_2\delta^2,
	\,\,
	4d\what N_2(\beta)\exp\left( -C_3[\lambda_1^2\gamma^{-1}\beta]^{-2\varepsilon} \right)\le \delta,
	\]
	there exists an event $\bbH_*$ with $\txs\prob{\bbH_*}\ge 1-2\delta$
	such that for any $n\in [\what N_2(\beta),\what N_3(\beta)]$
\begin{equation}\label{eq:main-result-thm2-LWLZ}
	\txs\E{\tan^2\Theta(U^{(n)},U_*) \overevent \bbH_*}
	\le C_4(1-\beta\gamma)^{2n}\delta^{-4}d^2 + C_4\beta\varphi(1,d;\Lambda)  + C_4\txs\sum\limits_{i=2}^{d}\tfrac{\lambda_1-\lambda_2}{\lambda_1-\lambda_i}[\lambda_1^2\gamma^{-1}\beta]^{3/2-4\varepsilon},
\end{equation}
	where $b_2,C_3,C_4$ are  absolute constants.
\end{theorem}

\begin{theorem}[{\cite[Theorem~3]{liWLZ2017near}}]\label{thm:thm3:ppr1-LWLZ}
	Under \emph{\Cref{asm:simplify:X}}  and $p=1$, suppose that $U^{(0)}$ is uniformly sampled from the unit sphere
	and let $\beta_*=\tfrac{2\ln N_*}{\gamma N}$.
	Then for any $\varepsilon\in(0,1/8),\, N_*\ge 1,\,\delta>0$ satisfying
	\[
	d[\lambda_1^2\gamma^{-1}\beta_*]^{1-2\varepsilon}\le b_3\delta^2,
	\,\,
	4d\what N_2(\beta_*)\exp\left( -C_6[\lambda_1^2\gamma^{-1}\beta_*]^{-2\varepsilon} \right)\le \delta,
	\]
	there exists an event $\bbH_*$ with $\txs\prob{\bbH_*}\ge 1-2\delta$
	such that
	\begin{equation}\label{eq:main-result-thm3-LWLZ}
	\txs\E{\tan^2\Theta(U^{(N_*)},U_*) \overevent \bbH_*}
	\le C_*(d,N_*,\delta)\tfrac{\varphi(1,d;\Lambda)}{\lambda_1-\lambda_2}\tfrac{\ln N_*}{N_*},
	\end{equation}
	where the constant $C_*(d,N_*,\delta) \to C_5$ as $d\to\infty,N_*\to\infty$,
	and $b_3,C_5,C_6$ are  absolute constants.
\end{theorem}

Our \Cref{thm:thm2:ppr1,thm:thm3:ppr1} when applied to the case $p=1$ do not exactly yield
\Cref{thm:thm2:ppr1-LWLZ,thm:thm3:ppr1-LWLZ}, respectively. But the resulting conditions and upper bounds
have the same orders in  constant parameters $d,\beta,\delta$, and the coefficients of $\beta$ and $\tfrac{\ln N_*}{N}$ in the upper bounds
are comparable.
Note that the first term in right-hand side of \cref{eq:main-result-thm2} is proportional to $d$, not $d^2$ as
in \cref{eq:main-result-thm2-LWLZ}; so ours is tighter for high-dimensional data.

Our proofs  for \Cref{thm:thm2:ppr1,thm:thm3:ppr1} are nearly the same as
those in \cite{liWLZ2017near} for \Cref{thm:thm2:ppr1-LWLZ,thm:thm3:ppr1-LWLZ} owing to the fact
that the difficult estimates have already been taken care of by either \Cref{thm:thm1:ppr1} or \Cref{thm:thm1:ppr1-LWLZ}.
But still there are some extras for $p>1$, namely, the need to
estimate the marginal probability for the uniform distribution on the Grassmann manifold of dimension higher than $1$.
We are not aware of anything like that in the literature, and thus have to build it ourselves with the help of
the theory of special functions of a matrix argument, rarely used in the statistical community.

It may also be worth pointing out that all absolute constants, except $C_p$ which has an explicit expression
in \cref{eq:Cp-dfn} and $C_{\psi}$, in our theorems will be explicitly  bounded as
in \cref{eq:C},
whereas those in \Cref{thm:thm1:ppr1-LWLZ,thm:thm2:ppr1-LWLZ,thm:thm3:ppr1-LWLZ} are not.

\section{Proof of Theorem \ref{thm:thm1:ppr1}}\label{sec:proof-of-cref-thm-thm1-ppr1}
We start by building a substantial amount of preparation material
in \cref{ssec:preparation,ssec:increments-in-one-iteration,,ssec:quasi-power-iteration-process} before
we prove the theorem in \cref{ssec:proof1st}.
In \cref{ssec:preparation}, we set the stage and introduce matrix $T^{(n)}$ to
serve the role of $\tan\Theta(U^{(n)},U_*)$ associated with the $n$th approximation. In particular we have
$\N{T^{(n)}}_{\UI}=\N{\tan\Theta(U^{(n)},U_*)}_{\UI}$.
In \cref{ssec:increments-in-one-iteration}, we present incremental estimates for one iterative step  of the subspace online PCA
in Lemma~\ref{lm:diff-V} and Lemma~\ref{lm:diff-T}. These estimates
allow us to associate one iterative step  with a quasi-power iterative step by
an operator $\opL$ defined
at the beginning of \cref{ssec:quasi-power-iteration-process}, and then further we relate $T^{(n)}$  to $\opL^n T^{(0)}$
by showing $T^{(n)}-\opL^n T^{(0)}$ is bounded with high probability in Lemma~\ref{lm:prop3:ppr1}.
This lemma is very critical to our proofs. It leads to Lemma~\ref{lm:prop5:ppr1}
which says that $\N{T^{(n)}}_2$ stagedly decreases
and Lemma~\ref{lm:prop4:ppr1} in which the expectation of $T^{(n)}$ is estimated.
Finally, we will be ready to prove \Cref{thm:thm1:ppr1} in \cref{ssec:proof1st}.
\Cref{fig:logical-dependences-in-the-proof-of-thm:name} shows a pictorial
description of our proving process.


\begin{figure}[ht]
	\centering\scriptsize
	\begin{tikzpicture}
	[
	scale = 0.8,
	]
	\node (qb) at (1,0) {Lemma~\ref{lm:quasi-bounded} };
	\node (diffV) at (1,-1) {Lemma~\ref{lm:diff-V} };
	\node (diffT) at (4,-1) {Lemma~\ref{lm:diff-T} };
	\node (prop3) at (7,-1) {Lemma~\ref{lm:prop3:ppr1} };
	\node (prop4) at (10,-1) {Lemma~\ref{lm:prop4:ppr1} };
	\node (prop5) at (10,-2) {Lemma~\ref{lm:prop5:ppr1} };
	\node (thm1) at (13,-2) {Proof of \Cref{thm:thm1:ppr1} };
	\graph[grow right sep] {
		{ (qb) -> (diffV), (qb) }
		-> (diffT) -> (prop3) -> { (prop5), (prop4) } -> (thm1);
		(qb) -> (prop4);
	};
	\node  at (2.5,-2) {\cref{ssec:increments-in-one-iteration} (one iterative step) };
	\node  at (8.5,0) {\cref{ssec:quasi-power-iteration-process} (quasi-power iteration) };
	\draw[dashed] (0,.5) rectangle (5.2,-1.5);
	\draw[dashed] (5.8,-.5) rectangle (11.2,-2.5);
	\end{tikzpicture}
	\caption{Proving process for \Cref{thm:thm1:ppr1}}
	\label{fig:logical-dependences-in-the-proof-of-thm:name}
\end{figure}

\subsection{Simplification}\label{ssec:preparation}
Without loss of generality, we may assume that the covariance matrix $\Sigma$ diagonal. Otherwise, we
can perform a (constant) orthogonal transformation as follows.
Recall the spectral decomposition $\Sigma=U\Lambda U^{\T}$ in \cref{eq:eigD-convar}. Instead of
the random vector $\bX$,  we  equivalently consider
\[
\bY\equiv [\bY_1,\bY_2,\ldots,\bY_n]^{\T}:=U^{\T}\bX.
\]
Accordingly, perform the same orthogonal transformation on all involved
quantities:
\begin{equation}\label{eq:YV-dfn}
Y^{(n)}=U^{\T}X^{(n)}, \quad V^{(n)}=U^{\T}U^{(n)}, \quad
V_*=U^{\T}U_*=\begin{bsmallmatrix} I_p \\ 0\end{bsmallmatrix}.
\end{equation}
As a consequence, we will have  equivalent versions of \Cref{alg:onlinePCA:ppr1}, \Cref{thm:thm1:ppr1,,thm:thm2:ppr1,thm:thm3:ppr1}.
Firstly, because
\[
(V^{(n-1)})^{\T}Y^{(n)}=(U^{(n-1)})^{\T}X^{(n)}=Z^{(n)}, \quad (Y^{(n)})^{\T}Y^{(n)}=(X^{(n)})^{\T}X^{(n)},
\]
the equivalent version of \Cref{alg:onlinePCA:ppr1} is obtained by symbolically
replacing
all letters $X,\, U$ by
$Y,\,V$ while keeping their respective superscripts.
If the algorithm converges, it is expected that $\cR(V^{(n)})\to\cR(V_*)$.
Secondly, noting
\[
\N{\Sigma^{-1/2}\bX}_{\psi_2} = \N{U\Lambda^{-1/2}U^{\T}\bX}_{\psi_2} = \N{\Lambda^{-1/2}\bY}_{\psi_2},
\]
we can restate \Cref{asm:simplify:X} equivalently as

\begin{enumerate}[label=(A-\arabic*$'$),ref=(A-\arabic*$'$),leftmargin=3\parindent]
	\item $\txs\E{\bY}=0,\txs\E{\bY\bY^{\T}}=\Lambda=\diag(\lambda_1,\dots,\lambda_d)$ with
	\cref{eq:eigv-assume}; 
	\item $\psi:=\N{\Lambda^{-1/2}\bY}_{\psi_2}<\infty$.
\end{enumerate}
Thirdly, all canonical angles between two subspaces are invariant under the orthogonal transformation.
Therefore the equivalent versions of
\Cref{thm:thm1:ppr1,,thm:thm2:ppr1,thm:thm3:ppr1}
for $\bY$ can be simply obtained by replacing all letters $X,\, U$ by
$Y,\,V$ while keeping their respective superscripts.


In what follows, we assume that $\Sigma$ is diagonal.
In the rest of this section, we will prove the mentioned equivalent version of \Cref{thm:thm1:ppr1}.
Likewise in the next section, we will prove the equivalent versions of
\Cref{thm:thm2:ppr1,thm:thm3:ppr1}.


To facilitate our proof, we introduce new notations for two particular submatrices of any $V\in\bbR^{d\times p}$:
\begin{equation}\label{eq:ol-ul}
\ol V=V_{(1:p,:)}, \quad \ul V=V_{(p+1:d,:)}.
\end{equation}
In particular, $\scrT(V)=\ul V\ol V^{-1}$ for the operator $\scrT$ defined in \cref{eq:scrT-dfn}, provided
$\ol V$ is nonsingular. Set
\begin{equation}\label{eq:ol-ul:Lambda}
\ol\Lambda=\diag(\lambda_1,\ldots,\lambda_p), \quad\ul\Lambda=\diag(\lambda_{p+1},\ldots,\lambda_d).
\end{equation}
Although the assignments to $\ol\Lambda$ and $\ul\Lambda$ are not consistent with the extractions defined by \cref{eq:ol-ul},
they don't seem to cause confusions in our later presentations.

For $\kappa>1$, define
$\sphere(\kappa ):=\txs\set{V\in\bbR^{d\times p}\setdelim \sigma(\ol V)\subset [\tfrac{1}{\kappa },1]}$,
where $\sigma(\ol V)$ is the set of the singular values of $\ol V$.
It can be verified that
\begin{equation}\label{eq:sphereK-normT}
V\in\sphere(\kappa )\Leftrightarrow \N{\scrT(V)}_2\le \sqrt{\kappa ^2-1}.
\end{equation}
For the sequence $V^{(n)}$, define
\[
\txs\Nout{\sphere(\kappa )}:=\min\txs\set{n\setdelim V^{(n)}\notin \sphere(\kappa )},\quad
\txs\Nin{\sphere(\kappa )}:=\min\txs\set{n\setdelim V^{(n)}\in \sphere(\kappa )}.
\]
$\txs\Nout{\sphere(\kappa )}$ is the first step of the iterative process at which $V^{(n)}$ jumps
from $\sphere(\kappa )$ to its outside, and
$\txs\Nin{\sphere(\kappa )}$ is the first step of the iterative process at which $V^{(n)}$ jumps
from the outside to $\sphere(\kappa )$.
Write
\[
\wtd\lambda_i:=\lambda_i\beta^{-2\varepsilon},\quad
\wtd\eta_i:=\wtd\lambda_1+\dots+\wtd\lambda_i=\eta_i\beta^{-2\varepsilon},\quad
\]
and define
\begin{equation}\label{eq:Nqb-dfn}
\txs\Nqb{\Lambda}:= \max\txs\set{n\ge1\setdelim \N{Z^{(n)}}_2\le \wtd\eta_p^{1/2},\abs{Y^{(n)}_i}\le \wtd\lambda_i^{1/2},i=1,\dots n}+1,
\end{equation}
where $Z^{(n)}=(U^{(n-1)})^{\T}X^{(n)}$ as is defined in \Cref{alg:onlinePCA:ppr1}.
$\txs\Nqb{\Lambda}$ is the first step of the iterative process at which either $\abs{Y^{(n)}_i}>\wtd\lambda_i^{1/2}$ for some $i$
or the norm of $Z^{(n)}$ exceeds $\wtd\eta_p^{1/2}$.
For $n<\txs\Nqb{\Lambda}$, we have
\[
\N{Y^{(n)}}_2\le \wtd\eta_d^{1/2}=\nu^{1/2}\wtd\eta_p^{1/2},
\quad
\N{Z^{(n)}}_2\le \wtd\eta_p^{1/2}
\quad\mbox{with}\quad
\nu=1/\mu_p.
\]

For convenience, we introduce
$T^{(n)}=\scrT(V^{(n)})$,
and let $\fil_n=\sigma\txs\set{Y^{(1)},\dots,Y^{(n)}}$ be the $\sigma$-algebra filtration, i.e., the information known by step $n$.
Also, since in this section $\varepsilon,\,\beta$ are fixed, we suppress the dependency information of $M(\varepsilon)$
on $\varepsilon$
and $N_s(\beta)$ on $\beta$ to simply write $M$ for $M(\varepsilon)$ and $N_s$ for $N_s(\beta)$.

Lastly, we discuss some of the important implications of the conditions:
\begin{gather}
\tag{\ref{eq:beta-bound}}
0< \beta < \min\txs\set*{1,
	\left( \tfrac{1}{8\kappa \eta_p } \right)^{\tfrac{2}{1-4\varepsilon}},
	\left(\tfrac{\gamma}{130\kappa ^2\eta_p ^2}\right)^{\tfrac{1}{\varepsilon}}
}
,
\\	
\tag{\ref{eq:beta-d-bound}}
(\sqrt{2}+1)\lambda_1d\beta^{1-7\varepsilon}\le \omega,
\quad
K>N_{3/2-37\varepsilon/4}(\beta)
\end{gather}
of \Cref{thm:thm1:ppr1}.
They guarantee that
\begin{enumerate}[label=($\beta$-\arabic*),ref=($\beta$-\arabic*),leftmargin=3\parindent]
	\item $\beta<1$; 
	\label{itm:beta}
	\item $\beta\gamma\le \beta\wtd\eta_p\le\nu \beta\wtd\eta_p=\beta\wtd\eta_d\le d\beta\wtd\lambda_1= d\lambda_1\beta^{1-2\varepsilon}\le (\sqrt{2}-1)\omega\le \sqrt{2}-1$.\label{itm:beta-p}
\end{enumerate}
Set
\begin{subequations}\label{eq:C}
	\begin{align}
	C_V &=\tfrac{5}{2}+\tfrac{7}{2}(\nu \wtd\eta_p \beta)+\tfrac{15}{8}(\nu \wtd\eta_p\beta)^2+\tfrac{3}{8}(\nu \wtd\eta_p\beta)^3\le\tfrac{16+13\sqrt{2}}{8}\approx 4.298;\label{eq:C_V}\\
	C_{\Delta} &= 2+\tfrac{1}{2}(\nu \wtd\eta_p\beta)+C_V \wtd\eta_p\beta\le\tfrac{22+7\sqrt{2}}{8}\approx 3.987;\label{eq:C_VDelta}\\
	C_T &= C_V+2C_{\Delta}+2C_{\Delta}C_V \wtd\eta_p\beta \le \tfrac{251+122\sqrt{2}}{16}\approx 26.471; \label{eq:C_T}\\
	C_{\kappa} &=\tfrac{(3-\sqrt{2})C_{\Delta}^2}{64(C_T+2C_{\Delta} )^2} \le \tfrac{565+171\sqrt{2}}{21504}\approx 0.038; \label{eq:C_kappa}\\
	C_{\nu} &=4\sqrt{2}C_TC_{\kappa} \le \tfrac{223702+183539\sqrt{2}}{86016}\approx 5.618; \label{eq:C_nu}\\
	C_{\circ} &=\tfrac{29+8\sqrt{2}}{16(3-\sqrt{2})}
	+\tfrac{4C_T}{(3-\sqrt{2})C_{\Delta}}\beta^{3\varepsilon}
	+[C_T+\tfrac{29+8\sqrt{2}}{32}]\beta^{1/2-3\varepsilon} 
	+\tfrac{3C_T^2}{2(3-\sqrt{2})C_{\Delta}^2}\beta^{1/2+3\varepsilon}
	+\tfrac{2C_T}{C_{\Delta}}\beta^{1-3\varepsilon}+	\tfrac{C_T^2}{2C_{\Delta}^2}\beta^{3/2-3\varepsilon} \nonumber \\
	&\le \tfrac{2582968+1645155\sqrt{2}}{14336}
	\approx 342.464.	 \label{eq:C_circE}
	\end{align}
\end{subequations}
The condition \cref{eq:beta-bound} also guarantees that
\begin{enumerate}[label=($\beta$-\arabic*),ref=($\beta$-\arabic*),leftmargin=3\parindent,resume]
	\item $2C_{\Delta} \wtd\eta_p \beta^{1/2}\kappa =2C_{\Delta} \eta_p \beta^{1/2-2\varepsilon}\kappa \le \tfrac{2C_{\Delta}}{8}<1$,
	and thus $2C_{\Delta} \wtd\eta_p \beta\kappa<1$; \label{itm:beta-C_Vdelta}
	\item $4\sqrt{2}C_T\kappa ^2\wtd\eta_p ^2\gamma^{-1}\beta^{5\varepsilon}\le1$, and thus $4\sqrt{2}C_T\kappa ^2\wtd\eta_p ^2\gamma^{-1}\beta^{1/2+\chi}\le1$ for $\chi\in[-1/2+5\varepsilon,0]$. \label{itm:beta-C_T}
\end{enumerate}

\subsection{Increments by One Iterative Step}\label{ssec:increments-in-one-iteration}

\begin{lemma}\label{lm:quasi-bounded}
	For any fixed integer $K\ge 1$,
	\[
	\txs\prob{\txs\Nqb{\Lambda}>K}\ge 1-
	K(\ee d+p+1)\exp\left(-C_{\psi}\min\txs\set{\psi^{-1},\psi^{-2}}\beta^{-2\varepsilon}\right),
	\]
	where $C_\psi$ is an absolute constant.
\end{lemma}

\begin{proof}
	Since
	$
	\txs\set*{\txs\Nqb{\Lambda}\le K}
	\subset \txs\bigcup\limits_{n\le K}\big(\txs\set*{\N{Z^{(n)}}_2\ge \wtd\eta_p^{1/2}}\cup\txs\bigcup\limits_{i=1}^d\txs\set*{\abs{e_i^{\T}Y^{(n)}}\ge \wtd\lambda_i^{1/2}} \big)
	,
	$
	we know
	\begin{equation}\label{eq:add2pf-1}
	\txs\prob{\txs\Nqb{\Lambda}\le K}
	\le \txs\sum\limits_{n\le K}\left(\txs\prob{\N{Z^{(n)}}_2\ge \wtd\eta_p^{1/2}}+\txs\sum\limits_{1\le i\le d}\txs\prob{\abs{e_i^{\T}Y^{(n)}}\ge \wtd\lambda_i^{1/2}}\right)
	.
	\end{equation}
	First,
	\begin{align}
	\txs\prob{\abs{e_i^{\T}Y^{(n)}}\ge \wtd\lambda_i^{1/2}}
	&= \txs\prob{\abs*{\tfrac{(\Lambda^{1/2}e_i)^{\T}}{\N{\Lambda^{1/2} e_i}_2}\Lambda^{-1/2}Y^{(n)}}\ge \tfrac{\wtd\lambda_i^{1/2}}{\N{\Lambda^{1/2} e_i}_2}}
	\nonumber \\
	&\le \exp\left(1-\tfrac{C_{\psi,i}\tfrac{\wtd\lambda_i}{e_i^{\T}\Lambda e_i}}{\N{\tfrac{(\Lambda^{1/2}e_i)^{\T}}{\N{\Lambda^{1/2} e_i}_2}\Lambda^{-1/2}Y^{(n)}}_{\psi_2}}\right)
	\qquad\qquad
	\text{(by \cite[(5.10)]{vershynin2012introduction})} \nonumber 	\\
	&\le \exp\left(1-\tfrac{C_{\psi,i}\wtd\lambda_i}{\N{\Lambda^{-1/2}Y^{(n)}}_{\psi_2}\lambda_i}\right)
	=\exp\left(1-C_{\psi,i}\psi^{-1}\beta^{-2\varepsilon}\right)		, \label{eq:add2pf-2}
	\end{align}
	where $C_{\psi,i},i=1,\dots,d$ are absolute constants \cite[(5.10)]{vershynin2012introduction}.
	Next, we claim 
	\begin{equation}\label{eq:norm-Z-claim}
	\txs\prob{\N{Z^{(n)}}_2\ge \wtd\eta_p^{1/2}}\le
	(p+1)\exp\left(-C_{\psi,d+1}\psi^{-2}\beta^{-2\varepsilon}\right)
	\end{equation}
	to be proven in the next paragraph. 
	Together, \cref{eq:add2pf-1} -- \cref{eq:norm-Z-claim} yield
	\begin{align*}
	\txs\prob{\txs\Nqb{\Lambda}\le K}
	&=\txs\sum\limits_{n\le K}\txs\sum\limits_{1\le i\le d} \exp\left(1-C_{\psi,i}\psi^{-1}\beta^{-2\varepsilon}\right)
      +\txs\sum\limits_{n\le K} (p+1)\exp\left(-C_{\psi,d+1}\psi^{-2}\beta^{-2\varepsilon}\right)
	\\ &\le K(\ee d+p+1)\exp\left(-C_{\psi}\min\txs\set{\psi^{-1},\psi^{-2}}\beta^{-2\varepsilon}\right),
	\end{align*}
	where $C_\psi=\min\limits_{1\le i\le d+1}C_{\psi,i}$.
	Finally, use $\txs\prob{\txs\Nqb{\Lambda}> K}=1-\txs\prob{\txs\Nqb{\Lambda}\le K}$ to complete the proof.

	It remains to prove \cref{eq:norm-Z-claim}.
	To avoid the cluttered superscripts, we drop the superscript ``$\cdot^{(n-1)}$'' on $V$, and drop the superscript ``$\cdot^{(n)}$'' on $Y,\,Z$.
	Consider
	\[
	W:=\begin{bsmallmatrix}
	0 & Z\\ Z^{\T} & 0
	\end{bsmallmatrix}
	= \begin{bsmallmatrix}
	& V^{\T}Y\\ Y^{\T}V &
	\end{bsmallmatrix}
	=\txs\sum\limits_{k=1}^{d} Y_k\begin{bsmallmatrix}
	& & & v_{k1} \\
	& & & \vdots \\
	& & & v_{kp} \\
	v_{k1} & \cdots & v_{kp} & 0\\
	\end{bsmallmatrix}
	=: \txs\sum\limits_{k=1}^{d} Y_kW_k,
	\]
	where $v_{ij}$ is the $(i,j)$th entry of $V$ and $Y_k$ is the $k$th entry of $Y$.
	By the matrix version of master tail bound \cite[Theorem~3.6]{tropp2012user}, for any $\alpha>0$,
we have
	\[
	\txs\prob{\N{Z}_2\ge \alpha}
	= \txs\prob{\lambda_{\max}(W)\ge \alpha}
	\le \inf\limits_{\theta>0}\ee^{-\theta \alpha}\trace\exp\left(\txs\sum\limits_{k=1}^{d}\ln \txs\E{\exp(\theta Y_kW_k)} \right).
	\]
	$Y$ is sub-Gaussian and $\txs\E{Y}=0$, and so is $Y_k$.
	Moreover,
	\[
	\N{Y_k}_{\psi_2}=\N{e_k^{\T}\Lambda^{1/2}}_2\N*{\tfrac{e_k^{\T}\Lambda^{1/2}}{\N{e_k^{\T}\Lambda^{1/2}}_2}\Lambda^{-1/2}Y}_{\psi_2}
	\le \lambda_k^{1/2}\N{\Lambda^{-1/2}Y}_{\psi_2}
	=\lambda_k^{1/2}\psi.
	\]
	Also, by \cite[(5.12)]{vershynin2012introduction},
	\[
	\txs\E{\exp(\theta W_kY_k)}\le \exp(C_{\psi,d+k}\theta^2W_k\circ W_k\N{Y_k}_{\psi_2}^2)
	\le \exp(c_{\psi,k}\theta^2\lambda_k\psi^2W_k\circ W_k)
	,
	\]
	where $c_{\psi,k}, k=1,\dots,d$ are absolute constants.
	Therefore, writing $[4C_{\psi,d+1}]^{-1}=\max\limits_{1\le k\le d}c_{\psi,k}$
	and $W_\psi:=\txs\sum\limits_{k=1}^{d} \lambda_kW_k\circ W_k$ with the spectral decomposition $W_\psi=V_\psi\Lambda_\psi V_\psi^{\T}$, we have
	\begin{align*}
	\trace\exp\left(\txs\sum\limits_{k=1}^{d}\ln \txs\E{\exp(\theta Y_kW_k)} \right)
	&
	\le \trace\exp\left(\txs\sum\limits_{k=1}^{d} c_{\psi,k}\theta^2\lambda_k\psi^2W_k\circ W_k\right)
	\\&\le \trace\exp([4C_{\psi,d+1}]^{-1}\theta^2\psi^2W_\psi)
	\\&= \trace\exp([4C_{\psi,d+1}]^{-1}\theta^2\psi^2V_\psi\Lambda_\psi V_\psi^{\T})
	\\&= \trace\left( V_\psi \exp([4C_{\psi,d+1}]^{-1}\theta^2\psi^2\Lambda_\psi) V_\psi^{\T}\right)
	\\&= \trace\exp([4C_{\psi,d+1}]^{-1}\theta^2\psi^2\Lambda_\psi)
	\\&\le (p+1)\exp([4C_{\psi,d+1}]^{-1}\theta^2\psi^2\lambda_{\max}(\Lambda_\psi))
	\\&= (p+1)\exp([4C_{\psi,d+1}]^{-1}\theta^2\psi^2\lambda_{\max}(W_\psi))
	.
	\end{align*}
	Note that
	\[
	W_\psi
	=
	\begin{bsmallmatrix}
	0&\cdots&0&\txs\sum\limits_{k=1}^{d} \lambda_kv_{k1}^2 \\
	\vdots&&\vdots&\vdots \\
	0&\cdots&0&\txs\sum\limits_{k=1}^{d} \lambda_kv_{kp}^2 \\
	\txs\sum\limits_{k=1}^{d} \lambda_kv_{k1}^2 & \cdots &\txs\sum\limits_{k=1}^{d} \lambda_kv_{kp}^2 & 0 \\
	\end{bsmallmatrix}
	=
	\begin{bsmallmatrix}
	0&\cdots&0&e_1^{\T}V^{\T}\Lambda Ve_1 \\
	\vdots&&\vdots&\vdots \\
	0&\cdots&0&e_p^{\T}V^{\T}\Lambda Ve_p \\
	e_1^{\T}V^{\T}\Lambda Ve_1 &\cdots & e_p^{\T}V^{\T}\Lambda Ve_p &0 \\
	\end{bsmallmatrix},
	\]
	and thus
	\[
		\lambda_{\max}(W_\psi)=\N*{\begin{bsmallmatrix}
				e_1^{\T}V^{\T}\Lambda Ve_1 \\
				\vdots \\
				e_p^{\T}V^{\T}\Lambda Ve_p \\
				\end{bsmallmatrix}}_2
			\le \txs\sum\limits_{k=1}^{p} e_k^{\T}V^{\T}\Lambda Ve_k
			= \trace(V^{\T}\Lambda V) 	
		    \le\txs\sum\limits_{k=1}^{p}\lambda_k
			=\eta_p.
	\]
	In summary, we have
	\[
	\txs\prob{\N{Z}_2\ge \alpha}
	\le (p+1)\inf\limits_{\theta>0}\exp([4C_{\psi,d+1}]^{-1}\theta^2\psi^2\eta_p-\theta \alpha)
	= (p+1)\exp\left(-\tfrac{C_{\psi,d+1}\alpha^2}{\psi^2\eta_p}\right).
	\]
	Substituting $\alpha=\wtd\eta_p^{1/2}$, we have the claim \cref{eq:norm-Z-claim}.
\end{proof}
\begin{lemma}\label{lm:diff-V}
	Suppose  the conditions of \emph{\Cref{thm:thm1:ppr1}} hold.
	If $n<\txs\Nqb{\Lambda}$, then
\begin{equation}\label{eq:diff-V}
	V^{(n+1)}			
	= V^{(n)} + \beta Y^{(n+1)}(Z^{(n+1)})^{\T}
	-\beta\big[1+\tfrac{\beta}{2}(Y^{(n+1)})^{\T}Y^{(n+1)}\big] V^{(n)}Z^{(n+1)}(Z^{(n+1)})^{\T}
	+R^{(n)}(Z^{(n+1)})^{\T},
\end{equation}
	where $R^{(n)}\in \bbR^d$ is a random vector with $\N{R^{(n)}}_2\le C_V \nu^{1/2}\wtd\eta_p^{3/2}\beta^2$ and $C_V$ is as in \cref{eq:C_V}.
\end{lemma}

\begin{proof}
	To avoid the cluttered superscripts, in this proof, we drop 
	``$\cdot^{(n)}$''
	and use 
	``$\cdot\new$'' to replace ``$\cdot^{(n+1)}$'' on $V$, and drop 
	``$\cdot^{(n+1)}$'' on $Y,\,Z$.
	
	On the set $\txs\set*{\txs\Nqb{\Lambda}>n}$,
	by \cref{eq:beta-d-bound} and \cref{itm:beta-p},
	we have
	\[
	\alpha=\beta(2+\beta Y^{\T}Y)Z^{\T}Z
	\le \beta(2+\nu\wtd\eta_p\beta)\wtd\eta_p
	\le (2+\sqrt{2}-1)(\sqrt{2}-1)/\nu < 1.
	\]
	By Taylor's expansion, there exists $0<\xi<\alpha$ such that
	\[
	(1+\alpha)^{-1/2}
	=1-\tfrac{1}{2}\alpha+\tfrac{3}{8}\tfrac{1}{(1+\xi)^{5/2}}\alpha^2
	=1-\beta Z^{\T}Z-\tfrac{\beta^2}{2}Y^{\T}YZ^{\T}Z+ \beta^2(Z^{\T}Z)^2\zeta 
	,
	\]
	where
	$ 
	\zeta
	=\tfrac{3}{8}\tfrac{1}{(1+\xi)^{5/2}}(2+\beta Y^{\T}Y)^2
	\le \tfrac{3}{8}(2+\nu \beta \wtd\eta_p)^2
	.
	$ 
	Thus
	\begin{align*}
	V\new
	&= (V + \beta YZ^{\T})\left(I-\left[\beta Z^{\T}Z+\tfrac{\beta^2}{2}Y^{\T}YZ^{\T}Z-\beta^2(Z^{\T}Z)^2\zeta\right]\tfrac{ZZ^{\T}}{Z^{\T}Z}\right)
	\\ &=  V + \beta YZ^{\T} -\beta VZZ^{\T}-\tfrac{\beta^2}{2}(Y^{\T}Y)VZZ^{\T} +RZ^{\T},
	\end{align*}
	where
	$
	R= -\tfrac{\beta^2}{2}(Z^{\T}Z)(2+\beta Y^{\T}Y)Y +\zeta\beta^2(Z^{\T}Z) VZ+\zeta\beta^3(Z^{\T}Z)^2Y
	$
	for which
	\begin{align*}
	\N{R}_2
	&\le \tfrac{\beta^2}{2}\wtd\eta_p(2+\beta\nu \wtd\eta_p )(\nu \wtd\eta_p)^{1/2}
            +\zeta\beta^2\wtd\eta_p^{3/2}+\zeta\beta^3\wtd\eta_p^2(\nu \wtd\eta_p)^{1/2}	\\
    &= \left[ \tfrac{1}{2}(2+\beta\nu \wtd\eta_p) +\tfrac{3}{8}(2+\beta\nu \wtd\eta_p)^2
            +\tfrac{3}{8}(2+\beta\nu \wtd\eta_p)^2(\beta \wtd\eta_p) \right]\nu^{1/2}\wtd\eta_p^{3/2}\beta^2 \\
    &=C_V\nu^{1/2}\wtd\eta_p^{3/2}\beta^2,
	\end{align*}
	as expected.
\end{proof}

\begin{lemma}\label{lm:diff-T}
	Suppose  the conditions of \emph{\Cref{thm:thm1:ppr1}} hold. Let
	$\tau=\N{T^{(n)}}_2$, and $C_T$ be as in \cref{eq:C_T}.
	If $n<\min\txs\set*{\txs\Nqb{\Lambda}, \txs\Nout{\sphere(\kappa)}}$, then 
	\begin{enumerate}[{\rm (1)}]
		\item \label{itm:lm:diff-T:well-defined}
		$T^{(n)}$ and $T^{(n+1)}$ are well-defined.
		\item \label{itm:lm:diff-T:dT}
		Define $E_T^{(n)}(V^{(n)}):=\txs\E{T^{(n+1)}-T^{(n)}\given \fil_n}-\beta (\ul\Lambda T^{(n)}-T^{(n)}\ol\Lambda )$. Then
		\begin{enumerate}[{\rm (a)}]
			\item \label{itm:lm:diff-T:normE_T}
			$\sup\limits_{V\in\sphere(\kappa )}\N{E_T^{(n)}(V)}_2\le C_T \nu^{1/2}(\wtd\eta_p \beta)^2(1+\tau^2)^{3/2}$;
			\item \label{itm:lm:diff-T:normdT}
			$\N{T^{(n+1)}-T^{(n)}}_2\le \nu^{1/2}(\wtd\eta_p \beta)(1+\tau^2)+C_T\nu^{1/2}(\wtd\eta_p \beta)^2(1+\tau^2)^{3/2}$.
		\end{enumerate}
		\item \label{itm:lm:diff-T:estimate}
		Define $R_{\circ}:=\varc{T^{(n+1)}-T^{(n)}\given\fil_n}-\beta^2H_{\circ}$. Then
		\begin{enumerate}[{\rm (a)}]
			\item \label{itm:lm:diff-T:varcdT}
			$H_{\circ}=\varc{\ul Y\ol Y^{\T}}\le 16\psi^4 H$, where $H=[\eta_{ij}]_{(d-p)\times p}$ with $\eta_{ij}=\lambda_{p+i}\lambda_j$ for $i=1,\dots,d-p$, $j=1,\dots,p$;
			\item \label{itm:lm:diff-T:normR_H}
			$\N{R_{\circ}}_2\le (\nu \wtd\eta_p\beta)^2\tau(1+\tfrac{11}{2}\tau+\tau^2+\tfrac{1}{4}\tau^3)+4C_T\nu(\wtd\eta_p\beta)^3(1+\tau^2)^{5/2}+2C_T^2\nu(\wtd\eta_p\beta)^4(1+\tau^2)^3$.
		\end{enumerate}
	\end{enumerate}
\end{lemma}

\begin{proof}
	For readability,
	we will drop 
	``$\cdot^{(n)}$'', and use 
	``$\cdot\new$'' to replace ``$\cdot^{(n+1)}$'' for $V,R$, drop 
	``$\cdot^{(n+1)}$'' on $Y,\,Z$,
	and drop the conditional sign ``$\given\fil_n$'' in the computation of $\txs\E{\cdot},\var(\cdot),\cov(\cdot)$
	with the understanding that they are conditional with respect to $\fil_n$.
	Finally, for any expression or variable $F$, we define $\Delta F:=F\new-F$.
	
	Consider \cref{itm:lm:diff-T:well-defined}.
	Since $n<\txs\Nout{\sphere(\kappa )}$, we have $V\in\sphere(\kappa )$ and $\tau=\N{T}_2\le(\kappa ^2-1)^{1/2}$.
	Thus, $\N{\ol V^{-1}}_2\le \kappa $ and $T=\ul V\ol V^{-1}$ is well-defined. Recall \cref{eq:diff-V} and the partitioning
	$
	Y=\kbordermatrix{ &\scs 1\\
		\scs p   & \ol Y \\
		\scs d-p & \ul Y}, \quad
	R=\kbordermatrix{ &\scs 1\\
		\scs p   & \ol R \\
		\scs d-p & \ul R}.
	$
	We have
	$\Delta \ol V=\beta(\ol YZ^{\T}-(1+\tfrac{\beta}{2}Y^{\T}Y)\ol VZZ^{\T})+\ol RZ^{\T}$,
	and
	$
	\ol R=-\tfrac{\beta^2}{2}(Z^{\T}Z)(2+\beta Y^{\T}Y)\ol Y +\zeta\beta^2(Z^{\T}Z)\ol VZ+\zeta\beta^3(Z^{\T}Z)^2\ol Y
	.
	$
	Noticing $\N{\ol Y}_2\le \wtd\eta_p^{1/2}$,
	we find
	\[
	\N{\Delta\ol V}_2
	\le \beta \wtd\eta_p+\beta(1+\tfrac{\beta}{2}\nu \wtd\eta_p)\wtd\eta_p+	C_V\wtd\eta_p^2\beta^2
	\le \left[2+ \tfrac{\beta}{2}\nu \wtd\eta_p+C_V \wtd\eta_p\beta \right]\wtd\eta_p\beta
	=C_{\Delta}\wtd\eta_p\beta,
	\]
	where $C_{\Delta}$ is as in \cref{eq:C_VDelta}.
	Thus $\N{\Delta \ol V \ol V^{-1}}_2\le \N{\Delta \ol V}_2\N{\ol V^{-1}}_2\le C_{\Delta}\wtd\eta_p \beta \kappa \le1/2$ by \cref{itm:beta-C_Vdelta}.
	As a result, $\ol V\new$ is nonsingular, and
	$
	\N{(\ol V\new)^{-1}}_2\le\tfrac{\N{\ol V^{-1}}_2}{1-\N{\ol V^{-1}\Delta \ol V}_2}
	\le 2\N{\ol V^{-1}}_2.
	$
	In particular, $T\new=\ul V\new(\ol V\new)^{-1}$ is well-defined. This proves \cref{itm:lm:diff-T:well-defined}.
	
	For \cref{itm:lm:diff-T:dT}, using the Sherman-Morrison-Woodbury formula \cite[p.~95]{demm:1997}, we get
	\begin{align*}
	\Delta T 
	&=(\ul V+\Delta \ul V)(\ol V+\Delta \ol V)^{-1}-\ul V\ol V^{-1}
	\\ &=(\ul V+\Delta \ul V)(\ol V^{-1}-\ol V^{-1}\Delta \ol V(\ol V+\Delta \ol V)^{-1})-\ul V\ol V^{-1}
	\\ &=\Delta \ul V\ol V^{-1}-\ul V\ol V^{-1}\Delta \ol V(\ol V+\Delta \ol V)^{-1}-\Delta \ul V\ol V^{-1}\Delta \ol V(\ol V+\Delta \ol V)^{-1}
	\\ &=\Delta \ul V\ol V^{-1}-\ul V\ol V^{-1}\Delta \ol V(\ol V^{-1}-\ol V^{-1}\Delta \ol V(\ol V+\Delta \ol V)^{-1})-\Delta \ul V\ol V^{-1}\Delta \ol V(\ol V+\Delta \ol V)^{-1}
	\\ &=\Delta \ul V\ol V^{-1}-T\Delta \ol V\ol V^{-1}+T\Delta \ol V\ol V^{-1}\Delta \ol V(\ol V\new)^{-1}-\Delta \ul V\ol V^{-1}\Delta \ol V(\ol V\new)^{-1}
	\\ &=[\Delta \ul V-T\Delta \ol V][I-(\ol V\new)^{-1}\Delta \ol V]\ol V^{-1}
	.
	\end{align*}
	Write
	$ T_L=\begin{bsmallmatrix} -T & I\end{bsmallmatrix}$ and $T_R=\begin{bsmallmatrix} I \\ T \end{bsmallmatrix}$.
	Then $T_LV=0$ and $V=T_R\ol V$.
	Thus,
	\[
	\Delta T= T_L\Delta V [I-(\ol V\new)^{-1}\Delta \ol V]V^{\T}T_R.
	\]
	Since $\Delta V$ is rank-$1$, $\Delta T$ is also rank-$1$.
	By Lemma~\ref{lm:diff-V},
		\[
	\begin{aligned}
	\Delta T
	&=  T_L\left[ \beta YZ^{\T}-\beta(1+\tfrac{\beta}{2}Y^{\T}Y) VZZ^{\T}+ RZ^{\T} \right]
	\left[I-(\ol V\new)^{-1}\Delta \ol V\right]V^{\T}T_R
	\\ &= T_L\left[ \beta YY^{\T}V+ RZ^{\T} \right]\left[I-(\ol V\new)^{-1}\Delta \ol V\right]V^{\T}T_R
	\\ &=T_L(\beta YY^{\T}VV^{\T}+R_T)T_R
	\\ &=T_L(\beta YY^{\T}+R_T)T_R
	,
	\end{aligned}
\]
	where
	$ 
	R_T = RZ^{\T}V^{\T}
	-(\beta Y+R)Z^{\T}(\ol V\new)^{-1}\Delta \ol VV^{\T}.
	$ 
	Note that 
	\begin{equation}\label{eq:dT:mainpart}
	T_LYY^{\T}T_R=\ul Y\ol Y^{\T}-T\ol Y\ul Y^{\T}T-T\ol Y\ol Y^{\T}+\ul Y\ul Y^{\T}T,
	\end{equation}
	and	
	\begin{subequations}\label{eq:dT:mainpart:E}
		\begin{alignat}{2}
		\txs\E{\ul Y\ol Y^{\T}}&=0,
		&\quad&\txs\E{T\ol Y\ol Y^{\T}}=T\txs\E{\ol Y\ol Y^{\T}}=T\ol\Lambda , 	\label{eq:dT:mainpart:E-1}	\\
		\txs\E{T\ol Y\ul Y^{\T}T}=T\txs\E{\ol Y\ul Y^{\T}}T&=0,
		&\quad&\txs\E{\ul Y\ul Y^{\T}T}=\txs\E{\ul Y\ul Y^{\T}}T=\ul\Lambda T. \label{eq:dT:mainpart:E-2}
		\end{alignat}
	\end{subequations}
	Thus,
	$\txs\E{\Delta T}
	=\beta (\ul\Lambda T-T\ol\Lambda )+E_T(V)$,
	where $E_T(V)=\txs\E{T_LR_TT_R}$.
	
	Since $V\in\sphere(\kappa )$,
	$\N{T}_2\le(\kappa ^2-1)^{1/2}$ by \cref{eq:sphereK-normT}.
	Thus
	\[
	\begin{aligned}
	\N{R_T}_2&\le \N{R}_2\wtd\eta_p^{1/2}+[(\nu \wtd\eta_p)^{1/2} \beta+\N{R}_2]\wtd\eta_p^{1/2} 2(1+\N{T}_2^2)^{1/2}C_{\Delta}\wtd\eta_p \beta
	\\ &\le C_V\nu^{1/2}\wtd\eta_p^2\beta^2+(1+\N{T}_2^2)^{1/2}[1+C_V \wtd\eta_p\beta]2C_{\Delta}\nu^{1/2}\wtd\eta_p^2\beta^2
	\\ &\le C_T\nu^{1/2}(\wtd\eta_p \beta)^2(1+\N{T}_2^2)^{1/2},
	\end{aligned}
	\]
	where $C_T=C_V+2C_{\Delta}(1+C_V\wtd\eta_p\beta)$.
	Therefore,
	$
	\N{E_T(V)}_2\le\txs\E{\N{T_LR_TT_R}_2}\le(1+\N{T}_2^2)\txs\E{\N{R_T}_2}.
	$
	\Cref{itm:lm:diff-T:normE_T} holds.
	For \cref{itm:lm:diff-T:normdT}, we have
	\begin{align*}
	\N{\Delta T}_2
	&\le (1+\N{T}_2^2)(\beta\N{YY^{\T}VV^{\T}}_2+\N{R_T}_2)
	\\ &\le  \beta(\nu \wtd\eta_p)^{1/2}\wtd\eta_p^{1/2}(1+\N{T}_2^2)+C_T\nu^{1/2}(\wtd\eta_p \beta)^2(1+\N{T}_2^2)^{3/2}
	\\ &\le  \nu^{1/2}\wtd\eta_p\beta(1+\N{T}_2^2)+C_T\nu^{1/2}(\wtd\eta_p \beta)^2(1+\N{T}_2^2)^{3/2}
	.
	\end{align*}
	
	The proof of \cref{itm:lm:diff-T:estimate} is similar to that of \cref{itm:lm:diff-T:dT} but involves more complicated calculations,
	and it is deferred to \cref{ssec:proof-of-item-itm-of-lemma-lm}.
\end{proof}

\subsection{Quasi-Power Iterative Process}\label{ssec:quasi-power-iteration-process}
Let $D^{(n+1)}=T^{(n+1)}-\txs\E{T^{(n+1)}\given\fil_n}$.
We have
\begin{gather*}
T^{(n)}-\txs\E{T^{(n)}\given\fil_n}=0,\quad
\txs\E{D^{(n+1)}\given\fil_n}=0,  \\
\txs\E{D^{(n+1)}\circ D^{(n+1)}\given\fil_n}=\varc{T^{(n+1)}-T^{(n)}\given\fil_n}.
\end{gather*}
By \cref{itm:lm:diff-T:dT} of Lemma~\ref{lm:diff-T}, we have
\begin{align*}
T^{(n+1)}
&=  D^{(n+1)} +T^{(n)}+\txs\E{T^{(n+1)}-T^{(n)}\given\fil_n}
\\ &= D^{(n+1)} +T^{(n)}+\beta (\ul\Lambda T^{(n)}-T^{(n)}\ol\Lambda )+ E_T^{(n)}(V^{(n)})
\\ &= \opL T^{(n)}+D^{(n+1)} + E_T^{(n)}(V^{(n)}),
\end{align*}
where $\opL\opdelim T\mapsto T+\beta\ul\Lambda T-\beta T\ol\Lambda $ is a bounded linear operator.
It can be verified that $\opL T=L\circ T$, the Hadamard product of $L$ and $T$, where
$L=[\lambda_{ij}]_{(d-p)\times p}$ with $\lambda_{ij}=1+\beta \lambda_{p+i}-\beta\lambda_j$.
Moreover, it can be shown that\footnote{
	Since  $\lambda(\opL)=\txs\set{\lambda_{ij}\setdelim i=1,\dots,d-p,\,j=1,\dots,p}$,
	we have the spectral radius $\rho(\opL)=1-\beta(\lambda_p-\lambda_{p+1})$.
	Thus for any $T$,
	\[
		\N{\opL T}_{\UI}=\N{T(I-\beta\ol\Lambda )+\beta\ul\Lambda T}_{\UI}
			\le\N{I-\beta\ol\Lambda }_2\N{T}_{\UI}+\N{\beta\ul\Lambda }_2\N{T}_{\UI}
			=(1-\beta\lambda_p+\beta\lambda_{p+1})\N{T}_{\UI}
			=\rho(\opL)\N{T}_{\UI}		,
	\]
	which means $\N{\opL}_{\UI}\le\rho(\opL)$. This ensures $\N{\opL}_{\UI}=\rho(\opL)$.
}
$\N{\opL}_{\UI}=\rho(\opL)=1-\beta\gamma$, where $\N{\opL}_{\UI}=\sup\limits_{\N{T}_{\UI}=1}\N{\opL T}_{\UI}$ is
an operator norm induced by the matrix norm $\N{\cdot}_{\UI}$.
Recursively, 
\begin{equation}\label{eq:I1I2I3}
T^{(n)}
= \opL^nT^{(0)}+\txs\sum\limits_{s=1}^{n}\opL^{n-s} D^{(s)} +\txs\sum\limits_{s=1}^{n} \opL^{n-s}E_T^{(s-1)}(V^{(s-1)})
=: J_1+J_2+J_3.
\end{equation}
Define events $\bbM_n(\chi)$, $\bbT_n(\chi)$, and $\bbQ_n$ as
\begin{gather}
\bbM_n(\chi)=\txs\set*{ \N{T^{(n)}-\opL^nT^{(0)}}_2\le \tfrac{1}{2}(\kappa ^2\beta^{2\chi-1}-1)^{1/2}\beta^{\chi-3\varepsilon}}, \label{eq:event-cMn}\\
\bbT_n(\chi)=\txs\set*{\N{T^{(n)}}_2\le (\kappa ^2\beta^{2\chi-1}-1)^{1/2}\beta^{\chi-3\varepsilon}},\quad
\bbQ_n=\txs\set*{n<\txs\Nqb{\Lambda}}.   \label{eq:event-cTn}
\end{gather}

\begin{lemma}\label{lm:prop3:ppr1}
	Suppose  the conditions of \emph{\Cref{thm:thm1:ppr1}} hold and that $\chi\in(5\varepsilon-1/2,0]$ and $\kappa>\sqrt{2}$.
	If $V^{(0)}\in\sphere(\kappa \beta^{\chi})$ and $n<\min\txs\set*{\txs\Nqb{\Lambda}, \txs\Nout{\sphere(\kappa\beta^{\chi})}}$,
	then
	\begin{equation}\label{eq:lm:prop3:ppr1}
	\txs\prob{\bbM_n(\chi+1/2) }\ge 1- 2d \exp(-C_{\kappa} \gamma\kappa^{-2}\nu^{-1}\eta_p^{-2} \beta^{-2\varepsilon}),
	\end{equation}
	where $C_{\kappa}$ is as in \cref{eq:C_kappa}.
\end{lemma}

\begin{proof}
	Since $\kappa >\sqrt{2}$, we have $\kappa^2\beta^{2\chi}>2$ and $\kappa\beta^{\chi}<[2(\kappa^2\beta^{2\chi}-1)]^{1/2}$.
	Thus, by \cref{itm:beta-C_T},
	\[
		4C_T\kappa ^3\wtd\eta_p ^2\gamma^{-1}\beta^{1+3\chi}(\kappa ^2\beta^{2\chi}-1)^{-1/2}\beta^{-1/2-\chi}
		\le4\sqrt{2}C_T\kappa ^2\wtd\eta_p ^2\gamma^{-1}\beta^{1/2+\chi}\le 1.
	\]
	For any $n<\min\txs\set*{\txs\Nqb{\Lambda}, \txs\Nout{\sphere(\kappa\beta^{\chi})}}$,
	$V^{(n)}\in\sphere(\kappa \beta^{\chi})$ and thus
	$\N{T^{(n)}}_2\le \sqrt{\kappa ^2\beta^{2\chi}-1}$ by \cref{eq:sphereK-normT}.
	Therefore, by \cref{itm:lm:diff-T:normdT} of Lemma~\ref{lm:diff-T}, we have
	\begin{equation}\label{eq:norm-D}
	\begin{aligned}
	\N{D^{(n+1)}}_2
	&= \N*{{T^{(n+1)}-T^{(n)}-\txs\E{T^{(n+1)}-T^{(n)}\given\fil_n}}}_2
	\\ &\le  \N{T^{(n+1)}-T^{(n)}}_2 +\txs\E{\N{T^{(n+1)}-T^{(n)}}_2 \given\fil_n}
	\\ &\le 2\nu^{1/2}\wtd\eta_p \beta(1+\N{T^{(n)}}_2^2)[1+C_T\wtd\eta_p \beta(1+\N{T^{(n)}}_2^2)^{1/2}]
	\\ &\le 2\kappa ^2\nu^{1/2}\wtd\eta_p \beta^{1+2\chi}[1+C_T\kappa \wtd\eta_p \beta^{1+\chi}].
	\end{aligned}
	\end{equation}
	For any $n<\min\txs\set*{\txs\Nqb{\Lambda}, \txs\Nout{\sphere(\kappa\beta^{\chi})}}$,
	\begin{align*}
	\N{J_3}_2
	&\le \txs\sum\limits_{s=1}^{n} \N{\opL}_2 ^{n-s}\N{E_T^{(s-1)}(V^{(s-1)})}_2
	\\ &\le  C_T\nu^{1/2}\kappa ^3\wtd\eta_p^2\beta^{2+3\chi} \txs\sum\limits_{s=1}^{n}(1-\beta\gamma)^{n-s}
	\\ &\le  \tfrac{C_T\nu^{1/2}\kappa ^3\wtd\eta_p^2\beta^{2+3\chi}}{\beta \gamma}
	= C_T\nu^{1/2}\kappa ^3\wtd\eta_p^2\gamma^{-1}\beta^{1+3\chi}
	\le  \tfrac{1}{4}\nu^{1/2}(\kappa ^2\beta^{2\chi}-1)^{1/2}\beta^{1/2+\chi}
	.
	\end{align*}
	Similarly,
	\begin{align*}
	\N{J_2}_2
	\le \txs\sum\limits_{s=1}^{n} \N{\opL}_2 ^{n-s}\N{D^{(s)}}_2
	\le  \tfrac{2\kappa ^2\nu^{1/2}\wtd\eta_p \beta^{2\chi}(1+C_T\kappa \wtd\eta_p \beta^{1+\chi})}{\gamma}
	\le  \tfrac{2\kappa ^2\nu^{1/2}\wtd\eta_p \beta^{2\chi}}{\gamma}+\tfrac{1}{2}\nu^{1/2}(\kappa ^2\beta^{2\chi}-1)^{1/2}\beta^{1/2+\chi}
	.
	\end{align*}
	Also,
	$
	\N{J_1}_2
	\le \N{\opL}_2 ^{n}\N{T^{(0)}}_2
	\le \N{T^{(0)}}_2
	\le \nu^{1/2}(\kappa ^2\beta^{2\chi}-1)^{1/2}.
	$
	For fixed $n>0$ and $\beta>0$,
	\[
	\txs\set*{
		M_0^{(n)}:= \opL^nT^{(0)},M_t^{(n)}:=\opL^nT^{(0)}+\txs\sum\limits\limits_{s=1}^{\min\txs\set{t,\txs\Nout{\sphere(\kappa )}-1}}\opL^{n-s}D^{(s)}
		\setdelim 1\le t\le n
	}
	\]
	forms a martingale with respect to $\fil_t$,
	because
	$
	\txs\E{\N{M_t^{(n)}}_2 }
	\le \N{J_1}_2 +\N{J_2}_2 <+\infty,
	$
	and
	\[
	\txs\E{M_{t+1}^{(n)}-M_{t}^{(n)}\given \fil_t}
	=\txs\E{ \opL^{n-t-1} D^{(t+1)}\given \fil_t}
	= \opL^{n-t-1} \txs\E{D^{(t+1)}\given \fil_t}=0.
	\]
	Use the matrix version of Azuma's inequality \cite[Section~7.2]{tropp2012user} to get,
	for any $\alpha>0$,
	\[
	\txs\prob{\N{M_n^{(n)}-M_0^{(n)}}_2\ge \alpha}\le 2d\exp(-\tfrac{\alpha^2}{2\sigma^2}),
	\]
	where
	\begin{align*}
	\sigma^2&= \txs\sum\limits_{s=1}^{\min\txs\set{n,\txs\Nout{\sphere(\kappa )}-1}}\N{\opL^{n-s}D^{(s)}}_2^2
	\\ &\le  [2\kappa ^2\nu^{1/2}\wtd\eta_p \beta^{1+2\chi}(1+C_T\kappa \wtd\eta_p \beta^{1+\chi})]^2\txs\sum\limits_{s=1}^{\min\txs\set{n,\txs\Nout{\sphere(\kappa )}-1}}(1-\beta\gamma)^{2(n-s)}
	\\ &\le  \tfrac{4\kappa ^4\nu \wtd\eta_p ^2\beta^{2+4\chi}(1+C_T \kappa \wtd\eta_p \beta^{1+\chi})^2}{\beta\gamma[2-\beta\gamma]}
	\\ &\le  \tfrac{4\kappa ^4\nu \wtd\eta_p ^2\gamma^{-1}\beta^{1+4\chi}(1+\tfrac{C_T}{2C_{\Delta}} )^2}{3-\sqrt{2}}
                \qquad\qquad\text{(by \cref{itm:beta-C_Vdelta} and $\wtd\eta_p \beta^{1/2}\le\tfrac{1}{2\kappa C_{\Delta}}$)}
	\\ &= C_{\sigma} \kappa ^4\nu \gamma^{-1}\wtd\eta_p ^2\beta^{1+4\chi},
	\end{align*}
	and $C_{\sigma}=\tfrac{(C_T+2C_{\Delta} )^2}{(3-\sqrt{2})C_{\Delta}^2}$.
	Thus, noticing $J_2=M_n^{(n)}-M_0^{(n)}$ for $n\le \txs\Nout{\sphere(\kappa )}-1$, we have
	\[
	\txs\prob{\N{J_2}_2 \ge \alpha}\le 2d\exp\left(-\tfrac{\alpha^2}{2C_{\sigma}\kappa ^4\nu \gamma^{-1}\wtd\eta_p ^2\beta^{1+4\chi}}\right).
	\]
	Choosing $\alpha=\tfrac{1}{4}(\kappa ^2\beta^{2\chi}-1)^{1/2}\beta^{\chi+1/2-3\varepsilon}$ and  noticing
$T^{(n)}-\opL^nT^{(0)}=J_2+J_3$ and $\N{J_3}_2 
	\le\tfrac{1}{4}(\kappa ^2\beta^{2\chi}-1)^{1/2}\beta^{\chi+1/2-3\varepsilon}$, we have
	\begin{align*}
	\txs\prob{\bbM_n(\chi+1/2)^{\rmc}}
	&=  \txs\prob{\N{T^{(n)}-\opL^nT^{(0)}}_2  \ge \tfrac{1}{2}(\kappa ^2\beta^{2\chi}-1)^{1/2}\beta^{\chi+1/2-3\varepsilon}}
	\\ &\le \txs\prob{\N{J_2}_2  \ge \tfrac{1}{4}(\kappa ^2\beta^{2\chi}-1)^{1/2}\beta^{\chi+1/2-3\varepsilon}}
	\\ &\le  2d\exp\left(-\tfrac{\kappa ^2\beta^{2\chi}-1}{32C_{\sigma}\kappa ^4\nu \gamma^{-1}\wtd\eta_p ^2\beta^{2\chi}}\beta^{-6\varepsilon}\right)
	\\ &\le  2d\exp\left(-\tfrac{\kappa ^2\beta^{2\chi}}{64C_{\sigma}\kappa ^4\nu \gamma^{-1}\wtd\eta_p ^2\beta^{2\chi}}\beta^{-6\varepsilon}\right)
	\\ &= 2d\exp(-C_{\kappa}\gamma\kappa^{-2}\nu^{-1}\eta_p^{-2} \beta^{-2\varepsilon}),
	\end{align*}
	where $C_{\kappa} =\tfrac{1}{64C_{\sigma}}$ which is the same as in \cref{eq:C_kappa}.
\end{proof}

\begin{lemma}\label{lm:prop5:ppr1}
	Suppose  the conditions of \emph{\Cref{thm:thm1:ppr1}} hold.
	If
	\[
	N_{2^{-m}(1-6\varepsilon)}<\min\txs\set*{\txs\Nqb{\Lambda}, \txs\Nout{\sphere(\kappa\beta^{\chi})}}
	\]
	and $V^{(0)}\in\sphere(\beta^{(1-2^{1-m})(3\varepsilon-1/2)}\kappa _m/2)$ with $m\ge 2$,
	then for $\kappa _m>\sqrt{2}$
	\[
	\txs\prob{\bbH_m}\ge 1-2d N_{2^{-m}(1-6\varepsilon)}\exp(-C_{\kappa}\gamma\kappa _m^{-2}\nu^{-1}\eta_p^{-2}\beta^{-2\varepsilon}),
	\]
	where
	$
	\bbH_m=\txs\set*{\txs\Nin{\sphere(\sqrt{3/2}\beta^{(1-2^{2-m})(3\varepsilon-1/2)}\kappa _m)}\le N_{2^{-m}(1-6\varepsilon)}}.
	$	
\end{lemma}

\begin{proof}
	By the definition of the event $\bbT_n$,
	\[
	\bbT_n(2^{-m}[1-6\varepsilon]+3\varepsilon)
	=\txs\set{\N{T^{(n)}}_2\le(\kappa _m^2-\beta^{(1-2^{1-m})(1-6\varepsilon)})^{1/2}\beta^{(1-2^{2-m})(3\varepsilon-1/2)}}
	.
	\]
	For $n\ge N_{2^{-m}(1-6\varepsilon)}$ and
	$V^{(0)}\in\sphere(\beta^{(1-2^{1-m})(3\varepsilon-1/2)}\kappa _m/2)$,
	we know
	$
	\bbM_n(2^{-m}(1-6\varepsilon)+3\varepsilon)\subset\bbT_n(2^{-m}(1-6\varepsilon)+3\varepsilon)
	$
	because
	\begin{align*}
	\N{T^{(n)}}_2
	&\le \N{T^{(n)}-\opL^nT^{(0)}}_2 +\N{\opL}_2 ^n\N{T^{(0)}}_2
	\\ &\le \tfrac{1}{2}\big(\kappa _m^2-\beta^{(1-2^{1-m})(1-6\varepsilon)}\big)^{1/2}\beta^{(1-2^{2-m})(3\varepsilon-1/2)}
	\\ &\qquad +\beta^{2^{-m}(1-6\varepsilon)}\big(\tfrac{\kappa _m^2}{4}-\beta^{(1-2^{1-m})(1-6\varepsilon)}\big)^{1/2}\beta^{(1-2^{1-m})(3\varepsilon-1/2)}
	\\ &\le  \big(\kappa _m^2-\beta^{(1-2^{1-m})(1-6\varepsilon)}\big)^{1/2}\beta^{(1-2^{1-m})(3\varepsilon-1/2)}
	.
	\end{align*}
	Therefore, noticing
	\begin{align*}
	\big(\kappa _m^2-\beta^{(1-2^{1-m})(1-6\varepsilon)}\big)^{1/2}\beta^{(1-2^{2-m})(3\varepsilon-1/2)}
	&= \big(\beta^{(1-2^{2-m})(6\varepsilon-1)}\kappa _m^2-\beta^{2^{1-m}(1-6\varepsilon)}\big)^{1/2} \\
	&\le \big(\tfrac{3}{2}\beta^{(1-2^{2-m})(6\varepsilon-1)}\kappa _m^2-1\big)^{1/2},
	\end{align*}
	we get 
	\[		
    \bbM_{N_{2^{-m}(1-6\varepsilon)}}(2^{-m}(1-6\varepsilon)+3\varepsilon)
           \subset\txs\set{\txs\what N\le N_{2^{-m}(1-6\varepsilon)}}=\bbH_m,
	\]
where $\what N=\Nin[]{\sphere(\sqrt{3/2}\beta^{(1-2^{2-m})(3\varepsilon-1/2)}\kappa _m)}$.
	Since
	\begin{multline*}
		\txs\bigcap\limits_{n\le\min\txs\set{N_{2^{-m}(1-6\varepsilon)},\txs\what N-1}}\bbM_n(2^{-m}(1-6\varepsilon)+3\varepsilon)
	\cap \bbH_m^{\rmc}
	\\\subset \txs\bigcap\limits_{n\le N_{2^{-m}(1-6\varepsilon)}}\bbM_n(2^{-m}(1-6\varepsilon)+3\varepsilon)
	\subset \bbM_{N_{2^{-m}(1-6\varepsilon)}}(2^{-m}(1-6\varepsilon)+3\varepsilon),
	\end{multline*}
	we have
	\[
		\txs\bigcap\limits_{n\le \min\txs\set{N_{2^{-m}(1-6\varepsilon)},\txs\what N-1}}\bbM_n(2^{-m}(1-6\varepsilon)+3\varepsilon)
	\subset \bbH_m.
	\]
	By Lemma~\ref{lm:prop3:ppr1} with $\chi=2^{-m}(1-6\varepsilon)+3\varepsilon-\tfrac{1}{2}=2^{-m}(1-2^{m-1})(1-6\varepsilon)$, we get
	\begin{align*}
	\txs\prob{\bbH_m^{\rmc}}
	&\le \txs\prob{\txs\bigcup\limits_{n\le \min\txs\set{N_{2^{-m}(1-6\varepsilon)},\txs\what N-1}}\bbM_n(2^{-m}(1-6\varepsilon)+3\varepsilon)^{\rmc}}
	\\ &\le \min\txs\set{N_{2^{-m}(1-6\varepsilon)},\txs\what N-1}
	\times 2d\exp(-C_{\kappa}\gamma\kappa _m^{-2}\nu^{-1}\eta_p^{-2}\beta^{-2\varepsilon})
	\\ &\le 2dN_{2^{-m}(1-6\varepsilon)}\exp(-C_{\kappa}\gamma\kappa _m^{-2}\nu^{-1}\eta_p^{-2}\beta^{-2\varepsilon})
	,
	\end{align*}
	as expected.
\end{proof}

\begin{lemma}\label{lm:prop4:ppr1}
	Suppose  the conditions of \emph{\Cref{thm:thm1:ppr1}} hold.
	If $V^{(0)}\in\sphere(\kappa /2)$ with $\kappa>2\sqrt{2}$,
	$K>N_{1-6\varepsilon}$,
	then there exists a high-probability event $\bbH_1\cap\bbQ_K=\txs\bigcap_{n\in [N_{1/2-3\varepsilon},K]}\bbT_n(1/2)\cap\bbQ_K$ satisfying
	\[
		\txs\prob{\bbH_1\cap\bbQ_K}\ge 1-
			2d K\exp(-C_{\kappa}  \gamma\kappa^{-2}\nu^{-1}\eta_p^{-2} \beta^{-2\varepsilon})
			- K(\ee d+p+1)\exp\left(-C_{\psi}\min\txs\set{\psi^{-1},\psi^{-2}}\beta^{-2\varepsilon}\right)
			,
	\]
	such that for any $n\in[N_{1-6\varepsilon},K]$,
	\[
	\txs\E{T^{(n)}\circ T^{(n)}\overevent \bbH_1\cap\bbQ_K}
	\le\opL^{2n}T^{(0)}\circ T^{(0)} +2\beta^2[I-\opL^2]^{-1}[I-\opL^{2n}]H_{\circ} +R_E,
	\]
	where
	$ \N{R_E}_2 \le C_{\circ}\kappa^4\gamma^{-1}\nu^2\wtd\eta_p^2\beta^{3/2-3\varepsilon}$,
	$H_{\circ}=\varc{\ul Y\ol Y^{\T}}\le 16\psi^4 H$ is as in \emph{\cref{itm:lm:diff-T:varcdT}} of \emph{Lemma~\ref{lm:diff-T}},
	and $C_{\circ}$ is as in \cref{eq:C_circE}.
\end{lemma}

\begin{proof}
	First we estimate the probability of the event $\bbH_1$.
	We know $\bbT_n(1/2)\subset\txs\set*{\N{T^{(n)}}_2\le(\kappa ^2-1)^{1/2}}$.
	If $K\ge\txs\Nout{\sphere(\kappa )}$, then there exists some $n\le K$, such that $V^{(n)}\notin\sphere(\kappa )$, i.e.,
	$\N{T^{(n)}}_2> (\kappa ^2-1)^{1/2}$ by \cref{eq:sphereK-normT}.
	Thus,
	\[
	\txs\set*{K\ge\txs\Nout{\sphere(\kappa )}}\subset\txs\bigcup\limits_{n\le K}\txs\set*{\N{T^{(n)}}_2>(\kappa ^2-1)^{1/2}}\subset\txs\bigcup\limits_{n\le K}\bbT_n(1/2)^{\rmc}.
	\]
	On the other hand, 
	for $n\ge N_{1/2-3\varepsilon}$ and $V^{(0)}\in\sphere(\kappa /2)$,
	$\bbM_n(1/2)\subset\bbT_n(1/2)$ because
	\begin{equation}\label{eq:Mn-subset-Tn}
	\begin{aligned}[b]
	\N{T^{(n)}}_2
	&\le\N{T^{(n)}-\opL^nT^{(0)}}_2 +\N{\opL}_2 ^n\N{T^{(0)}}_2\\
	&\le\tfrac{1}{2}(\kappa ^2-1)^{1/2}\beta^{1/2-3\varepsilon}+\beta^{1/2-3\varepsilon}\big(\tfrac{\kappa ^2}{4}-1\big)^{1/2}\\
	&\le (\kappa ^2-1)^{1/2}\beta^{1/2-3\varepsilon}.
	\end{aligned}
	\end{equation}
	Therefore,
	\[
	\txs\bigcap\limits_{n\in[N_{1/2-3\varepsilon},K]}\bbM_n(1/2)
	\subset \txs\bigcap\limits_{n\in[N_{1/2-3\varepsilon},K]}\bbT_n(1/2)
	\subset \txs\set{K\le\txs\Nout{\sphere(\kappa )}-1},
	\]
	and so
	\begin{align*}
	\txs\bigcap\limits_{n\le \min\txs\set{K,\txs\Nout{\sphere(\kappa )}-1}}\bbM_n(1/2)
	&\subset\txs\bigcap\limits_{n\in [N_{1/2-3\varepsilon},\min\txs\set{K,\txs\Nout{\sphere(\kappa )}-1}]}\bbM_n(1/2) \\
	&=\txs\bigcap\limits_{n\in [N_{1/2-3\varepsilon},K]}\bbM_n(1/2) \\
	&\subset\txs\bigcap\limits_{n\in [N_{1/2-3\varepsilon},K]}\bbT_n(1/2)
	=:\bbH_1.
	\end{align*}
	By Lemma~\ref{lm:prop3:ppr1} with $\chi=0$, we have
	\begin{align*}
	\MoveEqLeft[4] \txs\prob{\txs\bigcup\limits_{n\le \min\txs\set{K,\txs\Nout{\sphere(\kappa )}-1}}\bbM_n(1/2)^{\rmc}\cap\bbQ_K}  \\
	&\le \min\txs\set{K,\txs\Nout{\sphere(\kappa )}-1}\cdot 2d\exp(-C_{\kappa} \gamma\kappa^{-2}\nu^{-1}\eta_p^{-2} \beta^{-2\varepsilon}) \\
	&= 2d K\exp(-C_{\kappa}  \gamma\kappa^{-2}\nu^{-1}\eta_p^{-2} \beta^{-2\varepsilon}).
	\end{align*}
	Thus, by Lemma~\ref{lm:quasi-bounded},
	\begin{align*}
	\txs\prob{(\bbH_1\cap\bbQ_K)^{\rmc}}&=\txs\prob{\bbH_1^{\rmc}\cup\bbQ_K^{\rmc}}
	   =\txs\prob{\bbH_1^{\rmc}\cap\bbQ_K}+\txs\prob{\bbQ_K^{\rmc}}
	\\ &\le\txs\prob{\txs\bigcup\limits_{n\le \min\txs\set{K,\txs\Nout{\sphere(\kappa )}-1}}\bbM_n(1/2)^{\rmc}\cap\bbQ_K}
	+\txs\prob{\bbQ_K^{\rmc}}
	\\ &\le 2d K\exp(-C_{\kappa}  \gamma\kappa^{-2}\nu^{-1}\eta_p^{-2} \beta^{-2\varepsilon})
	+K(\ee d+p+1)\exp\left(-C_{\psi}\min\txs\set{\psi^{-1},\psi^{-2}}\beta^{-2\varepsilon}\right)
	.
	\end{align*}
	
	Next we estimate the expectation.
	Since
	\[
	\bbH_1=\txs\bigcap\limits_{n\in[N_{1/2-3\varepsilon},K]}\bbT_n(1/2)
	\subset\txs\bigcap\limits_{n\in[N_{1/2-3\varepsilon},K]}\txs\set*{\ind{\bbT_{n-1}}D^{(n)}=D^{(n)}},
	\]
	we have for $n\in[N_{1/2-3\varepsilon},K]$
	\begin{align*}
	T^{(n)}\ind{\bbH_1\cap\bbQ_K}
	&= \ind{\bbQ_K}\Big( \opL^nT^{(0)}+\txs\sum\limits_{s=1}^{N_{1/2-3\varepsilon}-1}\opL^{n-s} D^{(s)}+\txs\sum\limits_{s=N_{1/2-3\varepsilon}}^{n}\opL^{n-s} D^{(s)}\ind{\bbT_{s-1}}
	\txs\sum\limits_{s=1}^{n} \opL^{n-s}E_T^{(s-1)}(V^{(s-1)}) \Big)
	\\ &=: \wtd J_1+\wtd J_{21}+\wtd J_{22}+\wtd J_3.
	\end{align*}
	In what follows, we simply write $E_T^{(n)}=E_T^{(n)}(V^{(n)})$ for convenience.
	Then 
	\begin{align*}
	\txs\E{T^{(n)}\circ T^{(n)}\overevent \bbH_1\cap\bbQ_K}
	 &= \txs\E{T^{(n)}\circ T^{(n)}\ind{\bbH_1\cap\bbQ_K}}		\\
	&=\txs\E{\wtd J_1\circ \wtd J_1}+2\txs\E{\wtd J_1\circ \wtd J_{21}}+2\txs\E{\wtd J_1\circ \wtd J_{22}}+2\txs\E{\wtd J_1\circ \wtd J_3}		\\
	&\quad+\txs\E{[\wtd J_{21}+\wtd J_{22}]\circ [\wtd J_{21}+\wtd J_{22}]}			+2\txs\E{[\wtd J_{21}+\wtd J_{22}]\circ \wtd J_3} +\txs\E{\wtd J_3\circ \wtd J_3}		\\
	&\le\txs\E{\wtd J_1\circ \wtd J_1}+2\txs\E{\wtd J_1\circ \wtd J_{21}}+2\txs\E{\wtd J_1\circ \wtd J_{22}}+2\txs\E{\wtd J_1\circ \wtd J_3}			\\
	&\quad+2\txs\E{\wtd J_{21}\circ \wtd J_{21}}+4\txs\E{\wtd J_{21}\circ \wtd J_{22}} +2\txs\E{\wtd J_{22}\circ \wtd J_{22}} +2\txs\E{\wtd J_3\circ \wtd J_3}		.
	\end{align*}
	Each summand above for $n\in[N_{1-6\varepsilon},K]$ can be estimated with careful calculation (see \cref{ssec:estimation-in-proof-of-lemma-lm}), which reads
	\begin{enumerate}
		\item
		$ \txs\E{\wtd J_1\circ \wtd J_1}=\opL^{2n}T^{(0)}\circ T^{(0)}$.
		\item
		$ \txs\E{\wtd J_1\circ \wtd J_{21}}=0$.
		\item
		$\txs\E{\wtd J_1\circ \wtd J_{22}}=0$.
		\item
		$\N{\txs\E{\wtd J_1\circ \wtd J_3}}_2
		\le \tfrac{1}{2}C_T\nu^{1/2}\wtd\eta_p ^2\gamma^{-1}\kappa ^4\beta^{2-6\varepsilon}$.
		\item
		$\txs\E{\wtd J_{21}\circ \wtd J_{22}}=0$.
	\item $\txs\E{\wtd J_{21}\circ \wtd J_{21}}=\beta^2\txs\sum\limits_{s=1}^{N_{1/2-3\varepsilon}-1}\opL^{2(n-s)}H_{\circ} + E_{21}$,
		where
		\[
			\N{E_{21}}_2
			\le \left(\tfrac{29+8\sqrt{2}}{64}+2C_T\kappa(\wtd\eta_p\beta)+C_T^2\kappa ^2(\wtd\eta_p\beta)^2\right)\gamma^{-1}\kappa^4\nu^2\wtd\eta_p^2\beta^{2-6\varepsilon}
			.
		\]
		\item
			$
			 \txs\E{\wtd J_{22}\circ \wtd J_{22}}
			 =\beta^2\txs\sum\limits_{s=N_{1/2-3\varepsilon}}^{n}\opL^{2(n-s)}H_{\circ} + E_{22},
			 $
			 where
			 \[
				 \N{E_{22}}_2
				 \le  \tfrac{1}{3-\sqrt{2}}\left( \tfrac{29+8\sqrt{2}}{32}+4C_T\kappa\wtd\eta_p\beta^{1/2+3\varepsilon}+2C_T^2\kappa ^2\wtd\eta_p^2\beta^{3/2+3\varepsilon}  \right)\gamma^{-1}\kappa^4\nu^2\wtd\eta_p^2\beta^{3/2-3\varepsilon}
				 .
			 \]
		\item
		$\N{\txs\E{\wtd J_3\circ \wtd J_3}}_2
		\le  \tfrac{1}{3-\sqrt{2}}C_T^2\nu\wtd\eta_p^4\gamma^{-1}\kappa ^6\beta^3
		$.
	\end{enumerate}
	Collecting all estimates together, we obtain
	\begin{align*}
	\txs\E{T^{(n)}\circ T^{(n)}\overevent \bbH_1\cap\bbQ_K}
	&\le\opL^{2n}T^{(0)}\circ T^{(0)} +2\beta^2\txs\sum\limits_{s=1}^{n}\opL^{2(n-s)}H_{\circ} +R_E
	\\ &\le\opL^{2n}T^{(0)}\circ T^{(0)} +2\beta^2[I-\opL^2]^{-1}[I-\opL^{2n}]H_{\circ} +R_E,
	\end{align*}
	where, by \cref{itm:beta-C_Vdelta}, $2C_{\Delta} \kappa \wtd\eta_p\beta^{1/2}\le 1$, and
	\begin{align*}
	\N{R_E}_2
	&\le 2\bigg[
	\tfrac{C_T}{2}\beta^{1/2-3\varepsilon} +
	\left(\tfrac{29+8\sqrt{2}}{64}+2C_T\kappa \wtd\eta_p \beta+C_T^2\kappa ^2(\wtd\eta_p \beta)^2\right)\beta^{1/2-3\varepsilon}+
	\tfrac{C_T^2}{3-\sqrt{2}}\wtd\eta_p^2\kappa ^2\beta^{3/2+3\varepsilon} \\
	&\qquad+\tfrac{2}{3-\sqrt{2}}\left(\tfrac{29+8\sqrt{2}}{64}+2C_T\kappa \wtd\eta_p \beta^{1/2+3\varepsilon}+C_T^2\kappa ^2\wtd\eta_p^2 \beta^{3/2+3\varepsilon}\right)
	\bigg]\kappa^4\gamma^{-1}\nu^2\wtd\eta_p^2\beta^{3/2-3\varepsilon}		\\
	&\le 2\bigg[
	\tfrac{C_T}{2}\beta^{1/2-3\varepsilon} +
	\left(\tfrac{29+8\sqrt{2}}{64}+\tfrac{C_T}{C_{\Delta}}\beta^{1/2}+\tfrac{C_T^2}{4C_{\Delta}^2}\beta\right)\beta^{1/2-3\varepsilon}+
	\tfrac{C_T^2}{4(3-\sqrt{2})C_{\Delta}^2}\beta^{1/2+3\varepsilon}\\
	&\qquad+\tfrac{2}{3-\sqrt{2}}\left(\tfrac{29+8\sqrt{2}}{64}+\tfrac{C_T}{C_{\Delta}}\beta^{3\varepsilon}+\tfrac{C_T^2}{4C_{\Delta}^2}\beta^{1/2+3\varepsilon}\right)
	\bigg]\kappa^4\gamma^{-1}\nu^2\wtd\eta_p^2\beta^{3/2-3\varepsilon} \\
	&= C_{\circ}\kappa^4\gamma^{-1}\nu^2\wtd\eta_p^2\beta^{3/2-3\varepsilon},
	\end{align*}
	where $C_{\circ}$ is as given in \cref{eq:C_circE}.
\end{proof}

\subsection{Proof of Theorem \ref{thm:thm1:ppr1}}\label{ssec:proof1st}
Write $\wtd N_s=\tfrac{s\ln\beta }{\ln (1-\beta\gamma)}$. Then $(1-\beta\gamma)^{\wtd N_s}= \beta^s$ and
$N_s = \ceil*{\wtd N_s}$, where $N_s$ is defined in \cref{eq:Ns-dfn}.
It can be verified that $\wtd N_{s_1}+\wtd N_{s_2} = \wtd N_{s_1+s_2}$ for any $s_1,s_2$.

Write $\kappa _m=6^{(1-m)/2}\kappa$ for $m=1,\dots,M\equiv M(\epsilon)$.
Since $d\beta^{1-7\varepsilon}\le (\sqrt{2}-1)\lambda_1^{-1}\omega$,
we know
\[
\phi d^{1/2}\le \phi \omega^{1/2}\beta^{7\varepsilon/2-1/2}\le \beta^{(1-2^{1-M})(3\varepsilon-1/2)}\kappa _M/2.
\]
The key to our proof is to divide the whole process into $M$ segments of iterations.
Thanks to the strong Markov property of the process, we can use the final value of current segment as the initial guess of
the very next one.
By Lemma~\ref{lm:prop5:ppr1}, after the first segment of
\[
n_1:=\min\txs\set*{\txs\Nin{\sphere(\sqrt{3/2}\beta^{(1-2^{2-M})(3\varepsilon-1/2)}\kappa _1)},N_{2^{-M}(1-6\varepsilon)}}
\]
iterations,
$V^{(n_1)}$ lies in $\sphere(\sqrt{3/2}\beta^{(1-2^{2-M})(3\varepsilon-1/2)}\kappa _1)=\sphere(\beta^{(1-2^{2-M})(3\varepsilon-1/2)}\kappa _2/2)$
with high probability, 	which will be a good initial guess for the second segment. In general,
the $i$th segment of iterations starts with $V^{(n_{i-1})}$ and ends with $V^{(n_i)}$, where
\[
n_i=\min\txs\set*{\txs\Nin{\sphere(\beta^{(1-2^{i+1-M})(3\varepsilon-1/2)}\kappa_{i+1}/2)},\ceil*{\txs\sum\limits_{m=M+1-i}^M \wtd N_{2^{-m}(1-6\varepsilon)}}}.
\]
At the end of the $(M-1)$st segment of iterations, $V^{(n_{M-1})}$ is produced and it is going to be used as an initial guess for the last step,
at which
we can apply Lemma~\ref{lm:prop4:ppr1}.
Now $n_{M-1}=\min\txs\set*{\txs\Nin{\sphere(\kappa_M /2)},\what K}$,
where $\what K=\ceil*{\txs\sum\limits_{m=2}^M \wtd N_{2^{-m}(1-6\varepsilon)}}=\ceil*{\wtd N_{(1-2^{1-M})(1/2-3\varepsilon)}}$.
By $2^{2-M}\ge\tfrac{\varepsilon/2}{1/2-3\varepsilon}\ge2^{1-M}$,
we have
\[
N_{1/2-7\varepsilon/2}
= \ceil*{\wtd N_{1/2-7\varepsilon/2}}
\le \what K
\le \ceil*{\wtd N_{1/2-13\varepsilon/4}}
\le N_{1/2-13\varepsilon/4}
.
\]
Let $\what N=\Nin{\sphere(\sqrt{3/2}\beta^{(1-2^{2-m})(3\varepsilon-1/2)}\kappa _{M+1-m})}$, and
\begin{align*}
\wtd\bbH_m&=\txs\set*{\txs\what N\le \wtd N_{2^{-m}(1-6\varepsilon)}+n_{M-m}}\,\,\text{for $2\le m\le M$},\\
\wtd\bbH_1&= \txs\bigcap\limits_{n\in [N_{1/2-3\varepsilon},K-\txs\Nin{\sphere(\kappa_M /2)}]}\bbT_{n+\txs\Nin{\sphere(\kappa_M /2)}}(1/2),\\
\bbH&=\txs\bigcap\limits_{m=1}^M \wtd\bbH_m\cap\bbQ_K,
\end{align*}
where $n_0=0$.
We have
\begin{align*}
\txs\prob{\bbH^{\rmc}}
&=\txs\prob{\txs\bigcup\limits_{m=1}^M \wtd\bbH_m^{\rmc}\cup\bbQ_K^{\rmc}}\le \txs\sum\limits_{m=1}^M \txs\prob{\wtd\bbH_m^{\rmc}\cup\bbQ_K^{\rmc}}		\\
&\le\txs\sum\limits_{m=2}^M 2d N_{2^{-m}(1-6\varepsilon)}\exp(-C_{\kappa}\gamma\kappa _{M+1-m}^{-2}\nu^{-1}\eta_p^{-2}\beta^{-2\varepsilon}) \\
&\quad+2d \Big(K-\txs\sum\limits_{m=2}^M N_{2^{-m}(1-6\varepsilon)}\Big)\exp(-C_{\kappa}  \gamma\kappa_M^{-2}\nu^{-1}\eta_p^{-2} \beta^{-2\varepsilon}) \\
&\quad+ K(\ee d+p+1)\exp\left(-C_{\psi}\min\txs\set{\psi^{-1},\psi^{-2}}\beta^{-2\varepsilon}\right)
\\
&\le 2dK\exp(-C_{\kappa}  \gamma\kappa^{-2}\nu^{-1}\eta_p^{-2} \beta^{-2\varepsilon})
+ K(\ee d+p+1)\exp\left(-C_{\psi}\min\txs\set{\psi^{-1},\psi^{-2}}\beta^{-2\varepsilon}\right)
\\ &\le 2dK\exp(-C_{\kappa} 4\sqrt{2}C_T\nu^{-1}\beta^{\varepsilon}\beta^{-2\varepsilon})
+ K(\ee d+p+1)\exp\left(-C_{\psi}\min\txs\set{\psi^{-1},\psi^{-2}}\beta^{-2\varepsilon}\right) \\
&\hspace*{4cm}\qquad \text{(by \cref{itm:beta-C_T})}
\\ &\le 2dK\exp(-4\sqrt{2}C_TC_{\kappa} \nu^{-1}\beta^{-\varepsilon})
+ K(\ee d+p+1)\exp\left(-C_{\psi}\min\txs\set{\psi^{-1},\psi^{-2}}\beta^{-\varepsilon}\right)
\\ &\le K[(2+\ee)d+ p+1]\exp(-\max\txs\set{C_{\nu} \nu^{-1},C_\psi\min\txs\set{\psi^{-1},\psi^{-2}}}\beta^{-\varepsilon})
,
\end{align*}
where $C_{\nu}=4\sqrt{2}C_TC_{\kappa}$ is as given in \cref{eq:C_nu}.

Set $\bbH_{n'}':=\txs\set{\txs\Nin{\sphere(\kappa /2)}=n'}$.
If $n'> \what K$, then $\bbH\cap \bbH_{n'}'=\emptyset$. Otherwise if $n'\le \what K$, then, by Lemma~\ref{lm:prop5:ppr1},
$V^{(n')}\in\sphere(\kappa_M /2)$ and then $\N{T^{(n')}}_{\F}^2\le p( (\tfrac{\kappa_M }{2})^2-1)$.
Thus,
\[
\phi^2d(1-\beta\gamma)^{2(n'-1)}\ge\phi^2d(1-\beta\gamma)^{2(\what K-1)}
>\Big(\tfrac{\kappa_M }{2}\Big)^2\ge\tfrac{1}{p}\N{T^{(n')}}_{\F}^2.
\]
Hence, for any $n\in [N_{1-6\varepsilon}+\txs\Nin{\sphere(\kappa /2)},K]\subset[N_{1-6\varepsilon}+n',K+n']$, by Lemma~\ref{lm:prop4:ppr1}, we have
\[
\txs\E{T^{(n)}\circ T^{(n)}\ind{\bbH}\given\bbH_{n'}'\cap\fil_{n'}}
\le\opL^{2(n-n')}T^{(n')}\circ T^{(n')} +2\beta^2[I-\opL^2]^{-1}[I-\opL^{2(n-n')}]H_{\circ} +R_E.
\]
Recall that $\fil_{n'}$ is the $\sigma$-algebra filtration, i.e., the information known by step $n'$.
Introduce $\Sum(A)$ for the sum of all the entries of a matrix $A$. In particular, $\Sum(A\circ A)=\N{A}_{\F}^2$. We have
\begin{align*}
\MoveEqLeft\txs\E{\N{T^{(n)}}_{\F}^2\ind{\bbH}\given\bbH_{n'}'}
\\ &=\txs\E{\txs\E{\N{T^{(n)}}_{\F}^2\ind{\bbH}\given \bbH_{n'}'\cap\fil_{n'}}}
\\ &\le\txs\E{(1-\beta\gamma)^{2(n-n')}\N{T^{(n')}}_{\F}^2 +2\beta^2\Sum([I-\opL^2]^{-1}H_{\circ}) +\Sum(R_E)}
\\ &\le(1-\beta\gamma)^{2(n-1)}p\phi^2d +2\beta^2\Sum([I-\opL^2]^{-1}H_{\circ}) +\sqrt{p(d-p)}\N{R_E}_{\F}
\\ &\le(1-\beta\gamma)^{2(n-1)}p\phi^2d +2\beta^2\tfrac{1}{\beta(2-\lambda_1 \beta)}\Sum(G\circ H_{\circ})
+\sqrt{p(d-p)}C_{\circ}\sqrt{p}\kappa ^4(\nu \wtd\eta_p)^2\gamma^{-1}\beta^{3/2-3\varepsilon}
,
\end{align*}
where $G=[\gamma_{ij}]_{(d-p)\times p}$ with $\gamma_{ij}=\tfrac{1}{\lambda_j-\lambda_{p+i}}$.
Putting all together, we get
\begin{align*}
\txs\E{\N{T^{(n)}}_{\F}^2\overevent \bbH}
  &=\txs\E{\txs\E{\N{T^{(n)}}_{\F}^2\ind{\bbH} \given \bbH_{n'}'}}
\\ &\le(1-\beta\gamma)^{2(n-1)}p\phi^2d +\tfrac{2\beta}{2-\lambda_1 \beta}\Sum(G\circ H_{\circ})  +C_{\circ}\kappa ^4\nu^2\wtd\eta_p^2p\sqrt{d-p}\gamma^{-1}\beta^{3/2-3\varepsilon}
.
\end{align*}
Note that on $\bbH$, $\txs\Nin{\sphere(\kappa /2)}\le\what K$.
So the expectation is valid for any $n\in [N_{1-2\varepsilon}+\what K, K]$.
Finally we estimate $\Sum(G\circ H_{\circ})$.
By Lemma~\ref{lm:diff-T}, $H_{\circ}\le 16\psi^4 H$, and hence
\[
\Sum(G\circ H_{\circ})
\le \txs\sum\limits_{j=1}^{p}\txs\sum\limits_{i=1}^{d-p} \tfrac{16\psi^4\lambda_{p+i}\lambda_j}{\lambda_j-\lambda_{p+i}}
=16\psi^4\varphi(p,d;\Lambda).
\]
This completes the proof.

\section{Proofs of Theorems \ref{thm:thm2:ppr1} and \ref{thm:thm3:ppr1}}
\label{sec:proof-of-cref-thm-thm2-ppr1}
To prove \Cref{thm:thm2:ppr1}, we will first prove that it is a high-probability event that $V^{(0)}$ satisfies the initial
condition there, which is the result of Lemma~\ref{lm:lm1:ppr1} below.
Then, together with \Cref{thm:thm1:ppr1}, we will have its conclusion.
During estimating the probability, we need a property on the Gaussian hypergeometric function of a matrix argument,
as in Lemma~\ref{lm:2F1}.

The gamma function and the multivariate gamma function are
\[
\Gamma(x):=\txs\int_0^\infty t^{x-1}\exp(-t)\diff t, \quad
\Gamma_m(x):=\pi^{m(m-1)/4}\txs\prod_{i=1}^{m} \Gamma\Big(x-\tfrac{i-1}{2}\Big),
\]
respectively. Denote  by $\hyperg$ the Gaussian hypergeometric function of matrix argument
(see \cite[Definition~7.3.1]{muirhead1982aspects}), and also by $\hyperg[1][0]$ and $\hyperg[1][1]$
the generalized hypergeometric functions  that will be used later.

\begin{lemma}\label{lm:2F1}
	For any scalar $a,b,c$ and a symmetric matrix $T\in \bbR^{m\times m}$,
	\begin{multline} \label{eq:hyperg-reflect}
	\hyperg(a,b;c;T)=\tfrac{\Gamma_m(c-a-b)\Gamma_m(c)}{\Gamma_m(c-a)\Gamma_m(c-b)}\hyperg(a,b;a+b-c+\tfrac{m+1}{2};I-T)
	\\+\tfrac{\Gamma_m(a+b-c)\Gamma_m(c)}{\Gamma_m(a)\Gamma_m(b)}\det(I-T)^{c-a-b}\hyperg(c-a,c-b;c-a-b+\tfrac{m+1}{2};I-T).
	\end{multline}
\end{lemma}
Our proof of Lemma~\ref{lm:2F1} is similar to that for the case $p=1$ by Kummer's solutions of the hypergeometric differential equation (see, e.g., \cite[Section~3.8]{luke1969special}), and we leave it to \cref{ssec:proof-of-lm}.

\begin{lemma}\label{lm:lm1:ppr1}
	Suppose $p<(d+1)/2$.
	If $V^{(0)}$ satisfies the condition that $\cR(V^{(0)})$ is uniformly sampled from $\Grassmann_p(\bbR^d)$,
	then for sufficiently large $d$ and $\delta\in[0,1]$,
	there exists a constant $C_p$, independent of $\delta$ and $d$, such that
	\begin{equation}\label{eq:lm1:ppr1}
	\txs\prob{V^{(0)}\in \sphere(C_p\delta^{-1} d^{1/2})}
	\ge 1- \delta^{p^2}.
	\end{equation}
\end{lemma}

\begin{proof}
	Let $1\ge\sigma_1\ge\dots\ge\sigma_p\ge0$ be the singular values of $\ol V^{(0)}$, and then $\sigma_i=\cos\theta_i$,
	where $\theta_i$ are the canonical angles between $\cR(V^{(0)})$ and $\cR(V_*)$ (recall \cref{eq:YV-dfn}).
	By \cite[Theorem~1]{absilEK2006largest}, since $p<(d+1)/2$, the probability distribution function of $\sigma_p$ is
	\begin{align*}
	\txs\prob{V^{(0)}\in \sphere(1/x)}
	&= \txs\prob{\sigma_p\ge x}
	= \txs\prob{\theta_p\le \arccos x}
	\\ &= \tfrac{\Gamma(\tfrac{p+1}{2})\Gamma(\tfrac{d-p+1}{2})}{\Gamma(\tfrac{1}{2})\Gamma(\tfrac{d+1}{2})}(1-x^2)^{p(d-p)/2}\hyperg\left(\tfrac{d-p}{2},\tfrac{1}{2};\tfrac{d+1}{2};(1-x^2)I_p\right)
	.
	\end{align*}
	Set
	$
	f_d:=
	\tfrac{\Gamma_p(\tfrac{d+1}{2})\Gamma_p(\tfrac{p}{2})}{\Gamma_p(\tfrac{p+1}{2})\Gamma_p(\tfrac{d}{2})}
	, 
	g_d:=
	\tfrac{\Gamma_p(\tfrac{d+1}{2})\Gamma_p(-\tfrac{p}{2})}{\Gamma_p(\tfrac{d-p}{2})\Gamma_p(\tfrac{1}{2})}
	.
	$
	After some calculations that are deferred to \cref{ssec:complementary-calculation-in-proof-of-lm}),
	we know:
	\begin{itemize}
		\item
	in defining $g_d$, although $\Gamma_p(-\tfrac{p}{2})$ and $\Gamma_p(\tfrac{1}{2})$ may be  $\infty$,
	by analytic continuation, $\Gamma_p(-\tfrac{p}{2})/\Gamma_p(\tfrac{1}{2})$ is well-defined;
	\item
	$
	f_d^{-1}g_d
	= \tfrac{\Gamma(\tfrac{p+1}{2})\Gamma_p(\tfrac{d}{2})\Gamma_p(-\tfrac{p}{2})}{\Gamma(\tfrac{1}{2})\Gamma_p(\tfrac{d-p}{2})\Gamma_p(\tfrac{1}{2})}
	;
	$
	\item
	$
	\tfrac{\Gamma_p(\tfrac{d}{2})}{\Gamma_p(\tfrac{d-p}{2})}
	=\left( \tfrac{d}{2} \right)^{p^2/2}[1+\oo(1)]
	\,\,\text{as $d\to\infty$}.
	$
	\end{itemize}
	By \cref{eq:hyperg-reflect}, we have
	\[
		\hyperg\left(\tfrac{d-p}{2},\tfrac{1}{2};\tfrac{d+1}{2};(1-x^2)I_p\right)
			=  f_d\,\hyperg\left(\tfrac{d-p}{2},\tfrac{1}{2};\tfrac{1}{2};x^2I_p\right)
			+g_d\det(x^2I_p)^{p/2}\hyperg(\tfrac{p+1}{2},\tfrac{d}{2};\tfrac{2p+1}{2};x^2I_p).
	\]
	Also, \cite[Definition~7.3.1 and Corollary~7.3.5]{muirhead1982aspects} give us
	\[
	\hyperg\left(\tfrac{d-p}{2},\tfrac{1}{2};\tfrac{1}{2};x^2I_p\right)
	= \hyperg[1][0]\left(\tfrac{d-p}{2};x^2I_p\right)
	= \det(I_p-x^2I_p)^{-(d-p)/2}
	= (1-x^2)^{-p(d-p)/2}
	.
	\]
	Therefore,
	\[
	\txs\prob{V^{(0)}\in \sphere(1/x)}
	= 1+f_d^{-1}g_d\,(1-x^2)^{p(d-p)/2}x^{p^2}\hyperg\left(\tfrac{p+1}{2},\tfrac{d}{2};\tfrac{2p+1}{2};x^2I_p\right)
	.
	\]
	Substituting $x=(\delta^{-1} d^{1/2})^{-1}$  and
	by \cite[(8) of Section~7.4]{muirhead1982aspects},
	we get as $d\to\infty$
	\begin{align}
 \label{eq:Cp-dfn}
 \MoveEqLeft[1]\txs\prob{V^{(0)}\notin \sphere(\delta^{-1} d^{1/2})}
	\nonumber
	 \\&= -f_d^{-1}g_d(1-\delta^2 d^{-1})^{p(d-p)/2}(\delta^2 d^{-1})^{p^2/2}\hyperg\left(\tfrac{p+1}{2},\tfrac{d}{2};\tfrac{2p+1}{2};\tfrac{\delta^2}{d}I_p\right)
	 \\ &= \tfrac{\Gamma(\tfrac{p+1}{2})\Gamma_p(-\tfrac{p}{2})}{-\Gamma(\tfrac{1}{2})\Gamma_p(\tfrac{1}{2})}\tfrac{\Gamma_p(\tfrac{d}{2})}{\Gamma_p(\tfrac{d-p}{2})}\left(1-\tfrac{\delta^2}{d}\right)^{pd/2}\left(\tfrac{d}{\delta^2}-1\right)^{-p^2/2}\left[\hyperg[1][1]\left(\tfrac{p+1}{2};\tfrac{2p+1}{2};\tfrac{\delta^2}{2}I_p\right)+\oo(1)\right]
	\nonumber \\ &= \tfrac{\Gamma(\tfrac{p+1}{2})\Gamma_p(-\tfrac{p}{2})}{-\Gamma(\tfrac{1}{2})\Gamma_p(\tfrac{1}{2})}\left( \tfrac{d}{2} \right)^{p^2/2}[1+\oo(1)]\left[\exp\left(-\tfrac{p\delta^2}{2}\right)+\oo(1)\right]\left[\tfrac{\delta^{p^2}}{d^{p^2/2}}+\oo(1)\right]
	\left[\hyperg[1][1]\left(\tfrac{p+1}{2};\tfrac{2p+1}{2};\tfrac{\delta^2}{2}I_p\right)+\oo(1)\right]
	\nonumber \\ &= \tfrac{\Gamma(\tfrac{p+1}{2})\Gamma_p(-\tfrac{p}{2})}{-\Gamma(\tfrac{1}{2})\Gamma_p(\tfrac{1}{2})}\exp\left(-\tfrac{p\delta^2}{2}\right)\hyperg[1][1]\left(\tfrac{p+1}{2};\tfrac{2p+1}{2};\tfrac{\delta^2}{2}I_p\right)\delta^{p^2}[1+\oo(1)]
	\nonumber \\ &\le \tfrac{\Gamma(\tfrac{p+1}{2})\Gamma_p(-\tfrac{p}{2})}{-\Gamma(\tfrac{1}{2})\Gamma_p(\tfrac{1}{2})}\hyperg[1][1]\left(\tfrac{p+1}{2};\tfrac{2p+1}{2};\tfrac{1}{2}I_p\right)\delta^{p^2}2
	=: C_p^{p^2}\delta^{p^2}
	,
	\nonumber
	\end{align}
	where the  inequality is guaranteed by
	$\hyperg[1][1]\left( \tfrac{p+1}{2};\tfrac{2p+1}{2};\tfrac{\delta^2}{2}I_p \right)
	\le\hyperg[1][1]\left( \tfrac{p+1}{2};\tfrac{2p+1}{2};\tfrac{1}{2}I_p \right)$, according to \cite[Theorem~7.5.6]{muirhead1982aspects}.
	Substituting $\delta/C_p$ for $\delta$, we infer from \cref{eq:Cp-dfn} that
	$\txs\prob{V^{(0)}\notin \sphere(C_p\delta^{-1} d^{1/2})}\le\delta^{p^2}$.
	The claim \cref{eq:lm1:ppr1} is now a simple consequence.
\end{proof}

Now we are ready to prove \Cref{thm:thm2:ppr1}.

\begin{proof}[Proof of \emph{\Cref{thm:thm2:ppr1}}]
	Define the event
	$
	\bbH'_* = \txs\set*{ V^{(0)}\in \sphere(C_p\delta^{-1} d^{1/2})}.
	$
	Since  $\cR(V^{(0)})$ is uniformly sampled from $\Grassmann_p(\bbR^d)$,
	Lemma~\ref{lm:lm1:ppr1} says $\txs\prob{\bbH'_*}\ge1-\delta^{p^2}$.
	In the following, we will apply \Cref{thm:thm1:ppr1} with $\phi=C_p\delta^{-1},\,\omega=(\sqrt{2}+1)\lambda_1\delta^2$.
	Since \Cref{thm:thm1:ppr1} is valid on $\bbH'_*$,
	and
	\[
	K[(2+\ee)d+ p+1]\exp(-C_{\nu\psi}\beta^{-\varepsilon})\le\delta^{p^2},
	\]
	there exists an event $\bbH$ with
	\[
	\txs\prob{\bbH\given\bbH'_*}
	\ge 1 -
	K[(2+\ee)d+ p+1]\exp(-C_{\nu\psi}\beta^{-\varepsilon})
	\ge 1 -\delta^{p^2},
	\]
	such that for any $n\in [N_{3/2-37\varepsilon/4}(\beta),K]$,
	\begin{align*}
	\txs\E{\N{T^{(n)}}_{\F}^2\overevent \bbH\cap\bbH'_*}
	  &=\txs\prob{\bbH'_*}\txs\E{\N{T^{(n)}}_{\F}^2\ind{\bbH}\given \bbH'_*}
	\\ &\le \txs\E{\N{T^{(n)}}_{\F}^2\ind{\bbH}\given \bbH'_*}
	\\ &\le(1-\beta\gamma)^{2(n-1)}pC_p^2\delta^{-2}d
	+\tfrac{32\psi^4\beta}{2-\lambda_1 \beta}\varphi(p,d;\Lambda)
	+C_{\circ}\kappa ^4\nu^2\eta_p ^2\gamma^{-1}p\sqrt{d-p}\beta^{3/2-5\varepsilon}
	.
	\end{align*}
	Let $\bbH_*= \bbH\cap\bbH'_*$ for which
	$\txs\prob{\bbH_*}= \txs\prob{\bbH\given\bbH'_*} \txs\prob{\bbH'_*}\ge (1-\delta^{p^2})^2 \ge 1-2\delta^{p^2}$, as expected.
\end{proof}

Finally, we prove \Cref{thm:thm3:ppr1}.
\begin{proof}[Proof of \emph{\Cref{thm:thm3:ppr1}}]
	First we examine the conditions of \Cref{thm:thm2:ppr1} to make sue that they are satisfied.
	It can be seen $\beta_*\to 0$ as $N_*\to\infty$.
	Thus, $\beta_*$ satisfies \cref{eq:beta-bound} for sufficiently large $N_*$.
	We have
	\begin{multline*}
	(1-\beta_*\gamma)^{N_*}
	=\left(1-\tfrac{3\ln N_*}{2N_*}\right)^{N_*}
	=\exp(-\tfrac{3}{2}\ln N_*)[1+\oo(1)]
	=N_*^{-3/2}[1+\oo(1)]
	\\=\left(	\tfrac{3\ln N_*}{2\gamma\beta_*} \right)^{-3/2}[1+\oo(1)]
	=\tfrac{\beta_*^{3/2}\gamma^{3/2}}{(3/2)^{3/2}(\ln N_*)^{3/2}}[1+\oo(1)]
	\le\beta_*^{3/2}
	,
	\end{multline*}
	which implies $N_*\ge N_{3/2}(\beta)\ge N_{3/2-9\varepsilon}(\beta)$.
	
	The conclusion of the theorem will be a straightforward consequence if
	\[
	\wtd C(d,N_*,\delta)
	:=\tfrac{(1-\beta_*\gamma)^{2(N_*-1)}pC_p^2\delta^{-2}d
		+\tfrac{32\psi^4\beta_*}{2-\lambda_1 \beta_*}\varphi(p,d;\Lambda)
		+C_{\circ}\kappa ^4\nu^2\eta_p ^2\gamma^{-1}p\sqrt{d-p}\beta_*^{3/2-7\varepsilon}}
	{\tfrac{\varphi(p,d;\Lambda)}{\lambda_p-\lambda_{p+1}}\tfrac{\ln N_*}{N_*}}
	\]
	is bounded, say by $C_*(d,N_*,\delta)$ to be defined.
	In fact,
	\begin{align*}
	\wtd C(d,N_*,\delta)
	&=\gamma\tfrac{N_*}{\ln N_*}\left[(1-\beta_*\gamma)^{2(N_*-1)}C_p^2\delta^{-2}\tfrac{pd}{\varphi(p,d;\Lambda)}
	+\tfrac{32\psi^4\beta_*}{2-\lambda_1 \beta_*}
	+\tfrac{C_{\circ}\kappa ^4\nu^2\eta_p ^2\gamma^{-1}p\sqrt{d-p}}{\varphi(p,d;\Lambda)}\beta_*^{3/2-7\varepsilon}\right]
	\\ &\le
	\begin{multlined}[t]
	\gamma\tfrac{N_*}{\ln N_*}\left[\tfrac{\beta_*^{3}}{(1-\beta_*\gamma)^2}C_p^2\delta^{-2}\tfrac{pd}{\varphi(p,d;\Lambda)}
	+\tfrac{32\psi^4\beta_*}{2-\lambda_1 \beta_*}
	+\tfrac{C_{\circ}\kappa ^4\nu^2\eta_p ^2\gamma^{-1}p\sqrt{d-p}}{\varphi(p,d;\Lambda)}\beta_*^{3/2-7\varepsilon}\right]
	\\ \left(\text{by $N_*\ge N_{3/2}$, or equivalently, $(1-\beta_*\gamma)^{N_*}\le\beta_*^{3/2}$}\right)
	\end{multlined}
	\\ &\le
	\begin{multlined}[t]
	\gamma\tfrac{N_*}{\ln N_*}\beta_*\left[\tfrac{\beta_*^{2}}{(1-\beta_*\gamma)^2}C_p^2\delta^{-2}\tfrac{d}{p}\tfrac{1}{\tfrac{\lambda_1\lambda_d}{\lambda_1-\lambda_d}}
	+\tfrac{32\psi^4}{2-\lambda_1 \beta_*}
	+\tfrac{C_{\circ}\kappa ^4\nu^2\eta_p ^2\gamma^{-1}}{\sqrt{p}\tfrac{\lambda_1\lambda_d}{\lambda_1-\lambda_d}}\beta_*^{1/2-7\varepsilon}\right]
	\\ \left(\text{by $\varphi(p,d;\Lambda)\ge\tfrac{p(d-p)\lambda_1\lambda_d}{\lambda_1-\lambda_d}$ and $d\ge2p$}\right)
	\end{multlined}
	\\ &\le
	\begin{multlined}[t]
	\tfrac{3}{2}\left[\tfrac{\beta_*^{1+3\varepsilon}}{(1-\beta_*\gamma)^2}\tfrac{C_p^2}{p}\tfrac{\lambda_1-\lambda_d}{\lambda_1\lambda_d}
	+\tfrac{32\psi^4}{2-\lambda_1 \beta_*}
	+C_{\circ}\kappa ^4\nu^2\eta_p ^2\gamma^{-1}p^{-1/2}\tfrac{\lambda_1-\lambda_d}{\lambda_1\lambda_d}\beta_*^{1/2-7\varepsilon}\right]
	\\ \left(\text{by $d\beta_*^{1-3\varepsilon}\le\delta^2$}\right)
	\end{multlined}
	\\ &=:C_*(d,N_*,\delta)
	.
	\end{align*}
	Since $\beta_*\le1$ and $\beta_*\gamma\le \lambda_1\beta_*\le\sqrt{2}-1$, we have
	\[
	C_*(d,N_*,\delta)
	\le \tfrac{3}{2}\left[\tfrac{C_p^2}{2(3-2\sqrt{2})p}\tfrac{\lambda_1-\lambda_d}{\lambda_1\lambda_d} +\tfrac{32\psi^4}{3-\sqrt{2}}+\tfrac{C_{\circ}\kappa ^4\nu^2\eta_p^2(\lambda_1-\lambda_d)}{p^{1/2}\gamma\lambda_1\lambda_d}\right]
	,
	\]
	and also
	$C_*(d,N_*,\delta)\to 24\psi^4$ as $d\to\infty, N_*\to\infty$,
	as was to be shown.
\end{proof}

\section{Conclusion} \label{sec:conclusion}
We have presented a detailed convergence analysis for the multidimensional subspace online PCA iteration on
sub-Gaussian samples, following the recent work \cite{liWLZ2017near} by \LWLZ\ who considered
only the one-dimensional case, i.e.,
the most significant principal component. Our results bear similar forms to theirs and
when applied to the one-dimensional case  yield estimates of essentially the same quality, as expected.
As we embarked on the analysis presented in this paper, we found that a straightforward extension of
the analysis in \cite{liWLZ2017near} was not possible because of the involvement of
a $\cot$-matrix of dimension higher than $1$ in the multidimensional case but just a scalar in
the one-dimensional case.
Our results yields an explicit convergence rate, and it
is \emph{nearly optimal} because it nearly attains the minimax information lower bound for sub-Gaussian PCA
under a constraint, as well as
\emph{nearly global} because the finite sample error bound holds with high probability
if the initial value is uniformly sampled from the Grassmann manifold.

\subsection*{Acknowledgements}
{This work was supported in part by National Natural Science Foundation of China (Grant No. 11901340), National Science Foundation of USA (Grants Nos. DMS-1719620 and DMS-2009689), Ministry of Science and Technology of Taiwan, Nationcal Center for Theoretical Sciences, and the ST Yau Centre at the National Chiao Tung University.
The authors are indebted to the editor and anonymous referees for their constructive comments and suggestions that  improved the presentation.}
\begin{appendix}

\section{Supplementary Proofs}\label{sec:supplementary-proofs}
\subsection{Proof of Item \ref{itm:lm:diff-T:estimate} of Lemma \ref{lm:diff-T}}\label{ssec:proof-of-item-itm-of-lemma-lm}
	We have
	\begin{equation}\label{eq:itm:lm:diff-T:estimate:pf-1}
	\varc{\Delta T}=
	\varc{T_L(\beta YY^{\T}+R_T)T_R}
	=\beta^2 \varc{T_LYY^{\T}T_R}
	+2\beta R_{\circ,1}+ R_{\circ,2},
	\end{equation}
	where $R_{\circ,1}=\covc{T_LYY^{\T}T_R,T_LR_TT_R}$, and $R_{\circ,2}=\varc{T_LR_TT_R}$.
	By \cref{eq:dT:mainpart},
	\begin{equation}\label{eq:itm:lm:diff-T:estimate:pf-2}
	\varc{T_LYY^{\T}T_R}=\varc{\ul Y\ol Y^{\T}}+R_{\circ,0},
	\end{equation}
	where
	\begin{align*}
	R_{\circ,0}
	&=\varc{T\ol Y\ul Y^{\T}T}+\varc{T\ol Y\ol Y^{\T}}+\varc{\ul Y\ul Y^{\T}T} 	\\
	&\quad -2\covc{\ul Y\ol Y^{\T},T\ol Y\ul Y^{\T}T}-2\covc{\ul Y\ol Y^{\T},T\ol Y\ol Y^{\T}}+2\covc{\ul Y\ol Y^{\T},\ul Y\ul Y^{\T}T}
	\\ &\quad +2\covc{T\ol Y\ul Y^{\T}T,T\ol Y\ol Y^{\T}}-2\covc{T\ol Y\ul Y^{\T}T,\ul Y\ul Y^{\T}T} \\
    &\quad-2\covc{T\ol Y\ol Y^{\T},\ul Y\ul Y^{\T}T}.
	\end{align*}
	Examine \cref{eq:itm:lm:diff-T:estimate:pf-1} and \cref{eq:itm:lm:diff-T:estimate:pf-2} together to get
	$H_{\circ}=\varc{\ul Y\ol Y^{\T}}$ and  $R_{\circ}=\beta^2R_{\circ,0}+2\beta R_{\circ,1}+R_{\circ,2}$.
	We note
	\begin{gather*}
	Y_j=e_j^{\T}Y
	=e_j^{\T}\Lambda^{1/2}\Lambda^{-1/2}Y
	=\lambda_j^{1/2}e_j^{\T}\Lambda^{-1/2}Y, \\
	e_i^{\T}\varc{\ul Y\ol Y^{\T}}e_j
	=\var(e_i^{\T}\ul Y\ol Y^{\T}e_j)=\var(Y_{p+i}Y_{j})
	=\txs\E{Y_{p+i}^2Y_{j}^2}.
	\end{gather*}
	By \cite[(5.11)]{vershynin2012introduction},
	\[
	\txs\E{Y_j^4}
	= \lambda_j^2\txs\E{(e_j^{\T}\Lambda^{-1/2}Y)^4}
	\le 16\lambda_j^2\N{e_j^{\T}\Lambda^{-1/2}Y}_{\psi_2}^4
	\le 16\lambda_j^2\N{\Lambda^{-1/2}Y}_{\psi_2}^4
	= 16\lambda_j^2\psi^4.
	\]
	Therefore
	$
	e_i^{\T}\varc{\ul Y\ol Y^{\T}}e_j
	\le
	[\txs\E{Y_{p+i}^4} \txs\E{Y_j^4}]^{1/2}
	\le 16\lambda_{p+i}\lambda_j\psi^4,
	$
	i.e.,  $H_{\circ}=\varc{\ul Y\ol Y^{\T}}\le 16\psi^4 H$.
	This proves \cref{itm:lm:diff-T:varcdT}.
	To show \cref{itm:lm:diff-T:normR_H}, first we bound the entrywise variance and covariance.
	For any matrices $A_1,A_2$ of the same size, it holds that \cite[p.233]{hojo:1991}
	\begin{equation}\label{eq:schur-ineq-UI}
	\N{A_1\circ A_2}_2\le\N{A_1}_2\N{A_2}_2,
	\end{equation}
	and thus
	\begin{subequations}\label{eq:covc&varc}
		\begin{align}
		\N{\covc{A_1,A_2}}_2
		&=\N{\txs\E{A_1\circ A_2}-\txs\E{A_1}\circ\txs\E{A_2}}_2 \nonumber\\
		&\le \txs\E{\N{A_1}_2\N{A_2}_2}+ \N{\txs\E{A_1}}_2\N{\txs\E{A_2}}_2, \label{eq:covc} \\
		\N{\varc{A_1}}_2
		&\le \txs\E{\N{A_1}_2^2}+ \N{\txs\E{A_1}}_2^2. \label{eq:varc}
		\end{align}
	\end{subequations}
	Apply \cref{eq:covc&varc} to $R_{\circ,1}$ and $R_{\circ,2}$ to get
	\begin{equation}\label{eq:diff-T:pf-10}
	\N{R_{\circ,1}}_2\le2C_T\nu\wtd\eta_p^3 \beta^2(1+\N{T}_2^2)^{5/2},	\quad
	\N{R_{\circ,2}}_2\le2C_T^2\nu(\wtd\eta_p \beta)^4(1+\N{T}_2^2)^3,
	\end{equation}
	upon using
	\[
		\N{T_LYY^{\T}T_R}_2=\N{T_LYY^{\T}VV^{\T}T_R}_2\le \nu^{1/2}\wtd\eta_p(1+\N{T}_2^2), \quad
			\N{T_LR_TT_R}_2\le C_T\nu^{1/2}(\wtd\eta_p \beta)^2(1+\N{T}_2^2)^{3/2}.
	\]
	For $R_{\circ,0}$, 	by \cref{eq:dT:mainpart:E}, we have
	\begin{align*}
	\N{\covc{\ul Y\ol Y^{\T},T\ol Y\ul Y^{\T}T}}_2&\le\txs\E{\N{\ul Y\ol Y^{\T}}_2^2}\N{T}_2^2,\\
	\N{\covc{\ul Y\ol Y^{\T},T\ol Y\ol Y^{\T}}}_2&\le\txs\E{\N{\ul Y\ol Y^{\T}}_2\N{\ol Y\ol Y^{\T}}_2}\N{T}_2,\\
	\N{\covc{\ul Y\ol Y^{\T},\ul Y\ul Y^{\T}T}}_2&\le\txs\E{\N{\ul Y\ol Y^{\T}}_2\N{\ul Y\ul Y^{\T}}_2}\N{T}_2,\\
	\N{\covc{T\ol Y\ul Y^{\T}T,T\ol Y\ol Y^{\T}}}_2&\le\txs\E{\N{\ul Y\ol Y^{\T}}_2\N{\ol Y\ol Y^{\T}}_2}\N{T}_2^3,\\
	\N{\covc{T\ol Y\ul Y^{\T}T,\ul Y\ul Y^{\T}T}}_2&\le\txs\E{\N{\ul Y\ol Y^{\T}}_2\N{\ul Y\ul Y^{\T}}_2}\N{T}_2^3,\\
	\N{\varc{T\ol Y\ul Y^{\T}T}}_2&\le\txs\E{\N{\ul Y\ol Y^{\T}}_2^2}\N{T}_2^4,\\
	\N{\varc{T\ol Y\ol Y^{\T}}}_2&\le\txs\E{\N{\ol Y\ol Y^{\T}}_2^2}\N{T}_2^2+\N{T\ol\Lambda }_2^2,\\
	\N{\varc{\ul Y\ul Y^{\T}T}}_2&\le\txs\E{\N{\ul Y\ul Y^{\T}}_2^2}\N{T}_2^2+\N{\ul\Lambda T}_2^2,\\
	\N{\covc{T\ol Y\ol Y^{\T},\ul Y\ul Y^{\T}T}}_2&\le\txs\E{\N{\ol Y\ol Y^{\T}}_2\N{\ul Y\ul Y^{\T}}_2}\N{T}_2^2+\N{T\ol\Lambda }_2\N{\ul\Lambda T}_2.
	\end{align*}
	Since
	\begin{gather*}
	\N{\ol Y\ol Y^{\T}}_2+\N{\ul Y\ul Y^{\T}}_2=\ol Y^{\T}\ol Y+\ul Y^{\T}\ul Y=Y^{\T}Y\le \nu \wtd\eta_p, \\
	\N{\ul Y\ol Y^{\T}}_2=(\ol Y^{\T}\ol Y)^{1/2}(\ul Y^{\T}\ul Y)^{1/2}
	\le \tfrac{\ol Y^{\T}\ol Y+\ul Y^{\T}\ul Y}{2}\le \tfrac{\nu \wtd\eta_p }{2},
	\end{gather*}
	we have
	\begin{align}
	\N{R_{\circ,0}}_2
	&\le\txs\E{2\N{\ul Y\ol Y^{\T}}_2^2+(\N{\ol Y\ol Y^{\T}}_2+\N{\ul Y\ul Y^{\T}}_2)^2}\N{T}_2^2 + (\N{T\ol\Lambda }_2+\N{\ul\Lambda T}_2)^2\nonumber\\
	&\quad+2\txs\E{\N{\ul Y\ol Y^{\T}}_2(\N{\ol Y\ol Y^{\T}}_2+\N{\ul Y\ul Y^{\T}}_2)}\left( \N{T}_2 + \N{T}_2^3 \right)+\txs\E{\N{\ul Y\ol Y^{\T}}_2^2}\N{T}_2^4\nonumber\\
	& \le (\nu \wtd\eta_p)^2\N{T}_2+\big[\tfrac{3}{2}(\nu \wtd\eta_p)^2+(\lambda_1+\lambda_{p+1})^2\big]\N{T}_2^2+(\nu \wtd\eta_p)^2\N{T}_2^3+\tfrac{1}{4}(\nu \wtd\eta_p)^2\N{T}_2^4\nonumber	\\
	& \le (\nu \wtd\eta_p)^2\N{T}_2\big(1+\tfrac{11}{2}\N{T}_2+\N{T}_2^2+\tfrac{1}{4}\N{T}_2^3\big). \label{eq:diff-T:pf-11}
	\end{align}
	Finally collecting \cref{eq:diff-T:pf-10} and \cref{eq:diff-T:pf-11} yields the desired bound on
	$R_{\circ}=\beta^2R_{\circ,0}+2\beta R_{\circ,1}+R_{\circ,2}$.

\subsection{Estimation in Proof of Lemma \ref{lm:prop4:ppr1}}\label{ssec:estimation-in-proof-of-lemma-lm}
	\begin{enumerate}
		\item
		$ \txs\E{\wtd J_1\circ \wtd J_1}=\opL^{2n}T^{(0)}\circ T^{(0)}$.
		\item
		$ \txs\E{\wtd J_1\circ \wtd J_{21}}= \txs\sum\limits_{s=1}^{N_{1/2-3\varepsilon}-1}\opL^{2n-s} T^{(0)}\circ\txs\E{D^{(s)}\ind{\bbQ_K}}=0$,
		because
		\[
			\txs\E{D^{(s)}\ind{\bbQ_K}}= \txs\E{\txs\E{D^{(s)}\ind{\bbQ_K}\given\fil_{s-1}}}=0.
		\]
		\item
		$\txs\E{\wtd J_1\circ \wtd J_{22}}=\txs\sum\limits_{s=N_{1/2-3\varepsilon}}^{n}\opL^{2n-s} T^{(0)}\circ\txs\E{D^{(s)}\ind{\bbT_{s-1}}\ind{\bbQ_K}}=0$, because
		$\bbT_{s-1}\subset\fil_{s-1}$ and so
		\begin{align*}
		\txs\E{D^{(s)}\ind{\bbT_{s-1}}\ind{\bbQ_K}}
		&=\txs\prob{\bbT_{s-1}}\txs\E{D^{(s)}\ind{\bbQ_K}\given\bbT_{s-1}}
		\\&=\txs\prob{\bbT_{s-1}}\txs\E{\txs\E{D^{(s)}\ind{\bbQ_K}\given\fil_{s-1}}\given\bbT_{s-1}}
		=0.
		\end{align*}
		\item
		$\txs\E{\wtd J_1\circ \wtd J_3}=\txs\sum\limits_{s=1}^{n} \opL^{2n-s}T^{(0)}\circ\txs\E{E_T^{(s-1)}\ind{\bbQ_K}}$.
		Recall \cref{eq:schur-ineq-UI}. By \cref{itm:lm:diff-T:normE_T} of Lemma~\ref{lm:diff-T}, we have
		\begin{align*}
		\N{\txs\E{\wtd J_1\circ \wtd J_3}}_2
		&\le \txs\sum\limits_{s=1}^n\N{\opL}_2^{2n-s}\N{T^{(0)}}_2\N{\txs\E{E_T^{(s-1)}\ind{\bbQ_K}}}_2
		\\ &\le \txs\sum\limits_{s=1}^n(1-\beta\gamma)^{2n-s}(\tfrac{\kappa ^2}{4}-1)^{1/2}C_T\nu^{1/2}(\wtd\eta_p\beta)^2\kappa ^3
		\\ &\le (1-\beta\gamma)^n\tfrac{(\kappa ^2-1)^{1/2}C_T\nu^{1/2}\wtd\eta_p ^2\beta^2\kappa ^3}{2\beta\gamma}
		\\ &\le \tfrac{1}{2}\beta^{1-6\varepsilon}C_T\nu^{1/2}\wtd\eta_p ^2\gamma^{-1}\beta \kappa ^4
                       \qquad\text{(by $n\ge N_{1-6\varepsilon}$).}
		\end{align*}
		\item
		$\txs\E{\wtd J_{21}\circ \wtd J_{22}}=\txs\sum\limits_{s=1}^{N_{1/2-3\varepsilon}-1}\txs\sum\limits_{s'=N_{1/2-3\varepsilon}}^{n}\opL^{2n-s-s'} \txs\E{D^{(s)}\ind{\bbQ_K}\circ D^{(s')}\ind{\bbT_{s'-1}}\ind{\bbQ_K}} =0$,
		because $s<s'$ and
		\begin{align*}
		 \txs\E{D^{(s)}\ind{\bbQ_K}\circ D^{(s')}\ind{\bbT_{s'-1}}\ind{\bbQ_K}}
		 &= \txs\E{D^{(s)}\circ D^{(s')}\ind{\bbT_{s'-1}}\ind{\bbQ_K}}
		\\ &=  \txs\prob{\bbT_{s'-1}}\txs\E{D^{(s)}\circ D^{(s')}\ind{\bbQ_K}\given\bbT_{s'-1}}
		\\ &=\txs\prob{\bbT_{s'-1}}\txs\E{\txs\E{D^{(s)}\circ D^{(s')}\ind{\bbQ_K}\given\fil_{s'-1}}\given\bbT_{s'-1}}
		\\ &= \txs\prob{\bbT_{s'-1}}\txs\E{\txs\E{D^{(s')}\ind{\bbQ_K}\given\fil_{s'-1}}\circ D^{(s)}\given\bbT_{s'-1}} \\
		&=0.
		\end{align*}
		\item 
$
\begin{aligned}[t]
		\txs\E{\wtd J_{21}\circ \wtd J_{21}}&= \txs\sum\limits_{s=1}^{N_{1/2-3\varepsilon}-1}\txs\sum\limits_{s'=1}^{N_{1/2-3\varepsilon}-1}\opL^{2n-s-s'} \txs\E{D^{(s)}\ind{\bbQ_K}\circ D^{(s')}\ind{\bbQ_K}} \\
    &= \txs\sum\limits_{s=1}^{N_{1/2-3\varepsilon}-1}\opL^{2(n-s)} \txs\E{D^{(s)}\circ D^{(s)}\ind{\bbQ_K}},
\end{aligned}
$\\
		because for $s\ne s'$,
		\begin{align*}
		\txs\E{D^{(s)}\ind{\bbQ_K}\circ D^{(s')}\ind{\bbQ_K}}
		  &=\txs\E{D^{(s)}\circ D^{(s')}\ind{\bbQ_K}}
		\\ &=\txs\E{\txs\E{D^{(\max\txs\set{s,s'})}\ind{\bbQ_K}\given\fil_{\max\txs\set{s,s'}-1}}\circ D^{(\min\txs\set{s,s'})}} \\
		&=0.
		\end{align*}
		Use \cref{itm:lm:diff-T:varcdT,itm:lm:diff-T:normR_H} of Lemma~\ref{lm:diff-T} to get
		\begin{align*}
		\txs\E{D^{(s)}\circ D^{(s)}\ind{\bbQ_K}}
		&=\txs\E{\txs\E{D^{(s)}\circ D^{(s)}\ind{\bbQ_K}\given\fil_{s-1}}} \\
		&=\txs\E{\varc{[T^{(n+1)}-T^{(n)}]\ind{\bbQ_K}\given \fil_{s-1}}} \\
		&=\txs\E{\beta^2H_{\circ}+R_{\circ}}
		=\beta^2H_{\circ}+\txs\E{R_{\circ}}.
		\end{align*}
		Therefore
		$\txs\E{\wtd J_{21}\circ \wtd J_{21}}=\beta^2\txs\sum\limits_{s=1}^{N_{1/2-3\varepsilon}-1}\opL^{2(n-s)}H_{\circ} + \txs\sum\limits_{s=1}^{N_{1/2-3\varepsilon}-1}\opL^{2(n-s)}\txs\E{R_{\circ}}$.
		We have for $\kappa>2\sqrt{2}$
		\begin{align*}
		\N{R_{\circ}}_2
		&\le (\nu \wtd\eta_p \beta)^2\tau_{s-1}(1+\tfrac{11}{2}\tau_{s-1}+\tau_{s-1}^2+\tfrac{1}{4}\tau_{s-1}^3) 
        +4C_T\kappa ^5\nu(\wtd\eta_p\beta)^3+2C_T^2\kappa ^6\nu(\wtd\eta_p\beta)^4
		\\ &\le (\nu \wtd\eta_p \beta)^2\tau_{s-1}(\kappa ^2+\tfrac{21}{4}\kappa +\tfrac{1}{4}\kappa ^3)+4C_T\kappa ^5\nu(\wtd\eta_p\beta)^3+2C_T^2\kappa ^6\nu(\wtd\eta_p\beta)^4
		\\ &\le \tfrac{29+8\sqrt{2}}{32}\kappa ^3\nu^2(\wtd\eta_p \beta)^2\tau_{s-1}+4C_T\kappa ^5\nu(\wtd\eta_p\beta)^3+2C_T^2\kappa ^6\nu(\wtd\eta_p\beta)^4,
		\end{align*}
		where $\tau_{s-1}=\N{T^{(s-1)}}_2\le(\kappa ^2-1)^{1/2}$.
		Write $ E_{21}:= \txs\sum\limits_{s=1}^{N_{1/2-3\varepsilon}-1}\opL^{2(n-s)}\txs\E{R_{\circ}}$.
		Since $2N_{1/2-3\varepsilon}-1\le N_{1-6\varepsilon}\le 2N_{1/2-3\varepsilon}$ by definition, we get
		\begin{align*}
		\N{E_{21}}_2
		&\le \txs\sum\limits_{s=1}^{N_{1/2-3\varepsilon}-1}\N{\opL}_2^{2(n-s)}\txs\E{\N{R_{\circ}}_2}
		\\ &\le \tfrac{(1-\beta\gamma)^{2(n+1-N_{1/2-3\varepsilon})}}{\beta\gamma[2-\beta\gamma]}\txs\E{\N{R_{\circ}}_2}
		\\ &\le \tfrac{1-\beta\gamma}{2-\beta\gamma}\tfrac{(1-\beta\gamma)^n}{\beta\gamma}\txs\E{\N{R_{\circ}}_2}
		\\ &\le \tfrac{1}{2}\beta^{1-6\varepsilon}\gamma^{-1}\beta\kappa^4\nu\wtd\eta_p^2\left(\tfrac{29+8\sqrt{2}}{32}\nu+4C_T\kappa(\wtd\eta_p\beta)+2C_T^2\kappa ^2(\wtd\eta_p\beta)^2\right)
		\\ &\le \left(\tfrac{29+8\sqrt{2}}{64}+2C_T\kappa(\wtd\eta_p\beta)+C_T^2\kappa ^2(\wtd\eta_p\beta)^2\right)\gamma^{-1}\kappa^4\nu^2\wtd\eta_p^2\beta^{2-6\varepsilon}
		.
		\end{align*}
		\item 
			$
			\begin{aligned}[t]
				\txs\E{\wtd J_{22}\circ \wtd J_{22}}&=\txs\sum\limits_{s=N_{1/2-3\varepsilon}}^{n}\opL^{2(n-s)} \txs\E{D^{(s)}\ind{\bbQ_K}\ind{\bbT_{s-1}}\circ D^{(s)}\ind{\bbQ_K}\ind{\bbT_{s-1}}}
				\\ &=\beta^2\txs\sum\limits_{s=N_{1/2-3\varepsilon}}^{n}\opL^{2(n-s)}H_{\circ} + \txs\sum\limits_{s=N_{1/2-3\varepsilon}}^{n}\opL^{2(n-s)}\txs\E{R_{\circ}\ind{\bbT_{s-1}}},
			\end{aligned}
			 $\\
		because for $s\ne s'$,
		\begin{align*}
&\txs\E{D^{(s)}\ind{\bbQ_K}\ind{\bbT_{s-1}}\circ D^{(s')}\ind{\bbQ_K}\ind{\bbT_{s'-1}}}
=\txs\E{D^{(s)}\circ D^{(s')}\ind{\bbQ_K}\ind{\bbT_{s-1}}\ind{\bbT_{s'-1}}} \\
		&=\txs\E{D^{(s)}\circ D^{(s')}\ind{\bbQ_K}\given\bbT_{s-1}\cap\bbT_{s'-1}}\txs\prob{\bbT_{s-1}\cap\bbT_{s'-1}} \\
		&=\txs\E{\txs\E{D^{(\max\txs\set{s,s'})}\ind{\bbQ_K}\given\fil_{\max\txs\set{s,s'}-1}}
			\circ D^{(\min\txs\set{s,s'})}\given\bbT_{s-1}\cap\bbT_{s'-1}} \prob{\bbT_{s-1}\cap\bbT_{s'-1}} \\
		&=0,
		\end{align*}
		and
		\begin{align*}
		&\txs\E{D^{(s)}\ind{\bbQ_K}\ind{\bbT_{s-1}}\circ D^{(s)}\ind{\bbQ_K}\ind{\bbT_{s-1}}}
		 =\txs\E{D^{(s)}\circ D^{(s')}\ind{\bbQ_K}\ind{\bbT_{s-1}}}
		\\ &\quad=\txs\prob{\bbT_{s-1}}\txs\E{\txs\E{D^{(s)}\circ D^{(s)}\ind{\bbQ_K}\given\fil_{s-1}}\given\bbT_{s-1}}
		\\ &\quad\le\beta^2H_{\circ}+\txs\E{R_{\circ}\ind{\bbT_{s-1}}}.
		\end{align*}
		We have
		\begin{align*}
		\N{R_{\circ}\ind{\bbT_{s-1}}}_2
		&\le \tfrac{29+8\sqrt{2}}{32}\kappa ^3\nu^2(\wtd\eta_p \beta)^2\tau_{s-1}+4C_T\kappa ^5\nu(\wtd\eta_p\beta)^3+2C_T^2\kappa ^6\nu(\wtd\eta_p\beta)^4
		\\ &\le \tfrac{29+8\sqrt{2}}{32}\kappa ^3\nu^2(\wtd\eta_p \beta)^2(\kappa ^2-1)^{1/2}\beta^{1/2-3\varepsilon}+4C_T\kappa ^5\nu(\wtd\eta_p\beta)^3+2C_T^2\kappa ^6\nu(\wtd\eta_p\beta)^4
		\\ &\le \tfrac{29+8\sqrt{2}}{32}\kappa ^4\nu^2(\wtd\eta_p \beta)^2\beta^{1/2-3\varepsilon}+4C_T\kappa ^5\nu(\wtd\eta_p\beta)^3+2C_T^2\kappa ^6\nu(\wtd\eta_p\beta)^4
		.
		\end{align*}
		Write
		$ E_{22}:= \txs\sum\nolimits_{s=N_{1/2-3\varepsilon}}^{n}\opL^{2(n-s)}\txs\E{R_{\circ}\ind{\bbT_{s-1}}} $ for which we have
		\begin{align*}
		\N{E_{22}}_2
		&\le \txs\sum\limits_{s=N_{1/2-3\varepsilon}}^{n}\N{\opL}_2^{2(n-s)}\txs\E{\N{R_{\circ}\ind{\bbT_{s-1}}}_2}
		\\ &\le \tfrac{1}{\beta\gamma[2-\beta\gamma]}\txs\E{\N{R_{\circ}\ind{\bbT_{s-1}}}_2}
		\\ &\le  \tfrac{1}{3-\sqrt{2}}\gamma^{-1}\kappa^4\nu\wtd\eta_p^2\beta\left( \tfrac{29+8\sqrt{2}}{32}\nu\beta^{1/2-3\varepsilon}+4C_T\kappa(\wtd\eta_p\beta)+2C_T^2\kappa ^2(\wtd\eta_p\beta)^2  \right)
		\\ &\le  \tfrac{1}{3-\sqrt{2}}\left( \tfrac{29+8\sqrt{2}}{32}+4C_T\kappa\wtd\eta_p\beta^{1/2+3\varepsilon}+2C_T^2\kappa ^2\wtd\eta_p^2\beta^{3/2+3\varepsilon}  \right)\gamma^{-1}\kappa^4\nu^2\wtd\eta_p^2\beta^{3/2-3\varepsilon}
		.
		\end{align*}
		\item
		$\txs\E{\wtd J_3\circ \wtd J_3}=\txs\sum\limits_{s=1}^{n} \opL^{2(n-s)}\txs\E{E_T^{(s-1)}\ind{\bbQ_K}\circ E_T^{(s-1)}\ind{\bbQ_K}}$.
		Also, by \cref{eq:schur-ineq-UI},
		\begin{align*}
		\N{\txs\E{\wtd J_3\circ \wtd J_3}}_2
		&\le \txs\sum\limits_{s=1}^{n} \N{\opL}_2^{2(n-s)}\txs\E{\N{E_T^{(s-1)}\ind{\bbQ_K}}_2^2}
		\\ &\le \txs\sum\limits_{s=1}^{n} (1-\beta\gamma)^{2(n-s)} [C_T \nu^{1/2}(\wtd\eta_p \beta)^2\kappa ^3]^2
		\\ &\le \frac{ C_T^2\nu(\wtd\eta_p \beta)^4\kappa ^6}{\beta\gamma[2-\beta\gamma]}
		\le  \frac{1}{3-\sqrt{2}}C_T^2\nu\wtd\eta_p^4\gamma^{-1}\kappa ^6\beta^3
		.
		\end{align*}
	\end{enumerate}

\subsection{Proof of Lemma \ref{lm:2F1}}\label{ssec:proof-of-lm}
	The proof is the same as that for the case $p=1$ by Kummer's solutions of the hypergeometric differential equation (see, e.g., \cite[Section~3.8]{luke1969special}).
	Let the eigenvalues of $T$ be $\mu_1,\dots,\mu_m$.
	Since $\hyperg(a,b;c;T)$ is defined on the spectrum of $T$, it is a function of $\mu_1,\dots,\mu_m$.
	When treated as such, by \cite[Theorem~7.5.5]{muirhead1982aspects}, $\hyperg(a,b;c;T)$ is the unique solution of partial differential equations,
	\begin{multline}\label{eq:hyperg-PDE}
	\mu_i(1-\mu_i)\tfrac{\partial^2 F}{\partial \mu_i^2}
	+\left( c-\tfrac{m-1}{2} -(a+b+1-\tfrac{m-1}{2})\mu_i+\tfrac{1}{2} \txs\sum\limits_{1\le j\le m}^{j\ne i} \tfrac{\mu_i(1-\mu_i)}{\mu_i-\mu_j} \right)\tfrac{\partial F}{\partial \mu_i}
	\\-\tfrac{1}{2} \txs\sum\limits_{1\le j\le m}^{j\ne i} \tfrac{\mu_j(1-\mu_j)}{\mu_i-\mu_j}\tfrac{\partial F}{\partial \mu_j}
	-abF=0,
	\end{multline}
	subject to the conditions that $F$ is a symmetric function of $\mu_1,\dots,\mu_m$,   analytic at $(\mu_1,\dots,\mu_m)=(0,\ldots,0)$,
	and $F(0,\ldots,0)=1$.
	
	We claim that $\wtd F(\mu_1,\dots,\mu_m):=\hyperg(a,b;a+b-c+\tfrac{m+1}{2};I-T)$ satisfies \cref{eq:hyperg-PDE}.
	In fact, letting $\wtd \mu_i=1-\mu_i$ for $1\le i\le m$ which are the eigenvalues of $I-T$, we have
	\begin{align*}
	&
	\mu_i(1-\mu_i)\tfrac{\partial^2 \wtd F}{\partial \mu_i^2}
	+\left( c-\tfrac{m-1}{2} -(a+b+1-\tfrac{m-1}{2})\mu_i+\tfrac{1}{2} \txs\sum\limits_{1\le j\le m}^{j\ne i} \tfrac{\mu_i(1-\mu_i)}{\mu_i-\mu_j} \right)\tfrac{\partial \wtd F}{\partial \mu_i}
	-\tfrac{1}{2} \txs\sum\limits_{1\le j\le m}^{j\ne i} \tfrac{\mu_j(1-\mu_j)}{\mu_i-\mu_j}\tfrac{\partial \wtd F}{\partial \mu_j}
	-ab\wtd F
	\\ &=
	\begin{multlined}[t]
	(1-\wtd\mu_i)\wtd\mu_i\tfrac{\partial^2 \wtd F}{\partial \wtd\mu_i^2}
	+\tfrac{1}{2} \txs\sum\limits_{1\le j\le m}^{j\ne i} \tfrac{(1-\wtd\mu_j)\wtd\mu_j}{(1-\wtd\mu_i)-(1-\wtd\mu_j)}\tfrac{\partial \wtd F}{\partial \wtd\mu_j}
	-ab\wtd F
	\\-\left( c-\tfrac{m-1}{2} -(a+b+1-\tfrac{m-1}{2})(1-\wtd\mu_i)+\tfrac{1}{2} \txs\sum\limits_{1\le j\le m}^{j\ne i} \tfrac{(1-\wtd\mu_i)\wtd\mu_i}{(1-\wtd\mu_i)-(1-\wtd\mu_j)} \right)\tfrac{\partial \wtd F}{\partial \wtd\mu_i}
	\end{multlined}
	\\ &=
	\begin{multlined}[t]
	(1-\wtd\mu_i)\wtd\mu_i\tfrac{\partial^2 \wtd F}{\partial \wtd\mu_i^2}
	-\tfrac{1}{2} \txs\sum\limits_{1\le j\le m}^{j\ne i} \tfrac{(1-\wtd\mu_j)\wtd\mu_j}{\wtd\mu_i-\wtd\mu_j}\tfrac{\partial \wtd F}{\partial \wtd\mu_j}
	-ab\wtd F
	\\+\left( -c+\tfrac{m+1}{2} +a+b-\tfrac{m-1}{2}-(a+b+1-\tfrac{m-1}{2})\wtd\mu_i+\tfrac{1}{2} \txs\sum\limits_{1\le j\le m}^{j\ne i} \tfrac{(1-\wtd\mu_i)\wtd\mu_i}{\wtd\mu_i-\wtd\mu_j} \right)\tfrac{\partial \wtd F}{\partial \wtd\mu_i}
	\end{multlined}
	\\ &=0,
	\end{align*}
	where the last equality holds because $\wtd F(\mu_1,\dots,\mu_m)=\hyperg(a,b;a+b-c+\tfrac{m+1}{2};I-T)$ satisfies
	a version of \cref{eq:hyperg-PDE} after substitutions: $\mu_i\to\wtd\mu_i$ for all $i$ and
	$c\to a+b-c+\tfrac{m+1}{2}$.
	
	$\what F(\mu_1,\dots,\mu_m):=\det(T)^{\tfrac{m+1}{2}-c}\hyperg(a-c+\tfrac{m+1}{2},b-c+\tfrac{m+1}{2};m+1-c;T)$
	satisfies \cref{eq:hyperg-PDE}, too. Set
	$t=\tfrac{m+1}{2}-c$ and write
	$G(\mu_1,\dots,\mu_m)=\hyperg(a+t,b+t;c+2t;T)$. We have
	\begin{align*}
	\frac{\partial \what F}{\partial \mu_i}
	&=\frac{t}{\mu_i}\det(T)^t G+\det(T)^t\frac{\partial G}{\partial \mu_i},
	\\
	\frac{\partial^2 \what F}{\partial \mu_i^2}
	&=\frac{t(t-1)}{\mu_i^2}\det(T)^t G+2\frac{t}{\mu_i}\det(T)^t\frac{\partial G}{\partial \mu_i}
	+\det(T)^t\frac{\partial^2 G}{\partial \mu_i^2},
	\end{align*}
	and thus
	\begin{align*}
	&\mu_i(1-\mu_i)\tfrac{\partial^2 \what F}{\partial \mu_i^2}
	    +\left( c-\tfrac{m-1}{2} -(a+b+1-\tfrac{m-1}{2})\mu_i+\tfrac{1}{2} \txs\sum\limits_{1\le j\le m}^{j\ne i}
         \tfrac{\mu_i(1-\mu_i)}{\mu_i-\mu_j} \right)\tfrac{\partial \what F}{\partial \mu_i} 
	-\tfrac{1}{2} \txs\sum\limits_{1\le j\le m}^{j\ne i} \tfrac{\mu_j(1-\mu_j)}{\mu_i-\mu_j}
                    \tfrac{\partial \what F}{\partial \mu_j}-ab\what F   \\
	&=\mu_i(1-\mu_i)\left(\tfrac{t(t-1)}{\mu_i^2}\det(T)^t G+2\tfrac{t}{\mu_i}\det(T)^t\tfrac{\partial G}{\partial \mu_i}
	    +\det(T)^t\tfrac{\partial^2 G}{\partial \mu_i^2}\right)  \\
	&\quad+\left( c-\tfrac{m-1}{2} -(a+b+1-\tfrac{m-1}{2})\mu_i+\tfrac{1}{2} \txs\sum\limits_{1\le j\le m}^{j\ne i} \tfrac{\mu_i(1-\mu_i)}{\mu_i-\mu_j} \right) 
    \left( \tfrac{t}{\mu_i}\det(T)^t G+\det(T)^t\tfrac{\partial G}{\partial \mu_i} \right)
	\\ &\quad -\tfrac{1}{2} \txs\sum\limits_{1\le j\le m}^{j\ne i} \tfrac{\mu_j(1-\mu_j)}{\mu_i-\mu_j}\left( \tfrac{t}{\mu_i}\det(T)^t G+\det(T)^t\tfrac{\partial G}{\partial \mu_i} \right)
	-ab\det(T)^t G \\
 &=\det(T)^t\bigg\{
	\mu_i(1-\mu_i)\tfrac{\partial^2 G}{\partial \mu_i^2}-\tfrac{1}{2} \txs\sum\limits_{1\le j\le m}^{j\ne i} \tfrac{\mu_j(1-\mu_j)}{\mu_i-\mu_j}\tfrac{\partial G}{\partial \mu_j}
	\\ &\quad +\left(2\mu_i(1-\mu_i)\tfrac{t}{\mu_i}+ c-\tfrac{m-1}{2} -(a+b+1-\tfrac{m-1}{2})\mu_i+\tfrac{1}{2} \txs\sum\limits_{1\le j\le m}^{j\ne i} \tfrac{\mu_i(1-\mu_i)}{\mu_i-\mu_j} \right)\tfrac{\partial G}{\partial \mu_i}
	\\ &\quad+\bigg[\mu_i(1-\mu_i)\tfrac{t(t-1)}{\mu_i^2}+\left( c-\tfrac{m-1}{2} -(a+b+1-\tfrac{m-1}{2})\mu_i+\tfrac{1}{2} \txs\sum\limits_{1\le j\le m}^{j\ne i} \tfrac{\mu_i(1-\mu_i)}{\mu_i-\mu_j} \right)\tfrac{t}{\mu_i}
	-\tfrac{1}{2} \txs\sum\limits_{1\le j\le m}^{j\ne i} \tfrac{\mu_j(1-\mu_j)}{\mu_i-\mu_j}\tfrac{t}{\mu_j}-ab\bigg]G
	\bigg\}	\\
&=\det(T)^t\bigg\{
	\mu_i(1-\mu_i)\tfrac{\partial^2 G}{\partial \mu_i^2}-\tfrac{1}{2} \txs\sum\limits_{1\le j\le m}^{j\ne i} \tfrac{\mu_j(1-\mu_j)}{\mu_i-\mu_j}\tfrac{\partial G}{\partial \mu_j}
	\\ &\quad+\left(2(1-\mu_i)t+ c-\tfrac{m-1}{2} -(a+b+1-\tfrac{m-1}{2})\mu_i+\tfrac{1}{2} \txs\sum\limits_{1\le j\le m}^{j\ne i} \tfrac{\mu_i(1-\mu_i)}{\mu_i-\mu_j} \right)\tfrac{\partial G}{\partial \mu_i}
	\\ &\quad+\bigg[\tfrac{t(t-1)}{\mu_i}-t(t-1)+(c-\tfrac{m-1}{2})\tfrac{t}{\mu_i} -(a+b+1-\tfrac{m-1}{2})t 
       +\tfrac{1}{2} \txs\sum\limits_{1\le j\le m}^{j\ne i}(-1)t-ab\bigg]G
	\bigg\}	\\
&=
	\det(T)^t\bigg\{
	\mu_i(1-\mu_i)\tfrac{\partial^2 G}{\partial \mu_i^2}-\tfrac{1}{2} \txs\sum\limits_{1\le j\le m}^{j\ne i} \tfrac{\mu_j(1-\mu_j)}{\mu_i-\mu_j}\tfrac{\partial G}{\partial \mu_j}
	\\&\quad+\left(2t+ c-\tfrac{m-1}{2} -(2t+a+b+1-\tfrac{m-1}{2})\mu_i+\tfrac{1}{2} \txs\sum\limits_{1\le j\le m}^{j\ne i} \tfrac{\mu_i(1-\mu_i)}{\mu_i-\mu_j} \right)\tfrac{\partial G}{\partial \mu_i}
	-\left[t^2+(a+b)t+ab\right]G
	\bigg\}
	\\ &=0,
	\end{align*}
	where the last equality holds because $G(\mu_1,\dots,\mu_m)=\hyperg(a+t,b+t;c+2t;T)$ satisfies
	a version of \cref{eq:hyperg-PDE} after substitutions: $a\to a+t$, $b\to b+t$, and $c\to c+2t$.
	
	Similarly $\what{\wtd F}(\mu_1,\dots,\mu_m):= \det(I-T)^{c-a-b}\hyperg(c-b,c-a;c-a-b+\tfrac{m+1}{2};I-T)$ satisfies \cref{eq:hyperg-PDE}.
	Thus, any linear combination of $\wtd F$ and $\what{\wtd F}$, such as the right-hand side of \cref{eq:hyperg-reflect}, also satisfies \cref{eq:hyperg-PDE}.
	It can be verified that the combination is symmetric with respect to $\mu_1,\dots,\mu_m$, and analytic at $T=0$.
	Therefore, by the uniqueness and $F(0)=1$, similarly to the discussion in \cite[Section~3.9]{luke1969special}, we have \cref{eq:hyperg-reflect}.

\subsection{Complementary Calculation in Proof of Lemma \ref{lm:lm1:ppr1}}\label{ssec:complementary-calculation-in-proof-of-lm}
	Here in defining $g_d$, although $\Gamma_p(-\tfrac{p}{2})$ and $\Gamma_p(\tfrac{1}{2})$ may be  $\infty$,
	by analytic continuation, \linebreak $\Gamma_p(-\tfrac{p}{2})/\Gamma_p(\tfrac{1}{2})$ is well-defined because
	\begin{align*}
	\tfrac{\Gamma_p(-\tfrac{p}{2}+\epsilon)}{\Gamma_p(\tfrac{1}{2}+\epsilon)}
	&=\txs\prod_{i=1}^p\tfrac{\Gamma(-\tfrac{p}{2}-\tfrac{i-1}{2}+\epsilon)}{\Gamma(\tfrac{1}{2}-\tfrac{i-1}{2}+\epsilon)}
	\\ &=
	\begin{dcases*}
	\txs\prod_{i=1}^{p} \txs\prod_{j=1}^{(p-1)/2} \tfrac{1}{\tfrac{-i}{2}-j+1+\epsilon}
	& for odd $p$, \\
	\tfrac{\Gamma(\tfrac{1-2p}{2}+\epsilon)}{\Gamma(\tfrac{1}{2}+\epsilon)}\txs\prod_{i=1}^{p-1} \txs\prod_{j=1}^{p/2} \tfrac{1}{\tfrac{-i-1}{2}-j+1+\epsilon}
	& for even $p$, \\
	\end{dcases*}
	\\ & \xrightarrow{\epsilon\to0}
	\left.
	\begin{dcases*}
	\txs\prod_{i=1}^{p} \txs\prod_{j=1}^{(p-1)/2} \tfrac{-2}{i+2j-2},
	\\
	\txs\prod_{k=1}^{p}\tfrac{1}{\tfrac{-1}{2}-k+1}
	\txs\prod_{i=1}^{p-1} \txs\prod_{j=1}^{p/2} \tfrac{-2}{i+2j-1}
	\\
	\end{dcases*}
\right\}
	= \txs\prod_{i=1}^{2\floor{p/2}+1} \txs\prod_{j=1}^{\floor{p/2}} \tfrac{-2}{i+2j-2}
	.
	\end{align*}
	Also,
	\[
	\tfrac{\Gamma_p(\tfrac{p}{2})}{\Gamma_p(\tfrac{p+1}{2})}
	=\txs\prod_{i=1}^p\tfrac{\Gamma(\tfrac{p}{2}-\tfrac{i-1}{2})}{\Gamma(\tfrac{p+1}{2}-\tfrac{i-1}{2})}
	= \tfrac{\Gamma(\tfrac{1}{2})}{\Gamma(\tfrac{p+1}{2})}
	,\quad
	\tfrac{\Gamma_p(\tfrac{d}{2})}{\Gamma_p(\tfrac{d+1}{2})}
	=\txs\prod_{i=1}^p\tfrac{\Gamma(\tfrac{d}{2}-\tfrac{i-1}{2})}{\Gamma(\tfrac{d+1}{2}-\tfrac{i-1}{2})}
	= \tfrac{\Gamma(\tfrac{d-p+1}{2})}{\Gamma(\tfrac{d+1}{2})}
	,
	\]
	which implies
	$ f_d=\tfrac{\Gamma(\tfrac{1}{2})\Gamma(\tfrac{d+1}{2})}{\Gamma(\tfrac{p+1}{2})\Gamma(\tfrac{d-p+1}{2})}$.
	We have
	\[
	f_d^{-1}g_d
	= \tfrac{\Gamma_p(\tfrac{p+1}{2})\Gamma_p(\tfrac{d}{2})\Gamma_p(-\tfrac{p}{2})}{\Gamma_p(\tfrac{p}{2})\Gamma_p(\tfrac{d-p}{2})\Gamma_p(\tfrac{1}{2})}
	= \tfrac{\Gamma(\tfrac{p+1}{2})\Gamma_p(\tfrac{d}{2})\Gamma_p(-\tfrac{p}{2})}{\Gamma(\tfrac{1}{2})\Gamma_p(\tfrac{d-p}{2})\Gamma_p(\tfrac{1}{2})}
	.
	\]
	Note that
	\begin{align*}
	\tfrac{\Gamma_p(\tfrac{d}{2})}{\Gamma_p(\tfrac{d-p}{2})}
	=\txs\prod_{i=1}^p\tfrac{\Gamma(\tfrac{d}{2}-\tfrac{i-1}{2})}{\Gamma(\tfrac{d-p}{2}-\tfrac{i-1}{2})}
	=
	\begin{dcases*}
	\tfrac{\Gamma(\tfrac{d}{2})}{\Gamma(\tfrac{d-p}{2})}\txs\prod_{i=1}^{p} \txs\prod_{j=1}^{(p-1)/2} \left( \tfrac{d-i}{2}-j \right)
	& for odd $p$, \\
	\txs\prod_{i=1}^{p} \txs\prod_{j=1}^{p/2} \left( \tfrac{d-i}{2}-j \right)
	& for even $p$, \\
	\end{dcases*}
	\end{align*}
	and by $\lim_{n\to\infty}\tfrac{\Gamma(n+\alpha)}{\Gamma(n)n^\alpha}=1$ for any $\alpha$ (see, e.g., \cite[(16) of section~2.1]{luke1969special}),
	\[
	\tfrac{\Gamma(\tfrac{d}{2})}{\Gamma(\tfrac{d-p}{2})}
	=
	\begin{dcases*}
	\tfrac{\Gamma(\tfrac{d-1}{2})(\tfrac{d-1}{2})^{1/2}[1+\oo(1)]}{\Gamma(\tfrac{d-1}{2})(\tfrac{d-1}{2})^{(1-p)/2}[1+\oo(1)]},
	& for odd $d$, \\
	\tfrac{\Gamma(\tfrac{d}{2})}{\Gamma(\tfrac{d}{2})(\tfrac{d}{2})^{-p/2}[1+\oo(1)]},
	& for even $d$ \\
	\end{dcases*}
	=
	\begin{dcases*}
	\Big(\tfrac{d-1}{2}\Big)^{p/2}[1+\oo(1)]
	,\\
	\Big(\tfrac{d}{2}\Big)^{p/2}[1+\oo(1)]
	\end{dcases*}
	\]
	which implies
	\[
	\tfrac{\Gamma_p(\tfrac{d}{2})}{\Gamma_p(\tfrac{d-p}{2})}
	=\left( \tfrac{d}{2} \right)^{p^2/2}[1+\oo(1)]
	\,\,\text{as $d\to\infty$}.
	\]

\end{appendix}


\end{document}